\documentclass[reqno]{amsart}
\usepackage[top=1in,bottom=1in,left=1in,right=1in]{geometry}

\usepackage{amssymb}
\usepackage{multicol}
\usepackage{graphicx}
\usepackage[T1]{fontenc}
\usepackage[latin1]{inputenc}
\usepackage[all]{xy}

\usepackage{times}
\usepackage{graphicx}
\usepackage{extarrows}
\usepackage{flafter}
\usepackage{placeins}
\usepackage{enumitem}
\usepackage[OT2,T1]{fontenc}
\DeclareSymbolFont{cyrletters}{OT2}{wncyr}{m}{n}
\DeclareMathSymbol{\Sha}{\mathalpha}{cyrletters}{"58}

\usepackage{centernot}
\usepackage{tikz-cd}
\usepackage{mathrsfs}
\usepackage[color]{showkeys}
\usepackage{url}

\usepackage{graphicx} 
\usepackage{float} 
\usepackage{subfigure} 

\definecolor{refkey}{rgb}{1,1,1}
\definecolor{labelkey}{rgb}{1,1,1}
\definecolor{cite}{rgb}{0.9451,0.2706,0.4941}
\definecolor{ruri}{rgb}{0.0078,0.4022,0.8010}

\usepackage[%
bookmarks=true,bookmarksnumbered=true,%
colorlinks=true,linkcolor=blue,citecolor=blue%
 ]{hyperref}

 \usepackage{mathrsfs}
\usepackage{extarrows}
\usepackage{texnames}
\usepackage[nameinlink,capitalize]{cleveref}

\def\Aut{{\rm Aut}}

\def\Tr{{\rm Tr}}

\def\Tr{{\rm Tr}}
\def\Hom{{\rm Hom}}
\def\End{{\rm End}}

\def\Aut{{\rm Aut}}

\theoremstyle{plain}

\newtheorem{theorem}{Theorem}[section]

\newtheorem{proposition/example}[theorem]{Proposition/Example}
\newtheorem{proposition}[theorem]{Proposition}
\newtheorem{corollary}[theorem]{Corollary}
\newtheorem{lemma}[theorem]{Lemma}

\newtheorem{conjecture}[theorem]{Conjecture}

\theoremstyle{definition}
\newtheorem{definition}[theorem]{Definition}
\newtheorem{remark}[theorem]{Remark}

\newtheorem{example}[theorem]{Example}

\newtheorem{conjecture/question}[theorem]{Conjecture/Question}

\newtheorem{remark/definition}[theorem]{Remark/Definition}
\newtheorem{definition/notation}[theorem]{Definition/Notation}

\numberwithin{equation}{section}

 \theoremstyle{remark}

\numberwithin{equation}{section}

\usepackage{extarrows}\usepackage{bm}

\begin{document}
\title{\textbf{Moduli Spaces of  Parabolic  Bundles over $\mathbb{P}^1$ with Five Marked Points }}

\author{Zhi Hu}

\address{\textsc{School of Mathematics and Statistics, Nanjing University of Science and Technology, Nanjing 210094, China}}

\email{halfask@mail.ustc.edu.cn }

\author{Pengfei Huang}

\address{\textsc{School of Mathematics, Nanjing University, Nanjing 210093, China}}

\email{pfhwangmath@gmail.com}

\author{Runhong Zong}

\address{\textsc{School of Mathematics, Nanjing University, Nanjing 210093, China}}

\email{rzong@nju.edu.cn}


\maketitle

\begin{abstract}
This paper considers the moduli spaces/stacks of parabolic bundles (parabolic logarithmic flat bundles and parabolic logarithmic Higgs bundles with given spectrum) of rank 2 and degree 1 over $\mathbb{P}^1$ with five marked points. The foliation and stratification structures on these moduli spaces/stacks are investigated. In particular, we confirm Simpson's conjecture for the moduli space of parabolic logarithmic flat bundles with certain non-special weight system.
\end{abstract}

\renewcommand\abstractname{R\'esum\'e}

\begin{abstract}
Cet article \'etudie les espaces/champs de modules des fibr\'es paraboliques (fibr\'es plats logarithmiques paraboliques et fibr\'es de Higgs logarithmiques paraboliques avec spectre donn\'e) de rang 2 et de degr\'e 1 sur $\mathbb{P}^1$ avec cinq points marqu\'es. Les structures de feuilletage et de stratification sur ces espaces/champs de modules sont analys\'ees. En particulier, nous confirmons la conjecture de Simpson pour l'espace de modules des fibr\'es plats logarithmiques paraboliques avec un certain syst\`eme de poids non sp\'ecial.
\end{abstract}

\tableofcontents

\section{Introduction}

The notion of parabolic bundles on a compact Riemann surface $C$  with a set $D$ of marked points was first introduced by Seshadri \cite{se}. In this initial formulation, a parabolic bundle is a vector bundle endowed with a partial flag of the residual stalk at each marked point, and equipped with a real number in the range [0,1), called the weight, for each component of the flag. As a generalization of the celebrated Narasimhan--Seshadri theorem \cite{NS}, Mehta and Seshadri showed that there is a one-to-one correspondence between stable parabolic vector bundles over $C$ and irreducible unitary representations of the topological fundamental group $\pi_1(C\backslash D)$ with a fixed holonomy class around each marked point \cite{MS80}. In order to consider representations of $\pi_1(C\backslash D)$ into the general linear group, Simpson introduced the new object -- Higgs field compatible with the parabolic structure, and developed nonabelian Hodge theory in terms of filtered objects \cite{sx,s}. Maruyama and Yokogawa generalized the notion of parabolic bundles to the case of a higher dimensional variety with a reduced effective simple normal crossing divisor \cite{my,j}. Mochizuki established the parabolic nonabelian Hodge theory in the higher dimensional setting \cite{t, t1}. For rational weights, Biswas realized parabolic bundles on a smooth projective variety $X$ as orbifold bundles \cite{bis}, namely vector bundles on the Kawamata cover $p:Y\rightarrow X$ equipped with a lift of the action of the Galois group of $p$. Borne and Vistoli generalized Biswas's construction to the logarithmic scheme setting \cite{bv}; namely they showed the category of parabolic sheaves on a logarithmic scheme $X$ with weights lying in $\frac{\mathbb{Z}}{n}$ for some integer $n$ is equivalent to the category of quasi-coherent sheaves on the corresponding root stack $\sqrt[n]{X}$. Moreover, Talpo showed that parabolic sheaves with real weights on a fine saturated log analytic space can be treated as quasi-coherent sheaves of modules on its Kato--Nakayama space \cite{ta}.

Investigating the moduli space of parabolic bundles is an important problem. Generally speaking, there are two main approaches to construct moduli spaces. The first is algebraic: applying Mumford's geometric invariant theory, as done by the authors of \cite{MS80, bh,my,y}, where the crucial ingredient is the imposition of suitable stability conditions. The second is gauge-theoretic: identifying the moduli space with the space of solutions to certain Yang--Mills--Hitchin-type equations modulo gauge transformations, as done by the authors of \cite{p,ko,bq}. In this analytic approach, a key difficulty arises from the fact that some elliptic operators are not Fredholm anymore on the usual Sobolev spaces, so one needs to solve the equations in a smaller function space.

To establish a complete nonabelian Hodge theory, one should also take the Riemann--Hilbert problem into account, that is, the relationship between the Betti moduli space of representations of the fundamental group and the de Rham moduli space of flat connections. In the context of parabolic bundles (with non-special weights), this problem becomes more subtle since the Betti moduli space may have singularities when the local exponents (spectrum) of connections are special. As a result, new phenomena arise that do not occur in the absence of parabolic structures. For example, the Riemann--Hilbert correspondence can contract some families of compact subvarieties in the de Rham moduli space to the singular locus of the Betti moduli space\footnote{To establish a unified Riemann--Hilbert correspondence whenever the weights and the  spectra are special or non-special, we introduce the moduli space of parabolic representation pairs, consisting of a representation of the fundamental group and some parabolic subgroups \cite{hhq}.}. Inaba, Iwasaki and Saito studied this problem in detail for $C\simeq \mathbb{P}^1$ \cite{mk}, and later Inaba extended these results to the case when $C$ is a general Riemann surface \cite{m}.

Recently, the moduli spaces of parabolic (Higgs/flat) bundles have also appeared in the study of other interesting geometric objects. For example, they are closely related to gravitational instantons and Boalch's modularity conjecture \cite{Q1}, as well as to hyperpolygon spaces \cite{Q2,Q3}.

\vspace{2mm}

In \cite{lm}, Loray, Saito, and Simpson fixed a reduced effective divisor $\mathcal{D}=z_1+\cdots +z_4$ on $C\simeq\mathbb{P}^1$ and considered those pairs $(E, \nabla)$, where $E$ is a vector bundle of rank 2 and degree $d$ over $C$, and $\nabla : E\rightarrow E\otimes \Omega^1_{C}(\mathcal{D})$ is a connection having simple poles supported on $D=\{z_1, \cdots, z_4\}$. At each pole $z_i,i=1,\cdots,4$, we have two residual eigenvalues $\nu_i^+,\nu_i^-$, which together satisfy the Fuchs relation $d+\sum\limits_{i=1}^4(\nu_i^+ +\nu_i^-)=0$. Moreover, we can naturally introduce parabolic structures $\mathcal{L}=\{L_i\}_{i=1,\cdots,4}$ such that $L_i$ is a one dimensional subspace of $E|_{z_i}$ corresponding to an eigenspace of the residue of $\nabla$ at $z_i$ with the eigenvalue $\nu_i^+$. The dimension system of these parabolic structures is denoted by $\overrightarrow{d}$. Fixing the spectral data $\overrightarrow{\nu}=(\nu_1^+,\nu_1^-,\cdots,\nu_4^+,\nu_4^-)$, and introducing the weight system $\overrightarrow{w}$ for stability, by means of geometric invariant theory, we can construct the moduli space $M(d,\overrightarrow{d},\overrightarrow{\nu},\overrightarrow{w})$ of $\overrightarrow{w}$-stable parabolic logarithmic flat bundles of type $(d,\overrightarrow{d})$ with the spectrum $\overrightarrow{\nu}$. When the spectral data  $\overrightarrow{\nu}$ is non-special, $M(d,\overrightarrow{d},\overrightarrow{\nu},\overrightarrow{w})$ is independent of the choice of $\overrightarrow{w}$, and is then denoted by $M(d,\overrightarrow{d},\overrightarrow{\nu})$. It admits two kinds of fibrations, arising respectively from the apparent singularities of flat connections and from the parabolic structures. In fact, they can be derived from the $\mathbb{C}^\times$-limit into the moduli space of $\overrightarrow{w}$-stable strongly parabolic logarithmic Higgs bundles for $\overrightarrow{w}$ lying in the unstable zone and stable zone, respectively. Furthermore, these two fibrations are strongly transverse; that is, generic fibers intersect once.

In the present paper, we focus on various moduli spaces/stacks of parabolic bundles with logarithmic connections or Higgs fields over a Riemann sphere with five marked points, i.e. $(C,D)\simeq(\mathbb{P}^1,\{z_1,\cdots,z_5\})$.  Our main purpose is to generalize the results for $\mathbb{P}^1$ with four marked points in \cite{lm} to the case of five marked points. Similar considerations have also appeared in \cite{l,v,do}.

Let $M(1,\overrightarrow{d},\overrightarrow{\nu})$ be the moduli space of parabolic logarithmic flat bundles of type $(1,\overrightarrow{d})$ over $(C,D)$ with the non-special spectrum $\overrightarrow{\nu}$. Then there are two fibrations with 2-dimensional fibers on $M(1,\overrightarrow{d},\overrightarrow{\nu})$ as
\begin{align*}
  M^{\mathrm{ind}}(1,\overrightarrow{d})\xlongleftarrow{\varphi}M(1,\overrightarrow{d},\overrightarrow{\nu})\xlongrightarrow{\Psi_{\overrightarrow{w}}}\mathrm{Fix}(1,\overrightarrow{d},\overrightarrow{w}),
\end{align*}
where
\begin{itemize}
  \item $M^{\mathrm{ind}}(1,\overrightarrow{d})$ is the moduli space of indecomposable parabolic bundles of type $(1,\overrightarrow{d})$ over $(C,D)$;
  \item $\mathrm{Fix}(1,\overrightarrow{d},\overrightarrow{w})$ is the locus of $\mathbb{C}^\times$-fixed points lying in the moduli space $SH(1,\overrightarrow{d},\overrightarrow{w})$ of $\overrightarrow{w}$-stable strongly parabolic logarithmic Higgs bundles of type $(1, \overrightarrow{d})$ and parabolic degree zero over $(C,D)$ for a non-special weight system $\overrightarrow{w}=(w_1,\cdots,w_5)$ satisfying $\sum\limits_{i=1}^5w_i<1$;
\item the morphism $\varphi$ is just the forgetful map;
\item the morphism $\Psi_{\overrightarrow{w}}$ is given by taking the zero-limit of the $\mathbb{C}^\times$-action.
\end{itemize}

Let $N^{\mathrm{ind}}(1,\overrightarrow{d})$ be the moduli stack of indecomposable parabolic bundles of type $(1,\overrightarrow{d})$ over $(C,D)$. Then it is a $\mathbb{C}^\times$-gerbe over $M^{\mathrm{ind}}(1,\overrightarrow{d})$. By wall-crossing (variation of geometric invariant theory), the non-separated scheme $M^{\mathrm{ind}}(1,\overrightarrow{d})$ can be constructed as
\begin{align*}
      M^{\mathrm{ind}}(1,\overrightarrow{d})=P_0\biguplus P_1\biguplus P_2\biguplus P_3\biguplus P_4,
  \end{align*}
  where $P_0$ is isomorphic to a Del Pezzo surface of degree 4; $P_1, \cdots, P_4$ are all isomorphic to $ \mathbb{P}^2$; and $\biguplus$ stands for patching two charts via the morphisms from the Del Pezzo surface of degree 4 to $ \mathbb{P}^2$ along the maximal open subsets where they are one-to-one. Moreover, the connection between these two base schemes $M^{\mathrm{ind}}(1,\overrightarrow{d})$ and $\mathrm{Fix}(1,\overrightarrow{d},\overrightarrow{w})$ is given by the following theorem.
  
\begin{theorem}
 There exists a non-separated scheme $\mathfrak{F}$ such that
 \begin{itemize}
  \item $\mathfrak{F}$ is birationally equivalent to $M^{\mathrm{ind}}(1,\overrightarrow{d})$;
  \item $N^{\mathrm{ind}}(1,\overrightarrow{d})$ is a $\mathbb{C}^\times$-gerbe over $\mathfrak{F}$;
  \item $\mathfrak{F}$ contains an open subscheme $\mathfrak{F}'$, which is set-theoretically isomorphic to $\mathrm{Fix}(1,\overrightarrow{d},\overrightarrow{w})$ by a constructibly algebraic isomorphism.
\end{itemize}
\end{theorem}

The main result of this paper is to confirm the parabolic version of Simpson's foliation conjecture and (coarse granulated) stratification conjecture (cf. Conjecture \ref{c}) for the moduli space of parabolic logarithmic flat bundles of rank 2 and degree 1 over $\mathbb{P}^1$ with five marked points with given non-special spectrum and non-special \emph{unstable} weight system. More concretely, we show the following theorem.

\begin{theorem}\label{k}
Fix a non-special weight system $\overrightarrow{w}$ with $\sum\limits_{i=1}^5w_i<1$, and write $\mathrm{Fix}(1,\overrightarrow{d},\overrightarrow{w})=\coprod\limits_{\alpha}\mathrm{Fix}_\alpha(1,\overrightarrow{d},\overrightarrow{w})$ as the union of the connected components.
\begin{enumerate}
  \item $\Psi_{\overrightarrow{w}}: M(1,\overrightarrow{d},\overrightarrow{\nu})\rightarrow \mathrm{Fix}(1,\overrightarrow{d},\overrightarrow{w})$ is a surjective morphism with 2-dimensional fibers, which fit together into a regular foliation
on $ M(1,\overrightarrow{d},\overrightarrow{\nu})$.
  \item There is a coarse granulation of the index set $A=\{\alpha\}$ as
\begin{align*}
  A=A^0\coprod A^1,
\end{align*}
where
\begin{align*}
  A^0&=\{\alpha\in A: \dim_\mathbb{C}\mathrm{Fix}_\alpha(1,\overrightarrow{d},\overrightarrow{w})=0\}\neq \emptyset,\\
A^1&=\{\alpha\in A: \dim_\mathbb{C}\mathrm{Fix}_\alpha(1,\overrightarrow{d},\overrightarrow{w})=2\}\neq \emptyset,
\end{align*}
such that
\begin{align*}
 \overline{M^1(1,\overrightarrow{d},\overrightarrow{\nu})}\backslash M^1(1,\overrightarrow{d},\overrightarrow{\nu})=M^0(1,\overrightarrow{d},\overrightarrow{\nu}),
\end{align*}
where
\begin{align*}
  M^p(1,\overrightarrow{d},\overrightarrow{\nu})=\coprod_{\alpha\in A^p}M_\alpha(1,\overrightarrow{d},\overrightarrow{\nu}),\ p=0,1,
\end{align*}
with
\begin{align*}
  M_\alpha(1,\overrightarrow{d},\overrightarrow{\nu})
  =\ \{(E,\mathcal{L},\nabla)\in M(1,\overrightarrow{d},\overrightarrow{\nu})
 :\lim\limits_{c\rightarrow0}c\cdot (E,\mathcal{L},\nabla)\in \mathrm{Fix}_\alpha(1,\overrightarrow{d},\overrightarrow{w})\}.
\end{align*}
\end{enumerate}
\end{theorem}

\vspace{1mm}

\noindent\textbf{Related to other works}.
 For the case of parabolic logarithmic flat bundles of rank two over $\mathbb{P}^1$ with four marked points, Simpson's foliation conjecture was settled by Loray, Saito, and Simpson for \emph{any} non-special weight system \cite{lm}. When adding one marked point, the underlying vector bundle has two different types, which increases the complexity of the moduli space and the locus of $\mathbb{C}^\times$-fixed points. Therefore, in contrast with the case of four marked points, we not only describe the locus of $\mathbb{C}^\times$-fixed points in terms of a certain non-separated scheme, but also consider the deformations connecting different types by means of the \emph{middle logarithmic de Rham complex}. Recently, under the same setting, in comparison with our \emph{unstable} weight system, Fassarella and Loray confirmed Simpson's foliation conjecture with respect to some \emph{stable} weight system for the case of five marked points. Moreover, they showed that when the number of marked points is equal to or greater than 5, there exists a stable weight system such that the foliation conjecture is false \cite{fl}, more precisely, they checked that after a suitable choice of stable weight system, the regularity of foliation may be violated. However, if the number of marked points is greater than 5, it is not clear whether there exists a stable weight system such that the foliation conjecture still holds. Very recently, Esnault and Groechenig claimed a proof of the original Simpson's foliation conjecture on the moduli space of flat bundles over compact Riemann surfaces using positive characteristic methods \cite{eg}.

\vspace{3mm}

\noindent\textbf{Acknowledgements}.
 The authors would like to thank Ya Deng, Carlos Simpson and Kang Zuo for their useful discussions, and also thank Takuro Mochizuki and Frank Loray for their helpful communications. The authors are grateful to Fei Yu for introducing the excellent work \cite{e} of Eskin, Kontsevich, M\"{o}ller and Zorich, and to the anonymous referee for bringing the remarkable work \cite{do} of Donagi and Pantev to their attention. P. Huang would like to express deep gratitude to the Institut des Hautes \'Etudes Scientifiques, and the Max Planck Institute for Mathematics in the Sciences for their kind hospitality and support during the production of this paper.
  
  Z. Hu is supported by the National Natural Science Foundation of China (No. 12271253).  P. Huang is partially supported by the European Research Council (ERC) under the European Union's Horizon 2020 research and innovation program (grant agreement No 101018839) and Deutsche Forschungsgemeinschaft (DFG, Projektnummer 547382045). R. Zong is supported by the National Natural Science Foundation of China (Key Program 12331002).

\section{Parabolic Logarithmic Flat Bundles }

Let $C$ be a compact Riemann surface of genus $g$, $D=\{z_1,\cdots, z_n\}$ be a set consist of $n$ ($\geq 4$) distinct points in $C$, and let $X=C\backslash D$, $\mathcal{D}=z_1+\cdots +z_n$ be a reduced divisor on $C$. We first recall some definitions of parabolic bundles and parabolic logarithmic flat bundles.

\begin{definition}
Let $E$ be a vector bundle of rank $r$ and degree $d$ over $C$.
\begin{enumerate}
  \item  A \emph{quasi-parabolic structure} $\mathcal{L}$ on $E$ is given by
$\{\mathcal{F}^{(i)}\}_{i=1,\cdots, n}$, where 
 \begin{align*}
  \mathcal{F}^{(i)}: F^{(i)}_0=E|_{z_i}\supseteq F^{(i)}_1\supseteq\cdots \supseteq F^{(i)}_{\ell^{(i)}}\supset0=F^{(i)}_{\ell^{(i)}+1}
 \end{align*}
is a filtration of vector spaces on $E|_{z_i}$ of length $\ell^{(i)}$. Define
$\overrightarrow{d^{(i)}}=(d^{(i)}_{0},\cdots,d^{(i)}_{\ell^{(i)}})$, where $d^{(i)}_{\jmath}=\dim\mathrm{Gr}^{\mathcal{F}^{(i)}}_{\jmath}(E|_{z_i})$ for $ \mathrm{Gr}^{\mathcal{F}^{(i)}}_{\jmath}(E|_{z_i})=F^{(i)}_{\jmath}/F^{(i)}_{\jmath+1}$, $\jmath=1,\cdots,\ell^{(i)}$,
 then the collection $\{\overrightarrow{d^{(i)}}\}_{i=1,\cdots,n}$ is called the \emph{dimension system} of $\mathcal{L}$. In particular, $\mathcal{L}$ is called a \emph{ parabolic structure} if every component of $\overrightarrow{d^{(i)}},i=1,\cdots,n$,  is nonzero.
 The pair  $(E,\mathcal{L})$ is called a \emph{parabolic bundle } of type $(d, \{\overrightarrow{d^{(i)}}\}_{i=1,\cdots,n})$ on $(C,D)$ if $\mathcal{L}$ is a parabolic structure on $E$.
  \item A  \emph{proper  subbundle} $(E', \mathcal{L}')$  of $(E,\mathcal{L})$ consists of a proper subbundle  $E'$ of $ E$ with rank $r'>0$  and an induced quasi-parabolic structure  $\mathcal{L}'$ on $E'$ given by $\{\mathcal{F}'^{(i)}\}_{i=1,\cdots, n}$, where the filtration  $\mathcal{F}'^{(i)}$ is defined via $F'^{(i)}_\jmath=F^{(i)}_\jmath\bigcap E'|_{z_i}$, $\jmath=0,\cdots,\ell^{(i)}+1$. A parabolic  bundle $(E,\mathcal{L})$  is called \emph{indecomposable}, if it cannot be decomposed into the direct sum of proper  subbundles  of $(E,\mathcal{L})$ in the sense of equivalent classes of $E$.
  \item The \emph{weight system } $\{\overrightarrow{w^{(i)}}\}_{i=1,\cdots,n}$ on $(E,\mathcal{L})$ consists of  vectors $\overrightarrow{w^{(i)}}=(w^{(i)}_0,\cdots,w^{(i)}_{\ell^{(i)}})$ with  $0\leq w^{(i)}_0<\cdots<w^{(i)}_{\ell^{(i)}}<1$  at each $z_i$. The \emph{parabolic degree} of $(E,\mathcal{L})$ is defined as
\begin{align*}
\mathrm{ pdeg}(E,\mathcal{L})=\deg(E)+\sum_{i=1}^n\overrightarrow{d^{(i)}}\cdot\overrightarrow{w^{(i)}}.
\end{align*}
A parabolic  bundle $(E,\mathcal{L})$  is called $\{\overrightarrow{w^{(i)}}\}_{i=1,\cdots,n}$-\emph{stable} (resp. $\{\overrightarrow{w^{(i)}}\}_{i=1,\cdots,n}$-\emph{semistable}), if for any proper subbundle $(E',\mathcal{L}')$ the following inequality holds
       \begin{align*}
        \frac{\mathrm{ pdeg}(E',\mathcal{L}')}{r'}<(\textrm{resp.} \leq)\  \frac{\mathrm{ pdeg}(E,\mathcal{L})}{r}.
       \end{align*}
  \item  A parabolic bundle $(E,\mathcal{L})$ is called \emph{simple} if it has no non-scaler endomorphism of $E$ preserving the parabolic structure $\mathcal{L}$.
      \end{enumerate}
      \end{definition}

\begin{definition}
$E$ is the same as above.
      \begin{enumerate}
      \item A pair $(E,\nabla)$ is called a \emph{ logarithmic flat
bundle} over $(C,D)$  if
\begin{align*}
 \nabla: E\rightarrow E\otimes \Omega^1_C(\mathcal{D})
\end{align*}is a logarithmic connection on $E$. i.e. a connection having simple poles supported by $D$.
\item A logarithmic flat bundle $(E,\nabla)$ is called \emph{irreducible} if there is no nontrivial proper subbundle of $E$ which is preserved by the logarithmic connection $\nabla$.
\item For a logarithmic flat bundle $(E,\nabla)$, let $\overrightarrow{\nu^{(i)}}=(\nu^{(i)}_1,\cdots,\nu^{(i)}_r)$ be the eigenvalues (by algebraic multiplicities) of the residue $\mathrm{Res}_{z_i}(\nabla)$ of the logarithmic connection $\nabla$ at the point $z_i\in D$. The collection $\{\overrightarrow{\nu^{(i)}}\}_{i=1,\cdots,n}$, which satisfy the Fuchs relation
\begin{align*}
 d+\sum_{i=1}^n\sum_{\jmath=1}^r\nu^{(i)}_\jmath=0,
\end{align*}
 is called the \emph{spectrum} of $(E,\nabla)$.
  \item A triple $(E,\mathcal{L},\nabla)$ is called a \emph{parabolic logarithmic flat bundle} (resp. \emph{weak parabolic logarithmic flat
bundle}) over $(C,D)$ if

\begin{itemize}
  \item $(E,\mathcal{L})$ is  a parabolic bundle over  $(C,D) $;
  \item $(E,\nabla)$ is a logarithmic flat bundle;
  \item $(E,\nabla)$ is compatible with $(E,\mathcal{L})$, namely
  \begin{align*}
   (\mathrm{Res}_{z_i}(\nabla)-\nu^{(i)}_\jmath\mathrm{Id})F^{(i)}_\jmath\subseteq F^{(i)}_{\jmath+1}
  \end{align*}
 such that
 \begin{align*}
   \{\overrightarrow{\nu^{(i)}}=(\underbrace{\nu^{(i)}_0,\cdots,\nu^{(i)}_0}_{d^{(i)}_0},\cdots,\underbrace{\nu^{(i)}_{\ell^{(i)}},\cdots,\nu^{(i)}_{\ell^{(i)}}}_{d^{(i)}_{\ell^{(i)}}})\}_{i=1,\cdots,n}
 \end{align*}
is the spectrum of $(E,\nabla)$.

(resp. $(E,\nabla)$ is weakly compatible with $(E,\mathcal{L})$, namely $\mathrm{Res}_{z_i}(\nabla)$ preserves $\mathcal{F}^{(i)}$
  \begin{align*}
   (\mathrm{Res}_{z_i}(\nabla))F^{(i)}_\jmath\subset F^{(i)}_\jmath
  \end{align*}
with eigenvalues $_\jmath\nu^{(i)}_1,\cdots, {_\jmath\nu^{(i)}_{d^{(i)}_\jmath}}$ (by algebraic multiplicities) on $\mathrm{Gr}^{\mathcal{F}^{(i)}}_\jmath$ such that
 \begin{align*}
   \{\overrightarrow{\nu^{(i)}}=(_0\nu^{(i)}_1,\cdots, {_0\nu^{(i)}_{d^{(i)}_0}},\cdots, {_{\ell^{(i)}}\nu^{(i)}_1},\cdots, {_{\ell^{(i)}}\nu^{(i)}_{d^{(i)}_{\ell^{(i)}}}})\}_{i=1,\cdots,n}
 \end{align*}
is the spectrum of $(E,\nabla)$.)
\end{itemize}
  \item Given a weight system $\{\overrightarrow{w^{(i)}}\}_{i=1,\cdots,n}$, a (weakly) parabolic logarithmic flat bundle $(E,\mathcal{L},\nabla)$
bundle is called $\{\overrightarrow{w^{(i)}}\}_{i=1,\cdots,n}$-\emph{stable} (resp. $\{\overrightarrow{w^{(i)}}\}_{i=1,\cdots,n}$-\emph{semistable}), if for any proper subbundle $(E',\mathcal{L}')$ of $(E,\mathcal{L})$ with $E'$ being preserved by the logarithmic connection $\nabla$, the following inequality holds
       \begin{align*}
        \frac{\mathrm{ pdeg}(E',\mathcal{L}')}{r'}<(\textrm{resp.} \leq)\  \frac{\mathrm{ pdeg}(E,\mathcal{L})}{r}.
       \end{align*}
\end{enumerate}

\end{definition}

 Let $\mathcal{M}_{0,n}$ be the moduli space of $n$ marked points on $\mathbb{P}^1$, which is isomorphic to
\begin{align*}
 \underbrace{\mathcal{M}_{0,4}\times\cdots\times\mathcal{M}_{0,4}}_{n-3}\backslash\{\textrm{all diagonals}\}
\end{align*}
 with $\mathcal{M}_{0,4}\simeq \mathbb{P}^1\backslash\{0,1,\infty\}$, and take a universal family $\mathcal{U}$ over $\mathcal{M}_{0,n}$.
It is known that there exists a relative fine moduli scheme
\begin{align*}
 M(d,\{\overrightarrow{d^{(i)}}\}_{i=1,\cdots,n},\{\overrightarrow{w^{(i)}}\}_{i=1,\cdots,n})\rightarrow \mathcal{M}_{0,n}\times R(d)
\end{align*}
 of $\{\overrightarrow{w^{(i)}}\}_{i=1,\cdots,n}$-stable parabolic logarithmic flat bundles of type $(d, \{\overrightarrow{d^{(i)}}\}_{i=1,\cdots,n})$, which is smooth and quasi-projective \cite{mk,m}. The fiber $M_{\mathcal{U}_D,\{\overrightarrow{\nu^{(i)}}\}_{i=1,\cdots,n}}(d,\{\overrightarrow{d^{(i)}}\}_{i=1,\cdots,n},\{\overrightarrow{w^{(i)}}\}_{i=1,\cdots,n})$ over $(D,\{\overrightarrow{\nu^{(i)}}\}_{i=1,\cdots,n})\in \mathcal{M}_{0,n}\times R(d)$ is the moduli space of $\{\overrightarrow{w^{(i)}}\}_{i=1,\cdots,n}$-stable parabolic logarithmic flat bundles of type $(d, \{\overrightarrow{d^{(i)}}\}_{i=1,\cdots,n})$ with the spectrum
$\{\overrightarrow{\nu^{(i)}}\}_{i=1,\cdots,n}$. Fixing $(C,D)$, the moduli space $M_{\mathcal{U}_D,\{\overrightarrow{\nu^{(i)}}\}_{i=1,\cdots,n}}(d,\{\overrightarrow{d^{(i)}}\}_{i=1,\cdots,n},\{\overrightarrow{w^{(i)}}\}_{i=1,\cdots,n})$ is also simply denoted by $M(d,\{\overrightarrow{d^{(i)}}\}_{i=1,\cdots,n},\{\overrightarrow{w^{(i)}}\}_{i=1,\cdots,n},\{\overrightarrow{\nu^{(i)}}\}_{i=1,\cdots,n})$.

In this paper, unless otherwise stated,  we mainly focus on the following case: $C$ is the projective line $\mathbb{P}^1$, the rank $r$ of $E$ is 2, and the dimension system is $\{\overrightarrow{d^{(i)}}=(1,1)\}_{i=1,\cdots,n}$, which is replaced by the notation $\overrightarrow{d}$. Then  we also denote $F_1^{(i)}$ by $L_i$, thus the parabolic structure $\mathcal{L}$ is given by $\{L_i\}_{i=1,\cdots,n}$. On the other hand, the weight system  
$\{\overrightarrow{w^{(i)}}\}_{i=1,\cdots,n}$ is
 chosen as $w^{(i)}_0=0, i=1,\cdots,n$, then $\overrightarrow{w}$ is rewritten as $\overrightarrow{w}=(w_1,\cdots, w_n)$ with $w_i=w^{(i)}_1, i=1,\cdots,n$.

 For simplicity, we assume all marked points $z_i$, $i = 1,\cdots ,n$, lie in an affine part of $C$. Write $E\simeq\mathcal{O}_{\mathbb{P}^1}(d_0)\oplus \mathcal{O}_{\mathbb{P}^1}(d_1)$ with $d_0+d_1=d, d_0\leq d_1$, and let $e$ and $f$ be the unit section of $\mathcal{O}_{\mathbb{P}^1}(d_0)$ and $\mathcal{O}_{\mathbb{P}^1}(d_1)$, respectively, such that the zero or pole of $e,f$ is supported by $\infty\in C$, then over the affine part $C\backslash\{\infty\}$ with the coordinate $z$, the logarithmic connection $\nabla$ is given by
\begin{align*}
  \nabla=\mathrm{d}+\sum_{i=1}^n\frac{A_i}{z-z_i}\mathrm{d}z+G(z)\mathrm{d}z,
\end{align*}
where
\begin{align*}
  A_i=\left(
             \begin{array}{cc}
               A_i^{(11)} &  A_i^{(12)} \\
                A_i^{(21)} &  A_i^{(22)} \\
             \end{array}
           \right),\ G(z)=\left(
          \begin{array}{cc}
            G^{(11)}(z) & G^{(12)}(z) \\
            G^{(21)}(z) & G^{(22)}(z) \\
          \end{array}
        \right)
\end{align*}
 are $2\times 2$ matrices in terms of the trivialization $E|_{C\backslash\{\infty\}}\simeq \mathbb{C}e\oplus \mathbb{C}f$, and $G$ is holomorphic. Define new sections $e'=z^{-d_0}e, f'=z^{-d_1}f$, then over  the affine part $C\backslash\{0\}$, we have
\begin{align*}
  \nabla=\mathrm{d}+\sum_{i=1}^n\frac{A'_i}{z-z_i}\mathrm{d}z+G'(z)\mathrm{d}z,
\end{align*}
where
\begin{align*}
 A'_i=\left(
             \begin{array}{cc}
               A_i^{(11)} &  z^{d_1-d_0}A_i^{(12)} \\
               z^{d_0-d_1} A_i^{(21)} &  A_i^{(22)} \\
             \end{array}
           \right),\ G'(z)=\left(
          \begin{array}{cc}
            G^{(11)}(z)+d_0z^{-1} & z^{d_1-d_0}G^{(12)}(z) \\
           z^{d_0-d_1} G^{(21)}(z) & G^{(22)}(z) +d_1z^{-1}\\
          \end{array}
        \right)
\end{align*}
in terms of the trivialization $E|_{C\backslash\{0\}}\simeq \mathbb{C}e'\oplus \mathbb{C}f'$. Then the holomorphicity at $\infty$ leads to
\begin{itemize}
  \item $ \sum\limits_{i=1}^nA_i^{(11)}=-d_0, \sum\limits_{i=1}^nA_i^{(22)}=-d_1$;
  \item if $d_1-d_0+2\leq n$, $A_i^{(12)}$ satisfy
  \begin{align*}
    \sum_iA_i^{(12)}\sum_{k_1<\cdots<k_l\atop  k_1,\cdots,k_l\neq i}z_{k_1}\cdots z_{k_l}=0
  \end{align*}
  for $l=1,\cdots, d_1-d_0+1$, and if $d_1-d_0+2> n$, all $A_i^{(12)}$ vanish;
  \item  $\sum\limits_{i=1}^5A^{(21)}_i=0$ if  $d_0=d_1$;
  \item $ G^{(11)}= G^{(22)}=G^{(12)}=0$;
  \item if $d_1-d_0\geq 2$, $G^{(21)}(z)$ is a polynomial of degree of no greater than $d_1-d_0-2$, and if $d_1-d_0< 2$, $G^{(21)}(z)=0$.
\end{itemize}

In particular, we find that if $(E,\nabla)$ is an irreducible logarithmic flat bundle of degree $d$ over $(C,D)$, then if $d=2r$, $E$ is isomorphic to one of the following\footnote{For a real number $s$, one denotes by $[s]$ the integer $t$ satisfying $t\leq s$ and $0\leq |s-t|<1$.}
\begin{align*}
  \mathcal{O}_{\mathbb{P}^1}(r)\oplus \mathcal{O}_{\mathbb{P}^1}(r), \mathcal{O}_{\mathbb{P}^1}(r-1)\oplus \mathcal{O}_{\mathbb{P}^1}(r+1), \cdots, \mathcal{O}_{\mathbb{P}^1}(r+1-[\frac{n}{2}])\oplus\mathcal{O}_{\mathbb{P}^1}(r-1+[\frac{n}{2}]),
\end{align*}
and if $d=2r+1$, $E$ is isomorphic to one of the following
\begin{align*}
  \mathcal{O}_{\mathbb{P}^1}(r)\oplus \mathcal{O}_{\mathbb{P}^1}(r+1), \mathcal{O}_{\mathbb{P}^1}(r-1)\oplus \mathcal{O}_{\mathbb{P}^1}(r+2), \cdots, \mathcal{O}_{\mathbb{P}^1}(r+1-[\frac{n}{2}])\oplus\mathcal{O}_{\mathbb{P}^1}(r+[\frac{n}{2}]).
\end{align*}

We denote $\nu_1^{(i)}$ by $\nu_i^+$, $\nu_0^{(i)}$ by $\nu_i^-$, thus the spectrum of $(E,\nabla)$ is given by $
  \overrightarrow{\nu}=(\nu_1^+,\nu_1^-,\cdots,\nu_n^+,\nu_n^- )$.  Obviously, the parabolic structure adds nontrivial data only if $\nu_i^+=\nu_i^-$ at some $z_i$ since $L_i$ is exactly the eigenspace of $\mathrm{Res}_{z_i}(\nabla)$ with respect to the eigenvalue $\nu_i^+$.

\begin{definition}
We call $\overrightarrow{\nu}\in R(d):=\{\overrightarrow{\nu}=(\nu_1^+,\nu_1^-,\cdots,\nu_n^+,\nu_n^- )\in\mathbb{C}^{2n}: d+  \sum\limits_{i=1}^n(\nu_i^++\nu_i^-)=0\}\simeq \mathbb{C}^{2n-1}$
\begin{itemize}
    \item \emph{Kostov-generic}, if for any $\sigma_i\in \{+,-\}$, $i=1,\cdots, n$ we have
\begin{align*}
 \sum_{i=1}^n\nu_i^{\sigma_i}\notin\mathbb{Z};
\end{align*}
    \item \emph{non-resonant}, if
\begin{align*}
  \nu_i^+-\nu_i^-\notin\mathbb{Z}, \ \ i=1,\cdots,n;
\end{align*}
    \item \emph{non-special}, if it is both Kostov-generic and non-resonant.
\end{itemize}
\end{definition}

The following proposition is well-known (for example, see \cite{a,lm}).

\begin{proposition}\label{n}\
\begin{enumerate}
  \item Given Kostov-generic $\overrightarrow{\nu}\in R(d)$, let  $(E, \mathcal{L},\nabla)$ be  a parabolic logarithmic flat  bundle of type $(d,\overrightarrow{d})$ over $(C,D)$ with the spectrum $\overrightarrow{\nu}$, then
      \begin{enumerate}
        \item the underlying  logarithmic flat bundle $(E,\nabla)$ is irreducible, hence $d_1-d_0\leq n-2$ in terms of $E\simeq\mathcal{O}_{\mathbb{P}^1}(d_0)\oplus \mathcal{O}_{\mathbb{P}^1}(d_1)$ with $d_0\leq d_1$,
        \item the underlying  parabolic bundle $(E,\mathcal{L})$ is indecomposable,
        \item the underlying parabolic bundle $(E,\mathcal{L})$ is simple.
      \end{enumerate}
  \item If $(E,\mathcal{L})$ is an indecomposable parabolic bundle of type $(d,\overrightarrow{d})$ over $(C,D)$, then for a given $\overrightarrow{\nu}\in R(d)$ there is  a logarithmic connection $\nabla$ on $E$ such that $(E,\mathcal{L},\nabla)$ is a parabolic logarithmic flat  bundle with the spectrum $\overrightarrow{\nu}$.
\end{enumerate}

\end{proposition}

\begin{proof} (1)
(a)
Assume $F$ is a line subbundle of $E$ preserved by $\nabla$, then the residue of $\nabla|_F$ at each point  $z_i$ is exactly the number $\nu_i^{\sigma_i}$ for certain $\sigma_i\in\{+,-\}$. Indeed, for $z_i\in D$, denoting  $x_i=\mathrm{Res}_{z_i}(\nabla|_F)$ and  $y_i=\mathrm{Res}_{z_i}(\nabla|_{E/F})$, we have
\begin{align*}
  x_iy_i&=\nu_i^+\nu_i^-,\\
x_i+y_i&=\nu_i^++\nu_i^-,
\end{align*}
thus
\begin{align*}
  x_i&=\frac{\nu_i^++\nu_i^-\mp|\nu_i^+-\nu_i^-|}{2}\in \{\nu_i^+,\nu_i^-\},\\
y_i&=\frac{\nu_i^++\nu_i^-\pm|\nu_i^+-\nu_i^-|}{2}\in \{\nu_i^+,\nu_i^-\}.
\end{align*} Then by Fuchs relation,
\begin{align*}
  \sum_{i=1}^n\nu_i^{\sigma_i}=-\deg(F)\in \mathbb{Z},
\end{align*}
which contradicts to the Kostov-genericity condition.

 (b) Assume $(E,\mathcal{L})$ is decomposed into the direct sum of two proper subbundles as $(E,\mathcal{L})\simeq(E_1,\mathcal{L}_1)\oplus (E_2,\mathcal{L}_2)$.  According to the decomposition of $E$, one writes $\nabla=\left(
                  \begin{array}{cc}
                        \nabla_1 & \theta \\
                        \xi & \nabla_2 \\
                  \end{array}
         \right)$, 
thus $(E_1,\nabla_1)$ and $(E_2,\nabla_2)$ form  logarithmic flat line bundles. Note that  $\theta$ and $\xi$ are non-zero, and the residues at each point $z_i\in D$ of  $\theta$ and $\xi$ cannot be non-zero simultaneously. Then we have the identities
\begin{align*}
d&=\deg(E_1)+\deg(E_2)\\
&=\sum_{z_i\in D} \mathrm{Res}_{z_i}(\nabla_1)+\sum_{z_i\in D}\mathrm{Res}_{z_i}(\nabla_2)=\sum_{i=1}^n\nu_i^{\sigma_i}
\end{align*}
for certain $\sigma_i\in\{+,-\}$,
 which also contradicts to the Kostov-genericity condition.

 (c)  Let $e,f$ be the unit sections of $\mathcal{O}_{\mathbb{P}^1}(d_0)$, $ \mathcal{O}_{\mathbb{P}^1}(d_1) $, respectively, then the 1-dimensional subspaces $L_i$ at $z_i\in D$ is generated by $\kappa_ie(z_i)+\lambda_if(z_i)$, hence it is parameterized by $u_i=[\kappa_i:\lambda_i]\in \mathbb{P}^1$. Let $\sigma$ be an automorphism of the parabolic bundle $(E,\mathcal{L})$.
 \begin{description}
   \item[Case I $d_1>d_0$] In terms of $E\simeq\mathcal{O}_{\mathbb{P}^1}(d_0)\oplus \mathcal{O}_{\mathbb{P}^1}(d_1)$,  one writes
  \begin{align*}
   \sigma=\left(
           \begin{array}{cc}
             \rho & 0 \\
            \zeta & \rho' \\
           \end{array}
         \right)
  \end{align*}
  for $\rho,\rho'\in \mathbb{C}^{\times}, \zeta\in H^0(\mathbb{P}^1, \mathcal{O}_{\mathbb{P}^1}(d_1-d_0))$.
The action of  $\sigma$ on  the parabolic structure $\mathcal{L}=\{L_i\}_{i=1,\cdots,n}$ is given by
\begin{align*}
\sigma\cdot(u_1,\cdots,u_n)=(u_1^\prime,\cdots, u_n^\prime),
\end{align*}
where
\begin{align*}
  u_i^\prime=[\rho\kappa_i:\zeta(z_i)\kappa_i+\rho'\lambda_i]=[\kappa_i:\frac{\zeta(z_i)\kappa_i+\rho'\lambda_i}{\rho}].
\end{align*}
Let $D_\infty=\{z_i\in D: L_i=\mathcal{O}_{\mathbb{P}^1}(d_1)\}, \widehat{D_\infty}=D\backslash D_\infty$  and $n_\infty=|D_\infty|$.  Consider the short exact sequence
\begin{align*}
  0\rightarrow\mathcal{O}_{\mathbb{P}^1}(-n+n_\infty+d_1-d_0)\rightarrow\mathcal{O}_{\mathbb{P}^1}(d_1-d_0)\rightarrow\bigoplus_{z_i\in \widehat{D_\infty}}\mathcal{O}_{\mathbb{P}^1}(d_1)|_{z_i}\rightarrow0,
\end{align*}
which induces the  exact sequence
\begin{align*}
 0\rightarrow H^0(\mathbb{P}^1,\mathcal{O}_{\mathbb{P}^1}(d_1-d_0))\rightarrow\bigoplus_{z_i\in \widehat{D_\infty}}\mathcal{O}_{\mathbb{P}^1}(d_1)|_{z_i}\rightarrow H^1(\mathbb{P}^1,\mathcal{O}_{\mathbb{P}^1}(-n+n_\infty+d_1-d_0))\rightarrow 0.
\end{align*}
 Therefore \begin{align*}
&\bigoplus_{z_i\in \widehat{D}}\mathcal{O}_{\mathbb{P}^1}(d_1)|_{z_i}/(H^0(\mathbb{P}^1,\mathcal{O}_{\mathbb{P}^1}(d_1-d_0))\ltimes\mathbb{C}^\times)\\
\simeq &\ H^1(\mathbb{P}^1,\mathcal{O}_{\mathbb{P}^1}(-n+n_\infty+d_1-d_0))/\mathbb{C}^\times
 \simeq  \ H^0(\mathbb{P}^1,\mathcal{O}_{\mathbb{P}^1}(n-n_\infty-d_1+d_0-2)^*/\mathbb{C}^\times.
\end{align*}
It follows that the indecomposablity of $(E,\mathcal{L})$ guarantees $n-n_\infty\geq d_1-d_0+2$. On the other hand, $\sigma$ preserving the parabolic structure leads to
\begin{align*}
  \zeta(z_i)+(\rho'-\rho)u_i=0
\end{align*}
for any $z_i\in \widehat{D}$. There are following two cases.
\begin{itemize}
  \item If $\rho=\rho'$, i.e. $\zeta(z_i)=0$ for any $z_i\in \widehat{D}$,  we must have $d_1-d_0\geq n-n_\infty$, which is a contradiction.

  \item If $\rho\neq\rho'$, we take a subbundle $L\subset E$  generated by the section $e+\frac{\zeta(z)}{\rho-\rho'}f$, then $L\simeq  \mathcal{O}_{\mathbb{P}^1}(d_0)$ and
 $L|_{z_i}=L_i$ for any  $z_i\in \widehat{D}$, which contradicts to the indecomposablity of  $(E,\mathcal{L})$.
\end{itemize}

\item[Case II $d_1=d_0$] In terms of $E\simeq\mathcal{O}_{\mathbb{P}^1}(d_0)\oplus \mathcal{O}_{\mathbb{P}^1}(d_0)$,  one writes
  \begin{align*}
   \sigma=\left(
           \begin{array}{cc}
             \rho & \zeta \\
            \eta & \rho' \\
           \end{array}
         \right)
  \end{align*}
  for $\rho\in \mathbb{C}^{\times}, \zeta,\eta\in \mathbb{C}$. There are following two subcases.
  \begin{description}
    \item[Subcase I $\eta=0$] The argument is the same as the Case I.
    \item [Subcase II $\eta\neq0$] We have
    \begin{align*}
      u_i=\frac{\rho'-\rho+\sqrt{(\rho'-\rho)^2+4\eta\zeta}}{2\eta}
    \end{align*}
    for any $z_i\in D$. The right hand side of the above equality may take two values denoted by $\omega_1, \omega_2$, then there are two subbundles $L_1, L_2$ of $E$ generated by the section $e+\omega_1f$ and $e+\omega_2f$, respectively. It is obvious that $L_1\simeq L_2\simeq \mathcal{O}_{\mathbb{P}^1}(d_0)$ and $L_1|_{z_i}=L_i$ for $z_i$ with $ \frac{\lambda_i}{\kappa_i}=\omega_1$, $L_2|_{z_i}=L_i$ for $z_i$ with $ \frac{\lambda_i}{\kappa_i}=\omega_2$. This also contradicts to the indecomposablity of  $(E,\mathcal{L})$.
  \end{description}
\end{description}
 Finally, we conclude that the automorphism $\sigma$ has the form of $\rho\cdot\textrm{Id}$ for some nonzero constant $\rho$.

(2) This claim is just the application of Weil criterion of parabolic version \cite{b,w}.
\end{proof}

Let $N^{\mathrm{ind}}(d_0,d_1,\overrightarrow{d})$ be the moduli stack of parabolic structures with dimension system $\overrightarrow{d}$ on $E\simeq\mathcal{O}_{\mathbb{P}^1}(d_0)\oplus \mathcal{O}_{\mathbb{P}^1}(d_1)$ ($d_1-d_o\leq n-2$).
\begin{description}
   \item[Case I $d_1>d_0$]Let \begin{align*}
 U_{i_1\cdots i_\ell}&=\{(u_1, \cdots,u_n): u_{i_1}, \cdots,u_{i_\ell}\textrm{ lie in } \mathbb{C}, \textrm{ and other } u_i \textrm{ take } \infty\}\simeq\mathbb{C}^{\ell},
\end{align*}
where $ \ell, i_1,\cdots,i_n\in\{1,\cdots,n\}, i_1<\cdots<i_n$, be the set of parabolic structures on $E$ given by $(u_1, \cdots,u_n)\in (\mathbb{P}^1)^n$, and
let $A_E$ denote the automorphism of $E$. Since  $d_1>d_0$, $A_E$ preserves  $U_{i_1\cdots i_\ell}$, and the dimension of the $A_E$-quotient is given by
\begin{align*}
      \dim_\mathbb{C} U_{i_1\cdots i_\ell}/A_E&=
                            \dim_\mathbb{C} H^0(\mathbb{P}^1,\mathcal{O}_{\mathbb{P}^1}(\ell-d_1+d_0-2))-1\\
                           & = \ell-d_1+d_0-2.
\end{align*}
There are three subcases.
\begin{description}
  \item[Subcase I $\ell<d_1-d_0+2 $] $U_{i_1\cdots i_\ell}/A_E$ is a set of single point.
  \item[Subcase II $\ell=d_1-d_0+2 $]  We write
 \begin{align*}
  \zeta(z_i)=a_{d_1-d_0}z_i^{d_1-d_0}+a_{d_1-d_0-1}z_i^{d_1-d_0-1}+\cdots+a_0
 \end{align*}
 for $z_i\in \widehat{D_\infty}$,
and introduce the matrix
 \begin{align*}
       \Pi_{\widehat{D_\infty}}=\left(
                                  \begin{array}{ccccc}
                                        z_{i_1}^{d_1-d_0}&z_{i_1}^{d_1-d_0-1}&\cdots&1&u_{i_1}\\
                                        z_{i_2}^{d_1-d_0}&z_{i_2}^{d_1-d_0-1}&\cdots&1&u_{i_2}\\
                                                            \vdots&\vdots&\vdots&\vdots&\cdots\\
                                        z_{i_{d_1-d_0+2}}^{d_1-d_0}&z_{i_{d_1-d_0+2}}^{d_1-d_0-1}&\cdots&1&u_{i_{d_1-d_0+2}}
                                   \end{array}
                                  \right),
\end{align*}
 then $U_{i_1\cdots i_{d_1-d_0+2}}=U'_{i_1\cdots i_{d_1-d_0+2}}\coprod U''_{i_1\cdots i_{d_1-d_0+2}}$, where $U'_{i_1\cdots i_{d_1-d_0+2}}$ and $U''_{i_1\cdots i_{d_1-d_0+2}}$
 are two $A_E$-invariant subsets given by, respectively,
 \begin{align*}
 U_{i_1\cdots i_{d_1-d_0+2}}'&=\{(u_{i_1},\cdots,u_{i_{d_1-d_0+2}}):u_{i_1},\cdots,u_{i_{d_1-d_0+2}}\in \mathbb{C}, \det \Pi_{\widehat{D_\infty}}\neq0\},\\
  U_{i_1\cdots i_{d_1-d_0+2}}''&=\{(u_{i_1},\cdots,u_{i_{d_1-d_0+2}}):u_{i_1},\cdots,u_{i_{d_1-d_0+2}}\in \mathbb{C}, \det \Pi_{\widehat{D_\infty}}=0\}.
\end{align*}
The parabolic structures on $E$ parameterized by $U_{i_1\cdots i_{d_1-d_0+2}}'$ are indecomposable.
  \item[Subcase III $\ell>d_1-d_0+2$] The indecomposable parabolic structures on $E$ parameterized by the set
\begin{align*}
 \coprod\limits_{1\leq i_1<\cdots<i_{\ell}\leq n} U'_{i_1\cdots i_{\ell}}= \coprod\limits_{1\leq i_1<\cdots<i_{\ell}\leq n} U_{i_1\cdots i_{\ell}}\backslash[( \underbrace{ 0,\cdots,0}_{\ell})]
\end{align*}
where
$[(\underbrace{0,\cdots,0}_{\ell})]$ denotes the $A_E$-orbit of $(\underbrace{ 0,\cdots,0}_{\ell})$ lying in $U_{i_1\cdots i_{\ell}}$. For the $A_E$-quotient, we have
\begin{align*}
  U'_{i_1\cdots i_{\ell}}/A_E&\simeq (\ H^0(\mathbb{P}^1,\mathcal{O}_{\mathbb{P}^1}(\ell-d_1+d_0-2)^*\backslash\{0\})/\mathbb{C}^\times\\
  &\simeq \mathbb{P}^{\ell-d_1+d_0-2}.
\end{align*}\end{description}

\item[Case II $d_1=d_0$]The set of $A_E$-equivalent classes of parabolic structures on $E$ is the quotient $(\mathbb{P}^1)^n/\mathrm{PGL}(2,\mathbb{C})$. Moreover,
 restricting on the set of indecomposable parabolic structures, we consider the quotient
 \begin{align*}
  \{(u_1,\cdots,u_n): u_1,\cdots,u_n\in \mathbb{P}^1, \textrm{ there exist at least three distinct } u_i's\}/\mathrm{PGL}(2,\mathbb{C}).
 \end{align*}
  For the set  $D^\ell=\{z_2,\cdots,z_\ell\}$ and the subset   $D_0^\ell\subsetneq \{z_2,\cdots,z_\ell\}$ with $\ell\in\{2,\cdots, n-1\}$, one defines
  \begin{align*}
     U_{D_0^\ell}=\{&(u_1,u_2,\cdots,u_\ell,u_{\ell+1},\cdots,u_n)\in (\mathbb{P}^1)^n: \\
    & u_i=u \textrm{ for } z_i\in\{z_1\}\bigcup D_0^\ell, z_j=v \textrm{ for } z_j\in D^\ell\backslash D_0^\ell, u\neq v\neq u_{\ell+1} \},
  \end{align*}
  then

 \begin{align*}
   U_{D_0^\ell}/\mathrm{PGL}(2,\mathbb{C})&\simeq\{(0,u_2,\cdots,u_\ell,1,u_{\ell+2},\cdots,u_n):u_i\left\{
                                                                         \begin{array}{ll}
                                                                          = 0, & \hbox{$z_i\in D_0^\ell$;} \\
                                                                          = \infty, & \hbox{$z_i\in D^\ell\backslash D_0^\ell$;} \\
                                                                           \in\mathbb{P}^1, & \hbox{$i=\ell+2,\cdots, n$}
                                                                         \end{array}
          \right.
   \}\\
   &\simeq(\mathbb{P}^1)^{n-\ell-1}.
 \end{align*}

 \end{description}

In summary, we have the following theorem.

\begin{theorem}\label{q}Let $ N^{\mathrm{ind}}(d,\overrightarrow{d})$ be the moduli stack of indecomposable parabolic bundles of type $(d,\overrightarrow{d})$ over $(C,D)$, then
\begin{align*}
 N^{\mathrm{ind}}(d,\overrightarrow{d})=\left\{
                                          \begin{array}{ll}
                                         N^{\mathrm{ind}}(r,r,\overrightarrow{d})\coprod N^{\mathrm{ind}}(r-1,r+1,\overrightarrow{d})\coprod\cdots\coprod N^{\mathrm{ind}}(r+1-[\frac{n}{2}],r-1+[\frac{n}{2}],\overrightarrow{d}), & \hbox{$d=2r$;} \\
                                          N^{\mathrm{ind}}(r,r+1,\overrightarrow{d})\coprod N^{\mathrm{ind}}(r-1,r+2,\overrightarrow{d})\coprod\cdots\coprod N^{\mathrm{ind}}(r+1-[\frac{n}{2}],r+[\frac{n}{2}],\overrightarrow{d}), & \hbox{$d=2r+1$,}
                                          \end{array}
                                        \right.
\end{align*}
where
\begin{itemize}
  \item  $N^{\mathrm{ind}}(d_0,d_1,\overrightarrow{d})$ ($d_0<d_1$) has
 a stratification
\begin{align*}
  N^{\mathrm{ind}}(d_0,d_1,\overrightarrow{d})=\coprod_{\alpha=0}^{n-d_1+d_0-2}\mathfrak{N}_{\alpha},
\end{align*}
with the  stratum
\begin{align*}
  \mathfrak{N}_{\alpha}=
                              \coprod\limits_{1\leq i_1<\cdots<i_{\alpha+d_1-d_0+2}\leq n}[U'_{i_1\cdots i_{\alpha+d_1-d_0+2}}/A_E],
\end{align*}
thus $\mathfrak{N}_{\alpha}$ has $C_n^{n-d_1+d_0-2-\alpha}$ connected components  isomorphic to $\alpha$-dimensional projective stack $\hat{\mathbb{P}^\alpha}$,
  \item $N^{\mathrm{ind}}(d_0,d_0,\overrightarrow{d})$ is a disjoint union as
 \begin{align*}
  N^{\mathrm{ind}}(d_0,d_0,\overrightarrow{d})=&\coprod_{\alpha=0}^{n-3}\mathfrak{N}_{\alpha},
 \end{align*}
 for
 \begin{align*}
  \mathfrak{N}_{\alpha}=                           \coprod\limits_{D_0^{n-1-\alpha}}[U_{D_0^{n-1-\alpha}}/\mathrm{PGL}(2,\mathbb{C})],
\end{align*}
thus $ \mathfrak{N}_{\alpha}$ has $2^{n-1-\alpha}-1$ components isomorphic to the product of  $\alpha$ copies of the   projective stack $\hat{\mathbb{P}^1}$.
\end{itemize}

\end{theorem}

There is a coarse moduli space $M^{\mathrm{ind}}(d,\overrightarrow{d})$ associated to the moduli stack $N^{\mathrm{ind}}(d,\overrightarrow{d})$, which is a non-separated scheme. It is obvious that $N^{\mathrm{ind}}(d,\overrightarrow{d})\simeq N^{\mathrm{ind}}(d',\overrightarrow{d})$ iff $d\equiv d'$ ($\mathrm{mod}$ 2), in contrast, $M^{\mathrm{ind}}(d,\overrightarrow{d})\simeq M^{\mathrm{ind}}(d',\overrightarrow{d})$, which can be realized by elementary transformations (see Appendix \ref{aa}). Assume $\overrightarrow{\nu}\in R(d)$ is Kostov-generic,  since the parabolic logarithmic flat  bundles are stable with respect to any weight system $\overrightarrow{w}$, we have the moduli space $M(d,\overrightarrow{d},\overrightarrow{\nu})$ of parabolic logarithmic flat  bundles of type $(d,\overrightarrow{d})$ over $(C,D)$ with the spectrum $\overrightarrow{\nu}$.

\begin{theorem}[\cite{a}] 
Given Kostov-generic $\overrightarrow{\nu}\in R(d)$, $M(d,\overrightarrow{d},\overrightarrow{\nu})$ is of dimensions $2n-6$, and  the forgetful map defines a fibration
\begin{align*}
 \varphi:M(d,\overrightarrow{d},\overrightarrow{\nu})&\rightarrow M^{\mathrm{ind}}(d,\overrightarrow{d})\\
(E,\mathcal{L},\nabla)&\mapsto (E,\mathcal{L}),
\end{align*}
 whose fiber is an affine space of dimensions $n-3$.
\end{theorem}

\begin{proof}
Let\begin{align*}
 \mathcal{E}^0&=\{\Phi\in \mathcal{E}nd(E): \Tr(\Phi)=0\},\\
 \mathcal{E}^0_\mathcal{L}&=\{\Phi\in \mathcal{E}nd(E): \Tr(\Phi)=0, \Phi (L_i)\subset L_i, i=1,\cdots,n\},\\
  \mathcal{E}^1_\mathcal{L}&=\{\Phi\in \mathcal{E}nd(E)\otimes\Omega^1_C(\mathcal{D}): \Tr(\Phi)=0, {\mathrm {Res}}_{z_i}(\Phi)L_i=0, i=1,\cdots,n\}.
\end{align*}
Pick $(E,\mathcal{L})\in M^{\mathrm{ind}}(d,\overrightarrow{d})$, then by Proposition \ref{n}, the fiber of $\varphi$ at $(E,\mathcal{L})$ is an affine space isomorphic to $H^0(C,\mathcal{E}^1_\mathcal{L})\simeq H^1(C,\mathcal{E}^0_\mathcal{L})$. Since $H^0(C,\mathcal{E}^0_\mathcal{L})=0$, due to the Riemann--Roch theorem,
\begin{align*}
  \dim_\mathbb{C}H^1(C,\mathcal{E}^0_\mathcal{L})&=-\chi(\mathcal{E}^0_\mathcal{L})\\
  &=-\chi(\mathcal{E}^0)+n=n-3,
\end{align*}
where $\chi(\bullet)$ denotes the Euler characteristic  number of sheaves.
Therefore,
\begin{align*}
 \dim_\mathbb{C}M(d,\overrightarrow{d},\overrightarrow{\nu})&= \dim_\mathbb{C}H^1(C,\mathcal{E}^0_\mathcal{L})+ \dim_\mathbb{C}N^{\mathrm{ind}}(d,\overrightarrow{d})\\
 &=(n-3)+(n-3)=2n-6.
\end{align*}
We thus complete the proof.
\end{proof}

For the case $n=5, d=1$, the explicit construction of $M^{\mathrm{ind}}(1,\overrightarrow{d})$ by wall-crossing (variation of geometric invariant theory) is given in Appendix \ref{ab} (see also \cite{l}), and for the case $n=5, d=0$,  $M^{\mathrm{ind}}(0,\overrightarrow{d})$ is also studied in \cite{do}.

\begin{example}Assume $n=5$.
We choose a base point $x_0\in C$, and let $\gamma_i$ be the paths going from $x_0$ around $z_i$ and back, then $\gamma_1,\cdots,\gamma_5$ form the generators of the fundamental group $\pi_1(C,x_0)$. Let $\mathcal{M}(C,\mathrm{SL}(2,\mathbb{C}))$ be the moduli space of   semisimple representations $\rho:\pi_1(C,x_0)\rightarrow \mathrm{SL}(2,\mathbb{C})$ as a categorical quotient associated to the moduli stack $[\textrm{Rep}(\pi_1(C,x_0), \mathrm{SL}(2,\mathbb{C}))/\mathrm{PSL}(2,\mathbb{C}))]$, and let $M(C,\mathrm{SL}(2,\mathbb{C}))$ be the subvariety of $\mathcal{M}(C,\mathrm{SL}(2,\mathbb{C}))$ by imposing the conditions that $\Tr(\rho(\gamma_i))=0$, $i=1,\cdots,5$.
Simpson has showed that $M(C,\mathrm{SL}(2,\mathbb{C}))$ is a 2:1 ramified covering over the hypersurface $H\subset \mathbb{C}^5$ defined by the equation \cite{s2}
\begin{align*}
 f(x,y,x,u,v)=xyzuv+ (x^2y^2+y^2z^2+z^2u^2+u^2v^2+v^2x^2)-4(x^2+y^2+z^2+u^2+v^2)+16=0.
\end{align*}
Now define $\overrightarrow{\nu}\in R(1)$ by
 \begin{align*}
   (\nu_i^+,\nu_i^-)=\left\{
   \begin{array}{ll}
     (\frac{1}{4},-\frac{1}{4}), & \hbox{$i=1,\cdots,4$;} \\
     (-\frac{1}{4},-\frac{3}{4}), & \hbox{$i=5$,}
   \end{array}
 \right.
 \end{align*}
and $\overrightarrow{\tilde\nu}\in R(0)$ by
\begin{align*}
 (\tilde\nu_i^+,\tilde\nu_i^-)=(\frac{1}{4},-\frac{1}{4}),\  i=1,\cdots,5.
\end{align*}
 Obviously, both $\overrightarrow{\nu}$ and $\overrightarrow{\tilde\nu}$ are non-special. By the Riemann--Hilbert correspondence \cite{mk}, the moduli spaces  $M(1,\overrightarrow{d},\overrightarrow{\nu})$ and $M(0,\overrightarrow{d},\overrightarrow{\tilde\nu})$ are both analytically isomorphic to $M(C,\mathrm{SL}(2,\mathbb{C}))$. On the other hand, an algebraic isomorphism between $M(1,\overrightarrow{d},\overrightarrow{\nu})$ and $M(0,\overrightarrow{d},\overrightarrow{\tilde\nu})$ can be obtained by elementary transformations.
\end{example}

 For the cases of higher ranks, the corresponding moduli spaces are more complicated. In Appendix \ref{C}, we apply Katz's middle convolution algorithm to establish isomorphisms between moduli spaces of logarithmic flat bundles of different ranks.

\section{Parabolic Logarithmic Higgs Bundles}

\begin{definition}Let $C$ be a compact Riemann surface.
\begin{enumerate}

  \item A pair  $(E,\theta)$ is called a \emph{logarithmic Higgs
bundle} over $(C,D)$ if $\theta$ is a logarithmic Higgs field, i.e. $\theta\in H^0(C,\End(E)\otimes \Omega^1_C(\mathcal{D}))$.
Let $\overrightarrow{\mu^{(i)}}=(\mu^{(i)}_1,\cdots,\mu^{(i)}_r)$ be the eigenvalues (with algebraic multiplicities) of the residue $\mathrm{Res}_{z_i}(\theta)$ of the logarithmic Higgs field $\theta$ at the point $z_i\in D$. The collection $\{\overrightarrow{\mu^{(i)}}\}_{i=1,\cdots,n}$,
   which satisfies the Fuchs relation
\begin{align*}
\sum_{i=1}^n\sum_{\jmath=1}^r\mu^{(i)}_\jmath=0,
\end{align*}is called the \emph{spectrum } of $(E,\theta)$.
\item A pair  $(E,\theta)$ is called a \emph{parabolic logarithmic   Higgs
bundle} (resp. \emph{weakly parabolic logarithmic Higgs
bundle}) over $(C,D)$ if
\begin{itemize}
  \item $(E,\mathcal{L})$ is  a parabolic bundle over  $(C,D) $;
  \item $(E,\theta)$ is a logarithmic Higgs
bundle over $(C,D)$;
  \item $(E,\theta)$ is compatible with $(E,\mathcal{L})$, namely
  \begin{align*}
   (\mathrm{Res}_{z_i}(\theta)-\mu^{(i)}_\jmath\mathrm{Id})F^{(i)}_\jmath\subseteq F^{(i)}_{\jmath+1}
  \end{align*}
 such that the collection
 \begin{align*}
   \{\overrightarrow{\mu^{(i)}}=(\underbrace{\mu^{(i)}_0,\cdots,\mu^{(i)}_0}_{d^{(i)}_0},\cdots,\underbrace{\mu^{(i)}_{\ell^{(i)}},\cdots,\mu^{(i)}_{\ell^{(i)}}}_{d^{(i)}_{\ell^{(i)}}})\}_{i=1,\cdots,n}
 \end{align*}
is the spectrum of $(E,\theta)$. In particular, if all $\mu^{(i)}_\jmath$'s vanish,  $(E,\mathcal{L},\theta)$ is called a \emph{strongly  parabolic logarithmic Higgs bundle}.

(resp. $(E,\theta)$ is weakly compatible with $(E,\mathcal{L})$, namely $\mathrm{Res}_{z_i}(\theta)$ preserves $\mathcal{F}^{(i)}$
  \begin{align*}
   (\mathrm{Res}_{z_i}(\theta))F^{(i)}_\jmath\subset F^{(i)}_\jmath
  \end{align*}
with eigenvalues $_\jmath\mu^{(i)}_1,\cdots, {_\jmath\mu^{(i)}_{d^{(i)}_\jmath}}$ (by algebraic multiplicities) on $\mathrm{Gr}_\jmath^{\mathcal{F}^{(i)}}$ such that the collection 
 \begin{align*}
   \{\overrightarrow{\mu^{(i)}}=(_0\mu^{(i)}_1,\cdots, {_0\mu^{(i)}_{d^{(i)}_0}},\cdots, {_{\ell^{(i)}}\mu^{(i)}_1},\cdots, {_{\ell^{(i)}}\mu^{(i)}_{d^{(i)}_{\ell^{(i)}}}})\}_{i=1,\cdots,n}
 \end{align*}
is the spectrum of $(E,\theta)$. In particular, if all $_\jmath\mu^{(i)}_a$'s vanish,  $(E,\mathcal{L},\theta)$ is called a \emph{nilpotent parabolic logarithmic Higgs bundle}.) \end{itemize}

  \item
      Given a  weight system $\{\overrightarrow{w^{(i)}}\}_{i=1,\cdots,n}$, a (weakly) parabolic logarithmic Higgs bundle $(E,\mathcal{L},\theta)$ over $(C,D)$ is called \emph{$\{\overrightarrow{w^{(i)}}\}_{i=1,\cdots,n}$-stable} (resp. \emph{$\{\overrightarrow{w^{(i)}}\}_{i=1,\cdots,n}$-semistable}), if for any proper  subbundle $(E',\mathcal{L}')$ of $(E,\mathcal{L})$ with $E'$ being preserved by the  logarithmic Higgs field  $\theta$, the following inequality holds
       \begin{align*}
        \frac{\mathrm{ pdeg}(E',\mathcal{L}')}{r'}<(\textrm{resp.} \leq)\  \frac{\mathrm{ pdeg}(E,\mathcal{L})}{r}.
       \end{align*}
       \end{enumerate}
\end{definition}

We introduce the following notations of moduli spaces:
\begin{itemize} 
    \item $WM(d,\{\overrightarrow{d^{(i)}}\}_{i=1,\cdots,n}, \{\overrightarrow{w^{(i)}}\}_{i=1,\cdots,n},\{\overrightarrow{\nu^{(i)}}\}_{i=1,\cdots,n})$: the moduli space  of $\{\overrightarrow{w^{(i)}}\}_{i=1,\cdots,n}$-stable  weakly parabolic logarithmic  flat bundles of type $(d, \{\overrightarrow{d^{(i)}}\}_{i=1,\cdots,n})$ and parabolic degree zero over $(C,D)$ with spectrum $\{\overrightarrow{\nu^{(i)}}\}_{i=1,\cdots,n}$;
    \item $WH(d,\{\overrightarrow{d^{(i)}}\}_{i=1,\cdots,n},\{\overrightarrow{w^{(i)}}\}_{i=1,\cdots,n},\{\overrightarrow{\mu^{(i)}}\}_{i=1,\cdots,n})$:  the  moduli space of  $\{\overrightarrow{w^{(i)}}\}_{i=1,\cdots,n}$-stable  weakly parabolic logarithmic  Higgs bundles  of type $(d, \{\overrightarrow{d^{(i)}}\}_{i=1,\cdots,n})$  and parabolic degree zero  over $(C,D)$ with the spectrum $\{\overrightarrow{\mu^{(i)}}\}_{i=1,\cdots,n}$;
    \item  $H(d,\{\overrightarrow{d^{(i)}}\}_{i=1,\cdots,n},\{\overrightarrow{w^{(i)}}\}_{i=1,\cdots,n},\{\overrightarrow{\mu^{(i)}}\}_{i=1,\cdots,n})$:  the  moduli space of  $\{\overrightarrow{w^{(i)}}\}_{i=1,\cdots,n}$-stable  parabolic logarithmic  Higgs bundles  of type $(d, \{\overrightarrow{d^{(i)}}\}_{i=1,\cdots,n})$  and parabolic degree zero  over $(C,D)$ with the spectrum $\{\overrightarrow{\mu^{(i)}}\}_{i=1,\cdots,n}$;
     \item $NH(d,\{\overrightarrow{d^{(i)}}\}_{i=1,\cdots,n},\{\overrightarrow{w^{(i)}}\}_{i=1,\cdots,n})$: the  moduli space of  $\{\overrightarrow{w^{(i)}}\}_{i=1,\cdots,n}$-stable nilpotent  parabolic logarithmic  Higgs bundles  of type $(d, \{\overrightarrow{d^{(i)}}\}_{i=1,\cdots,n})$  and parabolic degree zero  over $(C,D)$;
    \item $SH(d,\{\overrightarrow{d^{(i)}}\}_{i=1,\cdots,n},\{\overrightarrow{w^{(i)}}\}_{i=1,\cdots,n})$: the  moduli space of  $\{\overrightarrow{w^{(i)}}\}_{i=1,\cdots,n}$-stable strongly   parabolic logarithmic  Higgs bundles  of type $(d, \{\overrightarrow{d^{(i)}}\}_{i=1,\cdots,n})$  and parabolic degree zero  over $(C,D)$.
\end{itemize}

\begin{theorem}[\cite{y,y1,by}]
 The moduli space $WH(d,\{\overrightarrow{d^{(i)}}\}_{i=1,\cdots,n},\{\overrightarrow{w^{(i)}}\}_{i=1,\cdots,n},\{\overrightarrow{\mu^{(i)}}\}_{i=1,\cdots,n})$  with  $_\jmath\mu^{(i)}_1=\cdots={_\jmath\mu^{(i)}_{d^{(i)}_\jmath}}$ for any $i, \jmath$ is an irreducible normal quasi-projective variety of dimension $r^2(2g-2+n)+2-nr$,  and  the moduli space $H(d,\{\overrightarrow{d^{(i)}}\}_{i=1,\cdots,n},\{\overrightarrow{w^{(i)}}\}_{i=1,\cdots,n},\{\overrightarrow{\mu^{(i)}}\}_{i=1,\cdots,n})$ is an irreducible normal  subvariety of $WH(d,\{\overrightarrow{d^{(i)}}\}_{i=1,\cdots,n},\{\overrightarrow{w^{(i)}}\}_{i=1,\cdots,n},\{\overrightarrow{\mu^{(i)}}\}_{i=1,\cdots,n})$ with
dimension $r^2(2g-2+n)+2-\sum\limits_{i=1}^n\sum\limits_{\jmath=0}^{\ell^{(i)}}(d_\jmath^{(i)})^2$.
\end{theorem}

\begin{remark}
We also have the  moduli space of  $\{\overrightarrow{w^{(i)}}\}_{i=1,\cdots,n}$-stable  weakly parabolic logarithmic  Higgs bundles  of type $(d, \{\overrightarrow{d^{(i)}}\}_{i=1,\cdots,n})$  and parabolic degree zero  over $(C,D)$ without fixing the spectrum, which contains $WH(d,\{\overrightarrow{d^{(i)}}\}_{i=1,\cdots,n},\{\overrightarrow{w^{(i)}}\}_{i=1,\cdots,n},\{\overrightarrow{\nu^{(i)}}\}_{i=1,\cdots,n})$ as a subvariety. It is an irreducible normal quasi-projective variety of dimension $r^2(2g-2+n)+1$. However, this moduli space is not a suitable object regarding the nonabelian Hodge correspondence.
\end{remark}

The connection between parabolic logarithmic  Higgs bundles and parabolic logarithmic  flat bundles is established by the nonabelian Hodge correspondence (see Appendix \ref{F}). In particular, the nonabelian Hodge correspondence provides  diffeomorphisms between certain moduli spaces. Let us define
\begin{align*}
   {_\jmath v^{(i)}_a}:=w^{(i)}_\jmath-2\mathrm{Re}({_\jmath\mu^{(i)}_{a}})-{_\jmath n^{(i)}_a},\ i=1,\cdots,n, a=1,\dots, d^{(i)}_\jmath, \jmath=0,\cdots, \ell^{(i)},
\end{align*}where
\begin{align*}
  {_\jmath n^{(i)}_a}=[w^{(i)}_\jmath-2\mathrm{Re}({_\jmath\mu^{(i)}_{a}})],
\end{align*}
thus $0\leq  {_\jmath v^{(i)}_a}<1$.
Picking all distinct $ _\jmath v^{(i)}_a$'s and arranging them in ascending order to produce a vector $\overrightarrow{t^{(i)}}=(t^{(i)}_0,\cdots, t^{(i)}_{\wp^{(i)}})$,
we define
\begin{align*}
  S(t^{(i)}_p)&=\{(\jmath,a): {_\jmath v^{(i)}_a}=t^{(i)}_p\}, \ p=0,\cdots, \wp^{(i)}, \\
  f^{(i)}_p&=|S(t^{(i)}_p)|,\ \overrightarrow{f^{(i)}}=(f^{(i)}_0,\cdots,f^{(i)}_{\wp^{(i)}})\\
  t&=d+\sum_{i=1}^n\sum_{(\jmath,a)\in  S(t^{(i)}_p)}{_\jmath n^{(i)}_a},\\
 \{{_p\nu^{(i)}_1},\cdots, {_p\nu^{(i)}_{f^{(i)}_p}}\}
&= \{
w^{(i)}_\jmath+2\sqrt{-1}\mathrm{Im}({_\jmath\mu^{(i)}_a})-{_\jmath n^{(i)}_a}: (\jmath,a)\in  S(t^{(i)}_p)\},\\\overrightarrow{\nu^{(i)}}&=(_0\nu^{(i)}_1,\cdots, {_0\nu^{(i)}_{f^{(i)}_0}},\cdots, {_{\ell^{(i)}}\nu^{(i)}_1},\cdots, {_{\ell^{(i)}}\nu^{(i)}_{f^{(i)}_{\wp^{(i)}}}}).
\end{align*}
By virtue of  tame harmonic bundles, there is a map from 
$WH(d,\{\overrightarrow{d^{(i)}}\}_{i=1,\cdots,n},\{\overrightarrow{w^{(i)}}\}_{i=1,\cdots,n},\{\overrightarrow{\mu^{(i)}}\}_{i=1,\cdots,n})$ to $WM(t,\{\overrightarrow{f^{(i)}}\}_{i=1,\cdots,n}, \{\overrightarrow{t^{(i)}}\}_{i=1,\cdots,n},\{\overrightarrow{\nu^{(i)}}\}_{i=1,\cdots,n})$.

\begin{theorem}[\cite{s,t,t1,bq,sz}]Assume  $_\jmath\mu^{(i)}_1=\cdots={_\jmath\mu^{(i)}_{d^{(i)}_\jmath}}$ for any $i,\jmath$.
The nonabelian Hodge correspondence provides the following diffeomorphism of moduli spaces
\begin{align*}
    WH(d,\{\overrightarrow{d^{(i)}}\}_{i=1,\cdots,n},\{\overrightarrow{w^{(i)}}\}_{i=1,\cdots,n},\{\overrightarrow{\mu^{(i)}}\}_{i=1,\cdots,n})\simeq WM(t,\{\overrightarrow{f^{(i)}}\}_{i=1,\cdots,n}, \{\overrightarrow{t^{(i)}}\}_{i=1,\cdots,n},\{\overrightarrow{\nu^{(i)}}\}_{i=1,\cdots,n}),
\end{align*}
In particular, we have the  diffeomorphism 
of moduli spaces 
\begin{align*}
    NH(d,\{\overrightarrow{d^{(i)}}\}_{i=1,\cdots,n},\{\overrightarrow{w^{(i)}}\}_{i=1,\cdots,n})\simeq WM(d,\{\overrightarrow{d^{(i)}}\}_{i=1,\cdots,n}, \{\overrightarrow{w^{(i)}}\}_{i=1,\cdots,n},\{\overrightarrow{w^{(i)}}\}_{i=1,\cdots,n}),
\end{align*}
which restricts to give the  following diffeomorphism of moduli spaces 
\begin{align*}
    SH(d,\{\overrightarrow{d^{(i)}}\}_{i=1,\cdots,n},\{\overrightarrow{w^{(i)}}\}_{i=1,\cdots,n})\simeq M(d,\{\overrightarrow{d^{(i)}}\}_{i=1,\cdots,n}, \{\overrightarrow{w^{(i)}}\}_{i=1,\cdots,n},\{\overrightarrow{w^{(i)}}\}_{i=1,\cdots,n}).
\end{align*}
\end{theorem}

As an interpolation of flat connections and Higgs fields, Deligne introduced the so-called flat $\lambda$-connections, a vector bundle together with a flat $\lambda$-connection is called a $\lambda$-flat bundle \cite{HH22a,HH22b}. This notion also works in the context of parabolic bundles \cite{mk,t1,lm,m}. Then there is a $\mathbb{C}^\times$-action on the moduli space of parabolic logarithmic $\lambda$-flat bundles, as Simpson has done, one can consider the limit behavior of this action. The following proposition generalizes a result of Simpson \cite{s1}.

\begin{proposition}(\cite[Lemma 3.12]{k})
For any parabolic logarithmic flat bundle $(E,\nabla,\mathcal{L})\in M(d,\{\overrightarrow{d^{(i)}}\}_{i=1,\cdots,n},\{\overrightarrow{w^{(i)}}\}_{i=1,\cdots,n},\{\overrightarrow{\nu^{(i)}}\}_{i=1,\cdots,n})$,
the limit \begin{align*}
 \lim\limits_{c\rightarrow0}c\cdot (E,\mathcal{L},\nabla)= \lim\limits_{c\rightarrow0} (E,\mathcal{L},c\nabla)
\end{align*}
exists as a strongly  parabolic logarithmic Higgs bundle lying in the locus of fixed points of the  $\mathbb{C}^\times$-action on Higgs fields in the moduli space
$SH(d,\{\overrightarrow{d^{(i)}}\}_{i=1,\cdots,n},\{\overrightarrow{w^{(i)}}\}_{i=1,\cdots,n})$.
\end{proposition}

Let us return to the case of parabolic bundles of type $(d,\overrightarrow{d})$ over $(C,D)$ for $C\simeq \mathbb{P}^1$.

\begin{definition}\label{zs}
The weight system $\overrightarrow{w}$ on the parabolic bundle $(E,\mathcal{L})$  of type $(d,\overrightarrow{d})$ over $(C,D)$ is called
\begin{itemize}
  \item \emph{Kostov-generic}, if for any $\epsilon_i\in \{1,-1\}$, $i=1,\cdots, n$, we have
\begin{align*}
\frac {d+\sum\limits_{i=1}^n\epsilon_iw_i}{2}\notin\mathbb{Z};
\end{align*}
    \item \emph{non-resonant}, if
\begin{align*}
  0<w_i<1, \ \ i=1,\cdots,n;
\end{align*}
    \item \emph{non-special}, if it is both Kostov-generic and non-resonant.
\end{itemize}
\end{definition}
It is obvious that for a Kostov-generic weight system $\overrightarrow{w}$, the parabolic bundle   $(E, \mathcal{L})$    is $\overrightarrow{w}$-semistable iff it is $\overrightarrow{w}$-stable. We always assume the weight system $\overrightarrow{w}$ is non-special.
  A strongly parabolic logarithmic Higgs bundle
 $(E,\mathcal{L},\Theta)\in SH(d,\overrightarrow{d},\overrightarrow{w})$ is called a $\mathbb{C}^\times$-fixed point if for any constant $c\in\mathbb{C}^\times$, there is an automorphism $\sigma_c$ of $E$ such that
 \begin{align*}
   \left\{
     \begin{array}{ll}
       \sigma_c\cdot \mathcal{L}=\mathcal{L},\\
       \sigma_c\cdot\theta\cdot \sigma_c^{-1}=c\theta.
     \end{array}
   \right.
 \end{align*}
 Denote by $\mathrm{Fix}(d, \overrightarrow{d},\overrightarrow{w})$ the locus of $\mathbb{C}^\times$-fixed points lying in the moduli space
$SH(d,\overrightarrow{d},\overrightarrow{w})$.

\begin{proposition}\label{az} 
Let $(E,\mathcal{L},\theta)\in \mathrm{Fix}(d, \overrightarrow{d},\overrightarrow{w})$.
 \begin{enumerate}
   \item The logarithmic  Higgs field $\theta$ is non-zero iff the underlying parabolic bundle $(E,\mathcal{L})$ is not $\overrightarrow{w}$-stable.
   \item Let $w=\sum\limits_{i\in  \widehat{D_\infty}}w_i-\sum\limits_{i\in  D_\infty}w_i$, then the bundle $E$ takes one of the following forms
   \begin{itemize}
     \item $
    \mathcal{O}_{\mathbb{P}^1}(d-k-[\frac{d+w}{2}])\oplus\mathcal{O}_{\mathbb{P}^1}([\frac{d+w}{2}]+k)
 $
 for some integer $k\in (0, [\frac{d+n-2}{2}]-[\frac{d+w}{2}]]$,
  \item $\mathcal{O}_{\mathbb{P}^1}(d-k-[\frac{d-1}{2}])\oplus\mathcal{O}_{\mathbb{P}^1}([\frac{d-1}{2}]+k)
 $
 for  some  integer $k\in (0,[\frac{d+w}{2}]-[\frac{d-1}{2}]]$,
 \item $\mathcal{O}_{\mathbb{P}^1}(d-k-[\frac{d-1}{2}])\oplus\mathcal{O}_{\mathbb{P}^1}([\frac{d-1}{2}]+k)
 $
 for some integer $k\in (0, [\frac{d+n-2}{2}]-[\frac{d-1}{2}]]$.
   \end{itemize}
 \end{enumerate}

\end{proposition}
\begin{proof}(1) Assume $E\simeq\mathcal{O}_{\mathbb{P}^1}(d_0)\oplus \mathcal{O}_{\mathbb{P}^1}(d_1)$ with $d_0\leq d_1$,  and  the parabolic structure $\mathcal{L}=\{L_i\}_{i=1,\cdots,n}$ is parameterized by $(u_1,\cdots,u_n)\in (\mathbb{P}^1)^n$ with $u_i=[\kappa_i:\lambda_i]$. Since $(E,\mathcal{L},\theta)$ is a $\mathbb{C}^\times$-fixed point with nonzero $\theta$,  we must have
\begin{itemize}
  \item $\theta=\left(
            \begin{array}{cc}
              0 & \Theta \\
              0 & 0 \\
            \end{array}
          \right)$ or $\theta=\left(
            \begin{array}{cc}
              0 & 0 \\
              \Phi & 0 \\
            \end{array}
          \right)$
with  non-zero morphism $\Theta: \mathcal{O}_{\mathbb{P}^1}(d_1)\rightarrow \mathcal{O}_{\mathbb{P}^1}(d_0)\otimes \Omega^1_C(\mathcal{D})\simeq\mathcal{O}_{\mathbb{P}^1}(n-2-d_0)$ or $\Phi: \mathcal{O}_{\mathbb{P}^1}(d_0)\rightarrow \mathcal{O}_{\mathbb{P}^1}(d_1)\otimes \Omega^1_C(\mathcal{D})\simeq\mathcal{O}_{\mathbb{P}^1}(n-2-d_1)$ in terms of the decomposition of $E$,
  \item either $L_i\simeq\mathcal{O}_{\mathbb{P}^1}(d_0)|_{z_i}$ or $L_i\simeq\mathcal{O}_{\mathbb{P}^1}(d_1)|_{z_i}$.
\end{itemize}
Therefore,  the parabolic bundle
$(E,\mathcal{L})$ is decomposable, hence it is not $\overrightarrow{w}$-stable.

(2) There are two cases.
\begin{description}
  \item[Case I $\theta=0$] $(E,\mathcal{L})$ is  $\overrightarrow{w}$-stable, which implies  the inequality
\begin{align*}
 d_1-d_0<\sum_{i\in  \widehat{D_\infty}}w_i-\sum_{i\in  D_\infty}w_i=w.
\end{align*}
 On the other hand, $(E,\mathcal{L})$ must be indecomposable, hence it can be endowed with a logarithmic connection $\nabla$ such that $(E,\mathcal{L},\nabla)$ is a parabolic logarithmic flat bundle with Kostov-generic spectrum. Then  the  logarithmic flat bundle $(E,\nabla)$ is irreducible, which implies $ d_1-d_0\leq n-2$. Therefore, we have
 \begin{align*}
   0\leq d_1-d_0\leq \omega=\min\{n-2,w\},
 \end{align*}
 together with $\frac{d+w}{2}\notin \mathbb{Z}$ leads to
\begin{align*}
   E\simeq \mathcal{O}_{\mathbb{P}^1}(d-k-[\frac{d-1}{2}])\oplus\mathcal{O}_{\mathbb{P}^1}([\frac{d-1}{2}]+k)
 \end{align*}
 for  some  integer $k\in (0,[\frac{d+\omega}{2}]-[\frac{d-1}{2}]]$.
  \item[Case I $\theta \neq 0$] It is divided into two subcases.
  \begin{description}
  \item[Subcase I $\theta=\left(
            \begin{array}{cc}
              0 & \Theta \\
              0 & 0 \\
            \end{array}
          \right)$] $\Theta\neq 0$ means that $d_1-d_0\leq n-2$ and $\overrightarrow{w}$-stability of $(E,\mathcal{L},\theta)$ yields  the inequality
$d_1-d_0>w$. Therefore
\begin{itemize}
  \item if $w\geq 0$, we have
  \begin{align*}
   E\simeq \mathcal{O}_{\mathbb{P}^1}(d-k-[\frac{d+w}{2}])\oplus\mathcal{O}_{\mathbb{P}^1}([\frac{d+w}{2}]+k)
 \end{align*}
 for some integer $k\in (0, [\frac{d+n-2}{2}]-[\frac{d+w}{2}]]$,
  \item if $w<0$, we have
  \begin{align*}
   E\simeq \mathcal{O}_{\mathbb{P}^1}(d-k-[\frac{d-1}{2}])\oplus\mathcal{O}_{\mathbb{P}^1}([\frac{d-1}{2}]+k)
 \end{align*}
 for some integer $k\in (0, [\frac{d+n-2}{2}]-[\frac{d-1}{2}]]$.
\end{itemize}

  \item[Subcase II $\theta=\left(
            \begin{array}{cc}
              0 & 0 \\
              \Phi & 0 \\
            \end{array}
          \right)$] $\overrightarrow{w}$-stability of $(E,\mathcal{L},\theta)$ yields  the inequality
$d_1-d_0<w$, hence
\begin{align*}
   E\simeq \mathcal{O}_{\mathbb{P}^1}(d-k-[\frac{d-1}{2}])\oplus\mathcal{O}_{\mathbb{P}^1}([\frac{d-1}{2}]+k)
 \end{align*}
 for  some  integer $k\in (0,[\frac{d+w}{2}]-[\frac{d-1}{2}]]$.
\end{description}
\end{description}
Therefore, we complete the proof.
\end{proof}

\begin{corollary}\label{sss}
Assume $n=|D|=5$. If $(E,\mathcal{L},\theta)\in \mathrm{Fix}(1, \overrightarrow{d},\overrightarrow{w})$, we have
$E\simeq B:=\mathcal{O}_{\mathbb{P}^1}\oplus\mathcal{O}_{\mathbb{P}^1}(1)$ or $E\simeq B':=\mathcal{O}_{\mathbb{P}^1}(-1)\oplus\mathcal{O}_{\mathbb{P}^1}(2)$.
\end{corollary}

As an application of $\mathbb{C}^\times$-limit of parabolic logarithmic flat bundles, we have the following theorem, which describes the non-stable zones of weight system on certain parabolic bundles.

\begin{theorem}\label{ss}Assume $n=|D|=5$.
\begin{enumerate}
  \item Given a non-special weight system $\overrightarrow{w}$, suppose one of the following  conditions is satisfied
\begin{itemize}
  \item $\sum\limits_{i=1}^5w_i<1$;
   \item there is a subset $D_1\subset D$ with $|D_1|=2$ such that 
        \begin{align*}
             \sum_{z_i\in D_1}w_i-\sum_{z_i\in D\backslash{D_1}}w_i>1;
        \end{align*}
\end{itemize}
then for any parabolic structure $\mathcal{L}$ on $B$, the parabolic bundle $(B,\mathcal{L})$ cannot be $\overrightarrow{w}$-stable.
  \item Given a non-special weight system $\overrightarrow{w}$, suppose one of the above two conditions and the following condition are satisfied
  \begin{itemize}
  \item there is a point $z_j\in D$ such that
        \begin{align*}
             -w_j+ \sum_{z_i\in D\backslash\{z_j\}}w_i>3,
        \end{align*}
\end{itemize}
  then for any parabolic structure $\mathcal{L}$ on $B$ with $n_\infty=|D_\infty|\leq1$, the parabolic bundle $(B,\mathcal{L})$ cannot be $\overrightarrow{w}$-stable.
  \item Given a non-special weight $\overrightarrow{w}$, suppose one of the following two conditions is satisfied
\begin{itemize}
  \item $\sum\limits_{i=1}^5w_i<3$;
   \item there is a point $z_j\subset D$ such that 
          \begin{align*}
             -w_j+\sum_{z_i\in D\backslash\{z_j\}}w_i>3,
          \end{align*}
\end{itemize}
then for any parabolic structure $\mathcal{L}$ on $B'$, the parabolic bundle $(B',\mathcal{L})$ cannot be $\overrightarrow{w}$-stable.
\end{enumerate}

\end{theorem}
\begin{proof}
A $\overrightarrow{w}$-stable parabolic bundle $(B,\mathcal{L})$ admits a logarithmic connection $\nabla$ with the  Kostov-generic spectrum $\overrightarrow{\nu}$. Obviously, lying in the moduli spaces, we have the limit
\begin{align*}
   \lim\limits_{c\rightarrow0}c\cdot (B,\mathcal{L},\nabla)=(B,\mathcal{L},0)\in \mathrm{Fix}(1,\overrightarrow{d}, \overrightarrow{w}),
\end{align*}
where $0$ on the right hand side denotes the zero Higgs field.
The logarithmic flat bundle $(B,\nabla)$ is irreducible, hence
one decomposes  $\nabla$ as $\nabla=\left(
                                                                                            \begin{array}{cc}
                                                                                              \nabla_1 & \theta \\
                                                                                              \xi & \nabla_2 \\
                                                                                            \end{array}
                                                                                          \right)$ with non-zero morphisms $\theta: \mathcal{O}_{\mathbb{P}^1}(1)\rightarrow \mathcal{O}_{\mathbb{P}^1}\otimes \Omega^1_C(\mathcal{D})\simeq\mathcal{O}_{\mathbb{P}^1}(3)$ and  $\xi: \mathcal{O}_{\mathbb{P}^1}\rightarrow \mathcal{O}_{\mathbb{P}^1}(1)\otimes \Omega^1_C(\mathcal{D})\simeq\mathcal{O}_{\mathbb{P}^1}(4)$.

 i) $L_i$ parameterized by $u_i=[\kappa_i:\lambda_i]\in \mathbb{P}^1$ is the eigenspace of $\mathrm{Res}_{z_i}(\nabla)$ with eigenvalue $\nu^+_i$.  Firstly, we claim that   if $u_i=\infty$, then $\theta|_{z_i}=0$. Indeed, the logarithmic connection can be expressed locally as
  \begin{align*}
    \nabla=
   d+\sum_{i=1}^5\frac{1}{z-z_i}\left(
             \begin{array}{cc}
               A_i^{(11)} &  A_i^{(12)} \\
                A_i^{(21)} &  A_i^{(22)} \\
             \end{array}
           \right)dz,
  \end{align*}
    where  the constraints on the matrix components of $A^{(i)}$ are given by
   \begin{align*}
    \left\{
      \begin{array}{ll}
        \sum\limits_{i=1}^5A_i^{(11)}=0, \\
        \sum\limits_{i=1}^5A_i^{(22)}=-1, \\
        \sum\limits_{i=1}^5A^{(12)}_i=0,\\
        A_i^{(11)}+A_i^{(22)}=\nu^+_i+\nu^-_i,\\
         A_i^{(11)}A_i^{(22)}- A_i^{(12)}A_i^{(21)}=\nu_i^+\nu_i^-.
      \end{array}
    \right.
   \end{align*}
   Hence, the eigenspace of the matrix $\left(
             \begin{array}{cc}
               A_i^{(11)} &  A_i^{(12)} \\
                A_i^{(21)} &  A_i^{(22)} \\
             \end{array}
           \right)$ with eigenvalue $\nu_i^+$ is
           \begin{align*}
            \textrm{ either }\ \mathbb{C}\cdot\left(
                              \begin{array}{c}
                                 A_i^{(12)} \\
                                \nu_i^+- A_i^{(11)}  \\
                              \end{array}
                            \right),\ \textrm{ or }\  \mathbb{C}\cdot\left(
                              \begin{array}{c}
                                \nu_i^+- A_i^{(22)}  \\
                               A_i^{(21)}  \\
                              \end{array}
                            \right).
           \end{align*}
          It follows that if $u_i=\infty$,  we  must have $  A_i^{(12)}=0$, $A_i^{(11)}=\nu_i^-$ and $  A_i^{(22)}=\nu_i^+$, which means that $\theta|_{z_i}=0$. Similarly, if $u_i=0$,  then $\xi|_{z_i}=0$.

ii) Taking a family of automorphisms
\begin{align*}
 \sigma_c=\left(
       \begin{array}{cc}
         c^{-1} & 0 \\
         0 & 1 \\
       \end{array}
     \right)
\end{align*}
of $B$,
 we get the limit
\begin{align*}
 \lim\limits_{c\rightarrow0}\sigma_c\cdot( B,\mathcal{L}, c\nabla )= \mathbb{E}_0:=(B,\mathcal{L'},\left(
                                                                                            \begin{array}{cc}
                                                                                             0 & \theta \\
                                                                                             0 & 0 \\
                                                                                            \end{array}
                                                                                          \right) ),
\end{align*}
where the parabolic structure $\mathcal{L'}=\{L'_i\}_{i=1,\cdots,5}$ is given by
\begin{align*}
 L'_i=\left\{
        \begin{array}{ll}
          \mathcal{O}_{\mathbb{P}^1}(1)|_{z_i}, & \hbox{$u_i=\infty$;} \\
          \mathcal{O}_{\mathbb{P}^1}|_{z_i}, & \hbox{$u_i\neq\infty$.}
        \end{array}
      \right.
\end{align*}
If $\sum\limits_{i=1}^5w_i<1$, as a strongly parabolic logarithmic Higgs bundle, $\mathbb{E}_0$ is  $\overrightarrow{w}$-stable, thus $\mathbb{E}_0$ lies in $\mathrm{Fix}(1,\overrightarrow{d}, \overrightarrow{w})$. This contradiction means that when $\sum\limits_{i=1}^5w_i<1$, any parabolic bundle $(B,\mathcal{L})$ cannot be $\overrightarrow{w}$-stable.

iii) Similarly, picking a family of automorphisms
\begin{align*}
  \sigma'_c=\left(
       \begin{array}{cc}
         1 & 0 \\
         0 & c^{-1} \\
       \end{array}
     \right)
\end{align*}
of $B$,
  we have the limit
\begin{align*}
 \lim\limits_{c\rightarrow0}\sigma'_c\cdot (B,\mathcal{L},  c\nabla)= \mathbb{E}_1:=(B,\mathcal{L''},\left(
                                                                                            \begin{array}{cc}
                                                                                             0 & 0 \\
                                                                                             \xi & 0 \\
                                                                                            \end{array}
                                                                                          \right) ),
\end{align*}
where the parabolic structure $\mathcal{L''}=\{L''_i\}_{i=1,\cdots,5}$ is given by
\begin{align*}
 L''_i=\left\{
        \begin{array}{ll}
          \mathcal{O}_{\mathbb{P}^1}(1)|_{z_i}, & \hbox{$u_i\neq0$;} \\
          \mathcal{O}_{\mathbb{P}^1}|_{z_i}, & \hbox{$u_i=0$.}
        \end{array}
      \right.
\end{align*}
 Let $D_0=\{z_i\in D:L_i=\mathcal{O}_{\mathbb{P}^1}|_{z_i}\}$. We can assume $|D_0|\geq2, |D_\infty|\leq 2$.
We check the second condition in (1)  with the  following three cases.
 \begin{description}
 \item[Case I  $D_1\bigcap  D_\infty=\emptyset$] we can always suppose $D_1\subseteq D_0$, then
\begin{align*}
  \sum_{z_i\in D_0}w_i-\sum_{z_i\in D\backslash{D_0}}w_i\geq \sum_{z_i\in D_1}w_i-\sum_{z_i\in D\backslash{D_1}}w_i>1.
\end{align*}
Therefore, $\mathbb{E}_1$ is a $\overrightarrow{w}$-stable strongly parabolic logarithmic Higgs bundle, which is  a contradiction.
  \item[Case II $D_1= D_\infty$] Since
\begin{align*}
  \sum_{z_i\in D_\infty}w_i-\sum_{z_i\in \widehat{D_\infty}}w_i>1>-1,
\end{align*}
 $\mathbb{E}_0$ turns out to be $\overrightarrow{w}$-stable, which is  a contradiction.
  \item[Case III $|D_1\bigcap  D_\infty|=1$]Assume $D_1\backslash(D_1\bigcap  D_\infty)=z_k$, then
\begin{align*}
  \sum_{z_i\in D_\infty}w_i-\sum_{z_i\in \widehat{D_\infty}}w_i\geq \sum_{z_i\in D_1}w_i-\sum_{z_i\in D\backslash{D_1}}w_i-2w_k>1-2w_k>-1,
\end{align*}
thus $\mathbb{E}_0$ is still $\overrightarrow{w}$-stable, which is  a contradiction.
\end{description}

iv)
To check  the third  condition in (2), we consider the short exact sequence
\begin{align*}
  0\rightarrow \mathcal{O}_{\mathbb{P}^1}(-1)\xrightarrow{\alpha} B\rightarrow  \mathcal{O}_{\mathbb{P}^1}(2)\rightarrow0,
\end{align*}
where the morphism $\alpha:\mathcal{O}_{\mathbb{P}^1}(-1)\hookrightarrow B$ described by $(q_0+q_1z, r_0+r_1z+r_2z^2)$ with $(q_0,q_1)\neq(0,0), (r_0,r_1,r_2)\neq(0,0,0)$. By a chosen $C^\infty$-splitting $\beta:\mathcal{O}_{\mathbb{P}^1}(2)\rightarrow B$, one decomposes the holomorphic structure $\bar\partial_B$ on $B$ as 
$\bar\partial_B=\left(
            \begin{array}{cc}
                 \bar\partial_{1} & \zeta\\
                 0& \bar\partial_{2}  \\
            \end{array}
                 \right)
$ and decomposes  $\nabla$ as 
$\nabla=\left(
            \begin{array}{cc}
                 \nabla'_1 & \theta' \\
                 \xi' & \nabla'_2 \\
            \end{array}
        \right)$. Then the holomorphicity of $\nabla$, i.e. $\bar\partial_B\circ\nabla=\nabla\circ\bar\partial_B$, gives rise to the relations 
        \begin{align*}
            \left\{
               \begin{array}{ll}
                   \bar\partial_{1}\circ\nabla'_1-\nabla'_1\circ\bar\partial_{1}+\zeta\circ\xi'=0, \\ \bar\partial_{2}\circ\nabla'_2-\nabla'_2\circ\bar\partial_{2}+\xi'\circ\zeta=0, \\ \bar\partial_{1}\circ\theta' -\theta'\circ\bar\partial_{2} +\zeta\circ\nabla'_2-\nabla'_1\circ\zeta=0,\\ \bar\partial_{2}\circ\xi'-\xi'\circ\bar\partial_{1}=0.
                \end{array}
             \right.
        \end{align*}This implies the holomorphic morphism $\xi':\mathcal{O}_{\mathbb{P}^1}(-1)\rightarrow \mathcal{O}_{\mathbb{P}^1}(2)\otimes \Omega^1_C(\mathcal{D})\simeq\mathcal{O}_{\mathbb{P}^1}(5)$ is non-zero, otherwise the logarithmic flat bundle  $(E,\nabla)$ is not irreducible.
There is a family of $C^\infty$-automorphisms
\begin{align*}
 \sigma''_c=\left(
                                        \begin{array}{cc}
                                          c & c \\
                                          0 & 1 \\
                                        \end{array}
                                      \right)
\end{align*}
of $B$ satisfying $\bar\partial_B\sigma''_c=0$ such that
\begin{align*}
 \lim\limits_{c\rightarrow0}\sigma''_c\cdot(\cdot B, \mathcal{L}, \nabla)= \mathbb{E}^{(\alpha)}_{2}:=(B',\mathcal{L'''},\left(
                         \begin{array}{cc}
                              0 & 0 \\
                              \xi' & 0 \\
                          \end{array}
                    \right) ),
\end{align*}
where the parabolic structure $\mathcal{L'''}=\{L'''_i\}_{i=1,\cdots,5}$ is given as follows,  \begin{align*}
 L'''_i=\left\{
        \begin{array}{ll}
          \mathcal{O}_{\mathbb{P}^1}(-1)|_{z_i}, & \hbox{$[q_0+q_1z_i:r_0+r_1z_i+r_2z_i^2]=u_i$;} \\
          \mathcal{O}_{\mathbb{P}^1}(2)|_{z_i}, & \hbox{$[q_0+q_1z_i:r_0+r_1z_i+r_2z_i^2]\neq u_i$.}
        \end{array}
      \right.
\end{align*}Let $D^{(\alpha)}_2=\{z_i\in D:L'''_i=\mathcal{O}_{\mathbb{P}^1}(2)|_{z_i}\}$. Since $n_\infty\leq 1$,
one can always find $\alpha$  such that $D^{(\alpha)}_2=\{z_j\}$ for given point $z_j\in D$,  then $\mathbb{E}^{(\alpha)}_{2}$ is a $\overrightarrow{w}$-stable strongly parabolic logarithmic Higgs bundle.

v) For $(B',\mathcal{L})$, the arguments are similar.
\end{proof}

\section{Simpson's  Conjectures: Foliation and Stratification on Moduli Spaces}

The locus of $\mathbb{C}^\times$-fixed points  lying in the moduli space
${SH}(d,\{\overrightarrow{d^{(i)}}\}_{i=1,\cdots,n},\{\overrightarrow{w^{(i)}}\}_{i=1,\cdots,n})$  is denoted by $ \mathrm{Fix}(d,\{\overrightarrow{d^{(i)}}\}_{i=1,\cdots,n},\{\overrightarrow{w^{(i)}}\}_{i=1,\cdots,n})$, which is divided into the union of connected components as
\begin{align*}
 \mathrm{Fix}(d,\{\overrightarrow{d^{(i)}}\}_{i=1,\cdots,n},\{\overrightarrow{w^{(i)}}\}_{i=1,\cdots,n})=\coprod\limits_{\alpha}\mathrm{Fix}_\alpha(d,\{\overrightarrow{d^{(i)}}\}_{i=1,\cdots,n},\{\overrightarrow{w^{(i)}}\}_{i=1,\cdots,n}),
\end{align*}
 and one defines
\begin{align*}
  M_\alpha(d,\{\overrightarrow{d^{(i)}}\}_{i=1,\cdots,n},\{\overrightarrow{w^{(i)}}\}_{i=1,\cdots,n},\{\overrightarrow{\nu^{(i)}}\}_{i=1,\cdots,n})
  =&\ \{(E,\mathcal{L},\nabla)\in M(d,\{\overrightarrow{d^{(i)}}\}_{i=1,\cdots,n},\{\overrightarrow{w^{(i)}}\}_{i=1,\cdots,n},\{\overrightarrow{\nu^{(i)}}\}_{i=1,\cdots,n})\\
  &\ :\lim\limits_{c\rightarrow0}c\cdot (E,\mathcal{L},\nabla)\in \mathrm{Fix}_\alpha(d,\{\overrightarrow{d^{(i)}}\}_{i=1,\cdots,n},\{\overrightarrow{w^{(i)}}\}_{i=1,\cdots,n})\}.
\end{align*}
 Then we have the disjoint union
  \begin{align*}
  M(d,\{\overrightarrow{d^{(i)}}\}_{i=1,\cdots,n},\{\overrightarrow{w^{(i)}}\}_{i=1,\cdots,n},\{\overrightarrow{\nu^{(i)}}\}_{i=1,\cdots,n})=\coprod\limits_\alpha M_\alpha(d,\{\overrightarrow{d^{(i)}}\}_{i=1,\cdots,n},\{\overrightarrow{w^{(i)}}\}_{i=1,\cdots,n},\{\overrightarrow{\nu^{(i)}}\}_{i=1,\cdots,n}),
  \end{align*} and taking the zero-limit of the $\mathbb{C}^\times$-action provides the morphisms
\begin{align*}
 \Psi_\alpha:M_\alpha(d,\{\overrightarrow{d^{(i)}}\}_{i=1,\cdots,n},\{\overrightarrow{w^{(i)}}\}_{i=1,\cdots,n},\{\overrightarrow{\nu^{(i)}}\}_{i=1,\cdots,n})&\rightarrow \mathrm{Fix}_\alpha(d,\{\overrightarrow{d^{(i)}}\}_{i=1,\cdots,n},\{\overrightarrow{w^{(i)}}\}_{i=1,\cdots,n}),\\
\Psi=\coprod_\alpha \Psi_\alpha:M(d,\{\overrightarrow{d^{(i)}}\}_{i=1,\cdots,n},\{\overrightarrow{w^{(i)}}\}_{i=1,\cdots,n},\{\overrightarrow{\nu^{(i)}}\}_{i=1,\cdots,n})&\rightarrow \mathrm{Fix}(d,\{\overrightarrow{d^{(i)}}\}_{i=1,\cdots,n},\{\overrightarrow{w^{(i)}}\}_{i=1,\cdots,n})\\
 (E,\mathcal{L},\nabla)&\mapsto \lim\limits_{c\rightarrow0}c\cdot (E,\mathcal{L},\nabla).
\end{align*}

The parabolic version of Simpson's conjecture concerning the foliation and stratification on the de Rham moduli space \cite{si} is the following.

\begin{conjecture}\label{c}
There exists a weight system $\{\overrightarrow{w^{(i)}}\}_{i=1,\cdots,n}$ and a coarse granulation of the index set $A=\{\alpha\}$
as
\begin{align*}
  A=\coprod_{p\in I}A^p
\end{align*}
for pairwise-disjoint non-empty subsets $A^p$'s of $A$ such that the following properties hold.
\begin{enumerate}
\item (foliation property) The morphism  $\Psi$ is surjective, and moreover, all  the fibers  of $\Psi$  fit together into a
regular foliation with closed leaves on $M(d,\{\overrightarrow{d^{(i)}}\}_{i=1,\cdots,n},\{\overrightarrow{w^{(i)}}\}_{i=1,\cdots,n},\{\overrightarrow{\nu^{(i)}}\}_{i=1,\cdots,n})$.
  \item (stratification property)
 There is a partial order on the index set $I$ such that
\begin{align*}
  &\overline{ M^p(d,\{\overrightarrow{d^{(i)}}\}_{i=1,\cdots,n},\{\overrightarrow{w^{(i)}}\}_{i=1,\cdots,n},\{\overrightarrow{\nu^{(i)}}\}_{i=1,\cdots,n})}\backslash  M^p(d,\{\overrightarrow{d^{(i)}}\}_{i=1,\cdots,n},\{\overrightarrow{w^{(i)}}\}_{i=1,\cdots,n},\{\overrightarrow{\nu^{(i)}}\}_{i=1,\cdots,n})\\
= &\ \coprod_{q<p}  M^q(d,\{\overrightarrow{d^{(i)}}\}_{i=1,\cdots,n},\{\overrightarrow{w^{(i)}}\}_{i=1,\cdots,n},\{\overrightarrow{\nu^{(i)}}\}_{i=1,\cdots,n}),
\end{align*}
where
\begin{align*}
  M^p(d,\{\overrightarrow{d^{(i)}}\}_{i=1,\cdots,n},\{\overrightarrow{w^{(i)}}\}_{i=1,\cdots,n},\{\overrightarrow{\nu^{(i)}}\}_{i=1,\cdots,n})=\coprod_{\alpha\in A^p}M_\alpha(d,\{\overrightarrow{d^{(i)}}\}_{i=1,\cdots,n},\{\overrightarrow{w^{(i)}}\}_{i=1,\cdots,n},\{\overrightarrow{\nu^{(i)}}\}_{i=1,\cdots,n}),
\end{align*}
and $\overline{\bullet}$ stands for the closure of $\bullet$.
\end{enumerate}
\end{conjecture}

In this section, we focus on  Simpson's conjecture for the moduli space $M(1,\overrightarrow{d},\overrightarrow{\nu})$ with the fixed non-special spectrum $\overrightarrow{\nu}$ when $C\simeq\mathbb{P}^1$, $n=|D|=5$. Given different non-special weight systems $\overrightarrow{w}$, there are different morphisms
\begin{align*}
 \Psi_{\overrightarrow{w}}: M(1,\overrightarrow{d},\overrightarrow{\nu})\rightarrow \mathrm{Fix}(1,\overrightarrow{d},\overrightarrow{w}).
\end{align*}
by taking the zero-limit of the $\mathbb{C}^\times$-action. We will only study the weight system $\overrightarrow{w}$ with $\sum\limits_{i=1}^5w_1<1$.
For such choice, $\mathrm{Fix}(1,\overrightarrow{d},\overrightarrow{w})$ consists of two components.

\begin{proposition}\label{mn}
$\mathrm{Fix}(1,\overrightarrow{d},\overrightarrow{w})=\mathrm{Fix}^0\coprod \mathrm{Fix}^1$, where

\begin{itemize}
  \item $\mathrm{Fix}^0$ consists of a single point  that represents a strongly  parabolic logarithmic Higgs bundle
\begin{align*}
 \mathbb{E}'= (B',\mathcal{L},\theta=\left(
                                                          \begin{array}{cc}
                                                            0 & 1 \\
                                                            0 & 0 \\
                                                          \end{array}
                                                        \right)
),
\end{align*} with $L_i=\mathcal{O}(-1)|_{z_i}$,
  \item  $\mathrm{Fix}^1$ consists of points representing strongly parabolic logarithmic Higgs bundles
\begin{align*}
  (B,\mathcal{L},\theta=\left(
                                                          \begin{array}{cc}
                                                            0 & \Theta \\
                                                            0 & 0 \\
                                                          \end{array}
                                                        \right)
),
\end{align*}
with
 \begin{align*}
   L_i=\left\{
         \begin{array}{ll}
          \mathcal{O}_{\mathbb{P}^1}|_{z_i} , & \hbox{$\theta|_{z_i}\neq 0$;} \\
       \mathcal{O}_{\mathbb{P}^1}|_{z_i}\ \text{or}\ \mathcal{O}_{\mathbb{P}^1}(1)|_{z_i}   , & \hbox{$\theta|_{z_i}= 0$,}
         \end{array}
       \right.
 \end{align*}
and nonzero morphisms  $\Theta: \mathcal{O}_{\mathbb{P}^1}(1)\rightarrow \mathcal{O}_{\mathbb{P}^1}\otimes \Omega^1_C(\mathcal{D})\simeq \mathcal{O}_{\mathbb{P}^1}(3)$    parameterized by $(H^0(\mathbb{P}^1, \mathcal{O}_{\mathbb{P}^1}(2))-\{0\})/\mathbb{C}^\times\simeq \mathbb{P}^2$.\end{itemize}

\end{proposition}

\begin{proof}
This proposition follows from Corollary \ref{sss}, and the stability condition with respect to the weight system $\overrightarrow{w}$ satisfying $\sum\limits_{i=1}^5w_1<1$.
\end{proof}

We confirm Simpson's conjecture for our case in the following steps.

\vspace{5pt}

\textbf{\emph{Step 1: $\mathrm{Fix}^1$  is set-theoretically realized as a non-separated scheme.}}
\vspace{5pt}

It can be achieved  via a suitable gluing of two copies since the parabolic structure at the point lying in $D$ has two choices if it is the zero of Higgs field.
Let $\tau: \mathbb{P}^1\times \mathbb{P}^1\rightarrow \mathbb{P}^2$ be the natural two-fold ramified covering. Identifying $\mathbb{P}^1$ with $C$,  for a pair $(z, z')\in \mathbb{P}^1\times \mathbb{P}^1$, the image  $\tau(z,z') $ determines a Higgs field $\theta$ by treating $z, z'$ as   zero points of $\theta$.   Lying in $\mathbb{P}^2$, we have
\begin{itemize}
  \item 15 special points: $S_1=\{\tau_{ij}\}_{1\leq i<j\leq 5}$ and $S_2=\{\tau_{ii}\}_{i=1,\cdots,5}$, where $\tau_{ij}=\tau(z_i,z_j)$ for a pair $(z_i,z_j)\in D\times D$;
  \item five special lines:  $T=\{ \tau_{i}\}_{i=1,\cdots,5}$, where $\tau_{i}$ is defined by
  \begin{align*}
    \tau_{i}:&\ \mathbb{P}^1\hookrightarrow \mathbb{P}^1\times  \mathbb{P}^1\rightarrow  \mathbb{P}^2\\
    &z\mapsto(z_i,z)\mapsto \tau(z_i,z).
  \end{align*}
\end{itemize}
 Blowing up $\mathbb{P}^2$ at $S_2$ with five exceptional curves $Q=\{O_i\}_{i=1,\cdots,5}$, we get a rational  surface   denoted by $\mathcal{B}$.
 Let $\tilde \tau_i$ be the strict transformation of $\tau_i$, and denote by $I_i$ the intersection point of $\tilde \tau_i$ and $O_i$. The points $z_j,j\neq i\in\{1,\cdots,5\}$, lying in  $O_i$ are denoted by $O_i^j$.
 The points in $S_1$ are also viewed  as those in $\mathcal{B}$. Let $ \underline{\mathcal{B}}$ be a copy of $\mathcal{B}$, and $\underline{\bullet}$ be the corresponding one in $ \underline{\mathcal{B}}$ of the object $\bullet$ in  $\mathcal{B}$. Now we  glue $\mathcal{B}$ and $\underline{\mathcal{B}}$ by the following identifications
 \begin{itemize}
   \item $\mathcal{B}\backslash(T\bigcup Q)\simeq \underline{\mathcal{B}}\backslash(\underline{T}\bigcup \underline{Q})$;
   \item $\tilde \tau_i\backslash I_i\simeq \underline{{O}_i}\backslash \underline{I_i}, i=1,\cdots,5$;
   \item $O_i\backslash(\{I_i\}\bigcup\{O_i^j\}_{j\neq i})\simeq \underline{\tilde\tau_i}\backslash(\{\underline{I_i}\}\bigcup\{\underline{\tau_{ij}}\}_{j>i}\bigcup\{\underline{\tau_{ki}}\}_{k<i} ), i=1,\cdots,5$,
 \end{itemize}
 then the resulting non-separated scheme is denoted by $\mathfrak{F}$.  The above gluing process is illuminated as follows (we exhibit only two marked points).
 \begin{figure}[H] 
\centering 
\includegraphics[width=0.5\textwidth]{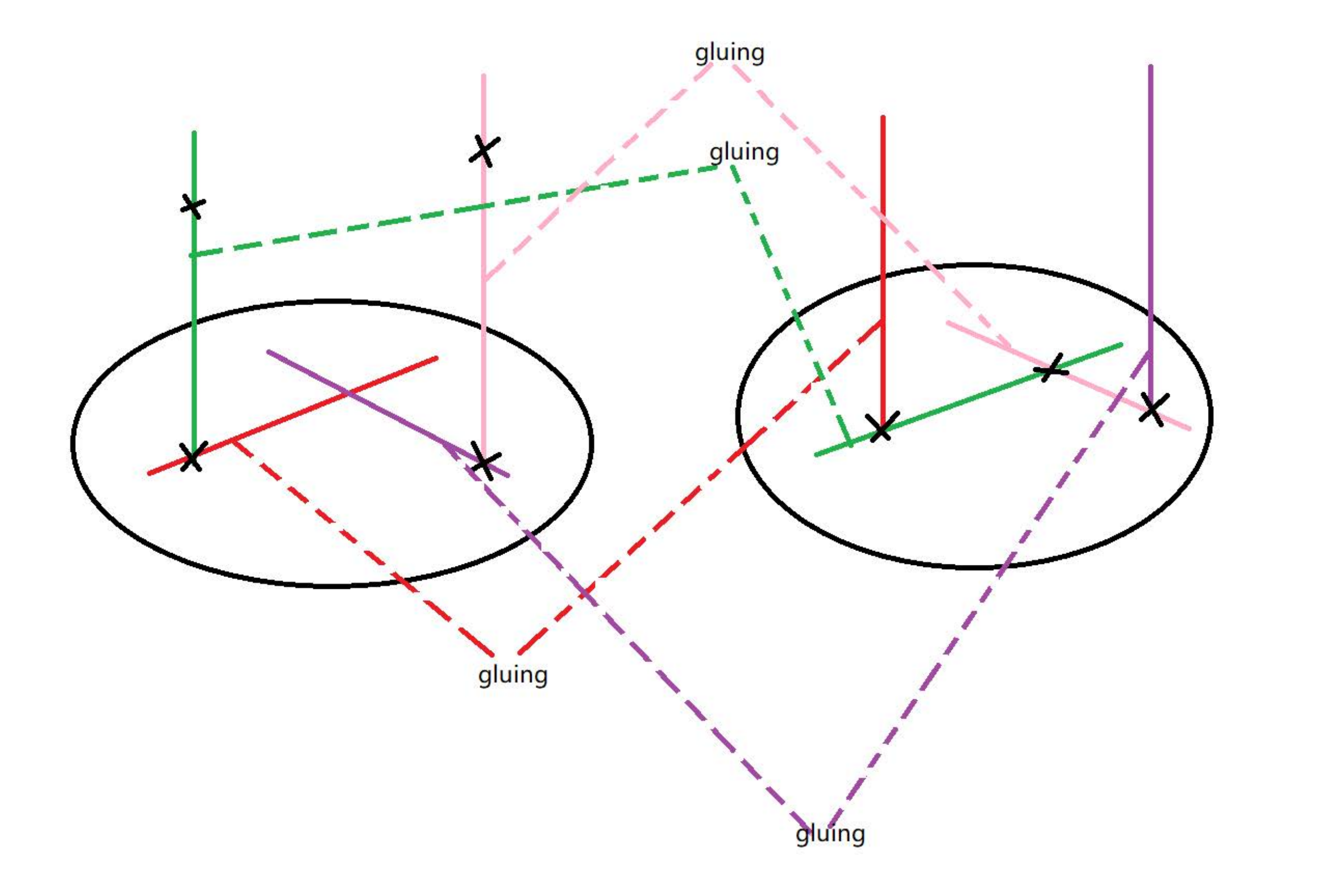} 
\end{figure}

By the above discussions and Theorem \ref{q}, we immediately obtain the following theorem, which is an analog of \cite[Lemma 6.1]{lm} and \cite[Theorem 5.6]{lm}.

\begin{theorem}\label{bv}\
\begin{enumerate}
\item Let $N^{\mathrm{ind}}_B(\overrightarrow{d})$ be the moduli stack of indecomposable parabolic structures with dimension system $\overrightarrow{d}$ on $B$, then $N^{\mathrm{ind}}_B(\overrightarrow{d})$ is a $\mathbb{C}^\times$-gerbe over $\mathfrak{F}$.
  \item Denote by $\Gamma_i$ the image of ${\tilde\tau_i}\backslash(\{{I_i}\}\bigcup\{{\tau_{ij}}\}_{j>i}\bigcup\{{\tau_{ki}}\}_{k<i} )$ (or $\underline{O_i}\backslash(\{\underline{I_i}\}\bigcup \{\underline{O_i^j}\}_{j\neq i})$) in $\mathfrak{F}$, and define $\mathfrak{F}'=\mathfrak{F}\backslash(\{\Gamma_i\}_{i=1,\cdots,5})$, then there is a set-theoretical isomorphism between $\mathrm{Fix}^1$ and  $\mathfrak{F}'$, which is constructibly algebraic.
  \end{enumerate}

\end{theorem}

\begin{remark}
 Inspired by Theorem \ref{bv}, we propose the following conjecture for $C$ being  a general compact Riemann surface.

\begin{conjecture} 
If $\mathrm{Fix}_\alpha(d,\{\overrightarrow{d^{(i)}}\}_{i=1,\cdots,n},\{\overrightarrow{w^{(i)}}\}_{i=1,\cdots,n})$ is of maximal dimension among the connected components, then it is an open subscheme of $M^{\textrm{sim}}(d,\{\overrightarrow{d^{(i)}}\}_{i=1,\cdots,n})$, where $M^{\textrm{sim}}(d,\{\overrightarrow{d^{(i)}}\}_{i=1,\cdots,n})$ is the moduli space of simple parabolic bundles of degree $d$ and dimension system $\{\overrightarrow{d^{(i)}}\}_{i=1,\cdots,n}$.
\end{conjecture}

 When $C$ is a compact Riemann surface of genus $g\geq 2$, $D=\emptyset$, this conjecture is implied by \cite[Corollary 4.6]{HH22a} and  \cite[Theorem 3.2]{lo}, and when $C\simeq \mathbb{P}^1, n=|D|=4,5$, this conjecture has been   checked for some cases. A similar result that relates stable Higgs bundles and indecomposable bundles in positive characteristic is as follows. The number of geometrically indecomposable vector bundles of rank $r$ and degree $d$
 over  $C$ defined over a finite field can be computed by the number of points of the moduli space of stable Higgs bundles of rank $r$ and degree $d$ over $C$ \cite{sh}.
\end{remark}

\textbf{\emph{Step 2:  $\Psi^{-1}_{\overrightarrow{w}}(\mathrm{Fix}^1)$  is made into a  regular foliation by the fibration over $\mathrm{Fix}^1$.}}

\begin{theorem}\label{bvv}
$\Psi^{-1}_{\overrightarrow{w}}(\mathrm{Fix}^1)$ is a fibration over $\mathrm{Fix}^1$ with fibers of dimension 2 and has a complement of  codimension 2 in $M(1,\overrightarrow{d},\overrightarrow{\nu})$.
\end{theorem}
\begin{proof}We  need to show $\Psi_{\overrightarrow{w}}$ is surjective and $\dim_\mathbb{C}(\Psi^{-1}_{\overrightarrow{w}}(\mathbb{E}))=2$ for a given $\mathbb{E}=(E,\mathcal{L},\theta
)\in \mathrm{Fix}(1,\overrightarrow{d},\overrightarrow{w})$. By proposition \ref{mn}, there are two cases.
\begin{description}
  \item[Case I $\mathbb{E}\in \mathrm{Fix}^1$] It follows  from the proof of Theorem \ref{ss} that  any  $(E',\mathcal{L}',\nabla)\in\Psi^{-1}_{\overrightarrow{w}}(\mathbb{E})$ satisfies
\begin{itemize}
  \item $E'\simeq B$;
  \item $L_i'=\left\{
                \begin{array}{ll}
                  u_i=[\kappa_i:\lambda_i]\neq \infty, & \hbox{$L_i=\mathcal{O}_{\mathbb{P}^1}|_{z_i}$,} \\
                  \mathcal{O}_{\mathbb{P}^1}(1)|_{z_i}, & \hbox{$L_i=\mathcal{O}_{\mathbb{P}^1}(1)|_{z_i}$ (\textrm{hence} $\Theta|_{z_i}=0$);}
                \end{array}
              \right.
$
  \item $ \nabla=
   d+\sum\limits_{i=1}^5\frac{A_i}{z-z_i}dz$,
           where \begin{align*}
 A_i=\left\{
       \begin{array}{ll}
          \left(
             \begin{array}{cc}
               \nu_i^+- A_i^{(12)}u_i &  A_i^{(12)} \\
               (\nu_i^+-\nu_i^-)u_i-A_i^{(12)}u_i^2  & \nu_i^-+ A_i^{(12)}u_i \\
             \end{array}
           \right), & \hbox{$L_i=\mathcal{O}_{\mathbb{P}^1}|_{z_i}$,} \\
        \left(
             \begin{array}{cc}
               \nu_i^-&  0 \\
               A_i^{(21)}  & \nu_i^+  \\
             \end{array}
           \right), & \hbox{$L_i=\mathcal{O}_{\mathbb{P}^1}(1)|_{z_i}$.}
       \end{array}
     \right.
\end{align*}
\end{itemize}
Note that $A^{(12)}_i,i=1,\cdots,5$, are totally  determined by $\theta$, and there is constraint on $\{u_i\}_{z_i\in D'}$ given by
\begin{align*}
 \sum_{z_i\in D'} (\nu_i^+- A_i^{(12)}u_i)+ \sum_{z_i\in D\backslash D'}\nu_j^-=0,
\end{align*}
where $D'=\{z_i\in D: L_i=\mathcal{O}_{\mathbb{P}^1}|_{z_i}\}$ with $|D'|\geq 3$.
 Therefore, the dimension of $\Psi^{-1}_{\overrightarrow{w}}(\mathbb{E})$ for $\mathbb{E}\in \mathbf{F}_1(1, \overrightarrow{w})$
is calculated as
\begin{align*}
  \dim_\mathbb{C}(\Psi^{-1}_{\overrightarrow{w}}(\mathbb{E}))=5-1-[\dim_\mathbb{C}(\Aut(B))-1]=2.
\end{align*}
  \item[ Case II $\mathbb{E}\in \mathrm{Fix}^0$, i.e. $\mathbb{E}=\mathbb{E}'$] In this case, any  $(E',\mathcal{L}',\nabla)\in\Psi^{-1}_{\overrightarrow{w}}(\mathbb{E})$ satisfies
\begin{itemize}
  \item $E'\simeq B'$;
  \item $L_i'=u_i=[\kappa_i:\lambda_i]$ with $u_i\neq \infty$;
  \item $ \nabla=
   d+\sum\limits_{i=1}^5\frac{1}{z-z_i}\left(
             \begin{array}{cc}
               A_i^{(11)} &  A_i^{(12)} \\
                A_i^{(21)} &  A_i^{(22)} \\
             \end{array}
           \right)dz+\left(
             \begin{array}{cc}
              0 &  0 \\
                G^{(21)}(z) &  0 \\
             \end{array}
           \right)dz$ with  a polynomial $G^{(21)}(z)$ of degree of not great than 1.
           \end{itemize}
           We have seen that $A_i$'s are totally determined by the parabolic structure $\mathcal{L}'$ on $B'$. Since the moduli space of indecomposable parabolic structures on  $B'$ is a single point,  $\Psi^{-1}_{\overrightarrow{w}}(\mathbb{E})$ is also of two dimensions by counting the degrees of freedom of $G^{(21)}(z)$.
\end{description}
We complete the proof.
\end{proof}

\textbf{\emph{Step 3: $M(1,\overrightarrow{d},\overrightarrow{\nu})$ is the closure of $\Psi^{-1}_{\overrightarrow{w}}(\mathrm{Fix}^1)$.}}

\begin{theorem}\label{bvvv}$\overline{ \Psi^{-1}_{\overrightarrow{w}}(\mathrm{Fix}^1)}= M(1,\overrightarrow{d},\overrightarrow{\nu})$.
\end{theorem}
\begin{proof}The  proof of this theorem follows the approach due to Simpson \cite{s1}, and the new ingredient is replacing the usual  logarithmic  de Rham complex by middle  logarithmic  de Rham complex.

Associated to a given parabolic logarithmic flat bundle  $(E\simeq B',\mathcal{L}, \nabla)$ with the  non-special spectrum $\overrightarrow{\nu}$, there is a middle logarithmic  de Rham complex $\mathbf{MDR}_{(E,\nabla)}$ defined as follows
\begin{align*}
\begin{array}{ccccccccc}
   &  &0  \\
   &  &\downarrow \\
\mathbf{MDR}_{(E,\mathcal{L},\nabla)}:\mathrm{End}(E) &\xrightarrow{\nabla}&\mathrm{MDR}^1_{(E,\mathcal{L},\nabla)}  \\
&  &\downarrow \\
\ \ \ \ \mathbf{DR}_{(E,\nabla)}: \mathrm{End}(E) &\xrightarrow{\nabla}&\mathrm{End}(E)\otimes \Omega^1_C(\mathcal{D})  \\
&  &\downarrow \\
&  &\mathrm{End}^0(E)|_D\\
&  &\downarrow \\
 &  &0
  \end{array},
 \end{align*}
 where $\mathrm{End}^0(E)|_D=\bigcup\limits_{z_i\in D}\mathrm{End}^0(E)|_{z_i}$ for the zero eigen-subspace $\mathrm{End}^0(E)|_{z_i}$ of $\mathrm{End}(E)|_{z_i}$ with respect to $\mathrm{Res}_{z_i}(\nabla)$ (cf. \cite{si}, also Appendix \ref{C}). The deformation theory for $M(1,\overrightarrow{d},\overrightarrow{\nu})$ at the point  $(E,\mathcal{L},\nabla)$ is controlled by the hypercohomology $\mathbb{H}^\bullet(\mathbf{MDR}_{(E,\mathcal{L},\nabla)})$ of middle logarithmic de Rham complex. More precisely, the 0th-, 1st- and 2nd-hypercohomologies describe the automorphism, deformation and obstruction of $(E,\mathcal{L},\nabla)$, respectively (cf. \cite[Theorem 2.9]{si}).

 We choose a subsheaf $\varphi:F\simeq \mathcal{O}_{\mathbb{P}^1}(1)\hookrightarrow E$ and consider the deformation of the quadruple $(E,\mathcal{L},\nabla,\varphi)$, which is controlled by the hypercohomology of the complex
\begin{align*}
  \mathbf{C}: \End(E)\oplus \Hom(F,F)\xrightarrow{\Upsilon}\mathrm{MDR}^1_{(E,\mathcal{L},\nabla)}\oplus \Hom(F,E)
\end{align*}
with the differential\begin{align*}
                       \Upsilon=\left(
                                  \begin{array}{cc}
                                    \nabla & 0\\
                                  \bullet\circ\varphi  & \varphi\circ \bullet\\
                                  \end{array}
                                \right).
                     \end{align*}
                     Hence, there is a long exact sequence
                     \begin{align*}
                      \cdots\rightarrow \mathbb{H}^i(\Hom(F,F)\xrightarrow{\varphi\circ\bullet}\Hom(F,E))\rightarrow\mathbb{H}^i(\mathbf{C})\rightarrow\mathbb{H}^i(\mathbf{MDR}_{(E,\mathcal{L},\nabla)})\rightarrow\cdots.
                     \end{align*}
\begin{lemma}$\mathbb{H}^2(\Hom(F,F)\xrightarrow{\varphi\circ\bullet}\Hom(F,E))\simeq \mathbb{C}$.
\end{lemma}
    \begin{proof}   By the exact sequence
                     \begin{align*}
                       \cdots\rightarrow H^1(C,\Hom(F,F))\xrightarrow{\varphi\circ\bullet}H^1(C,\Hom(F,E))\rightarrow\mathbb{H}^2(\Hom(F,F)\xrightarrow{\varphi\circ\bullet}\Hom(F,E))\rightarrow0,
                     \end{align*}
                  we have  \begin{align*}
                          \mathbb{H}^2(\Hom(F,F)\xrightarrow{\varphi\circ\bullet}\Hom(F,E))\simeq
                   H^1(C,\Hom(F,E))\simeq H^1(\mathbb{P}^1,\mathcal{O}_{\mathbb{P}^1}(-2)).
                        \end{align*}The lemma follows.
                        \end{proof}
                     \begin{lemma} The hypercohomology $\mathbb{H}^1(\mathbf{MDR}_{(E,\mathcal{L},\nabla)})$ has a direct summand isomorphic to
$H^1(\mathbb{P}^1,\mathcal{O}_{\mathbb{P}^1}(-3))$. \end{lemma}

\begin{proof}Consider the   subbundle $(E_1,\mathcal{L}^1)$ of $(E,\mathcal{L})$ with $E_1\simeq \mathcal{O}_{\mathbb{P}^1}(2)$ and $L_i^1=L_i\bigcap F$, and the quotient bundle $(E_0,\mathcal{L}^0)$ with $G=E/F$ and $L_i^0=L_i/L_i^1$. Indeed, by  the automorphism of $E$, one can fix $\mathcal{L}=(0,0,0,0,0,1)$, hence
$L'_i=\{0\}
$ and $(E_0,\mathcal{L}'')\simeq (\mathcal{O}_{\mathbb{P}^1}(-1),\{\mathcal{O}_{\mathbb{P}^1}(-1)|_{z_i}\}_{i=1,\cdots,5})$. Writing $\nabla=\left(
                                                                                            \begin{array}{cc}
                                                                                              \nabla_0 & \theta \\
                                                                                              \xi & \nabla_1 \\
                                                                                            \end{array}
                                                                                          \right)$, the $\mathbb{C}^\times$-limits of $(E,\mathcal{L},\nabla)$ is exactly $(E_0\oplus E_1,\mathcal{L}^0\oplus \mathcal{L}^1, \left(
                                                                                            \begin{array}{cc}
                                                                                              0 & \theta \\
                                                                                              0 & 0 \\
                                                                                            \end{array}
                                                                                          \right) )\simeq \mathbb{E}'$. This subbundle induces a filtration
\begin{align*}
  \mathcal{F}: F^{-1}(\mathrm{End}(E))\supset F^{0}(\mathrm{End}(E))\supset F^{1}(\mathrm{End}(E))
\end{align*}
 on $\mathrm{End}(E)$
such that the corresponding gradings are given by
\begin{align*}
\mathrm{Gr}^\mathcal{F}_{-1}(\mathrm{End}(E))&=\Hom(E_1,E_0),\\
  \mathrm{Gr}^\mathcal{F}_0(\mathrm{End}(E))&=\Hom(E_0,E_0)\oplus\Hom(E_1,E_1),\\
   \mathrm{Gr}^\mathcal{F}_1(\mathrm{End}(E))&=\Hom(E_0,E_1),
\end{align*}
hence we have  a filtration $\mathbf{F}$ on the complex $\mathbf{MDR}_{(E,\mathcal{L},\nabla)}$ as
\begin{align*}
\mathbf{F}^p(\mathbf{MDR}_{(E,\mathcal{L},\nabla)}):F^p(\mathrm{End}(E)) &\xrightarrow{\nabla} F^p(\mathrm{MDR}^1_{(E,\mathcal{L},\nabla)}  ), \ p=-1,0,1,
\end{align*}where
\begin{align*}
F^p(\mathrm{MDR}^1_{(E,\mathcal{L},\nabla)}  )=\mathrm{MDR}^1_{(E,\nabla)}  )\bigcap F^{p-1}(\mathrm{End}(E))\otimes \Omega^1_C(\mathcal{D}),
\end{align*}
and the corresponding gradings
\begin{align*}
 \mathbf{Gr}^{\mathbf{F}}_p(\mathbf{MDR}_{(E,\mathcal{L},\nabla)}): \mathrm{{Gr}}^{\mathcal{F}}_p(\mathrm{End}(E)) &\xrightarrow{\theta_p}  \mathrm{{Gr}}^{\mathcal{F}}_p(\mathrm{MDR}^1_{(E,\mathcal{L},\nabla)}  ), \ p=-1,0,1.
\end{align*}
Therefore, we get the spectral sequence
\begin{align*}
  E_1^{p,i-p}=: \mathbb{H}^i(\mathbf{Gr}^{\mathbf{F}}_p(\mathbf{MDR}_{(E,\mathcal{L},\nabla)}))\Rightarrow\mathbb{H}^i(\mathbf{MDR}_{(E,\mathcal{L},\nabla)}).
\end{align*}
Since $\mathbb{E}'$ is $\overrightarrow{w}$-stable, $\mathbb{H}^0( \bigoplus\limits_{p=-1,0,1}\mathbf{Gr}^{\mathbf{F}}_p(\mathbf{MDR}_{(E,\mathcal{L},\nabla)}))\simeq \mathbb{H}^2(\bigoplus\limits_{p=-1,0,1} \mathbf{Gr}^{\mathbf{F}}_p(\mathbf{MDR}_{(E,\mathcal{L},\nabla)}))\simeq\mathbb{C}$, which implies that the above spectral sequence degenerates at $E_1$ (cf.   \cite[Lemma 7.1 ]{s1}). Hence, there is a decomposition
\begin{align*}
  \mathbb{H}^i(\mathbf{MDR}_{(E,\mathcal{L},\nabla)})\simeq\bigoplus_{p=-1,0,1}\mathbb{H}^i(\mathbf{Gr}^{\mathbf{F}}_p(\mathbf{MDR}_{(E,\mathcal{L},\nabla)})).
\end{align*}
In particular, it follows that $\mathbb{H}^1(\mathbf{MDR}_{(E,\mathcal{L},\nabla)})$ has a direct summand isomorphic to $\mathbb{H}^1(\mathbf{Gr}^{\mathbf{F}}_{-1}(\mathbf{MDR}_{(E,\mathcal{L},\nabla)}))\simeq {H}^1(C,\Hom(E_1,E_0))\simeq H^1(\mathbb{P}^1,\mathcal{O}_{\mathbb{P}^1}(-3))$.
\end{proof}
 As a result, we have surjective morphisms
\begin{align*}
  \mathbb{H}^1(\mathbf{MDR}_{(E,\mathcal{L},\nabla)})\twoheadrightarrow H^1(\mathbb{P}^1,\mathcal{O}_{\mathbb{P}^1}(-3))
  \twoheadrightarrow H^1(\mathbb{P}^1,\mathcal{O}_{\mathbb{P}^1}(-2))\simeq \mathbb{H}^2(\Hom(F,F)\xrightarrow{\varphi\circ\bullet}\Hom(F,E)),
\end{align*}
where the second surjective morphism is due to the short exact sequence
\begin{align*}
 0\rightarrow\mathcal{O}_{\mathbb{P}^1}(-3)\rightarrow\mathcal{O}_{\mathbb{P}^1}(-2)\rightarrow S_z\rightarrow 0
\end{align*}
for a skyscraper sheaf $S_z$ supported at a  single point $z\in \mathbb{P}^1$.
The irreducibility of $(E,\nabla)$ leads to the vanishing of  $\mathbb{H}^2(\mathbf{MDR}_{(E,\mathcal{L},\nabla)})$ (cf. \cite[Corollary 2.10]{s1}), then together with the exact sequence
\begin{align*}
                      \cdots\rightarrow\mathbb{H}^1(\mathbf{MDR}_{(E,\mathcal{L},\nabla)})\rightarrow \mathbb{H}^2(\Hom(F,F)\xrightarrow{\varphi\circ\bullet}\Hom(F,E))\rightarrow\mathbb{H}^2(\mathbf{C})\rightarrow0.
                     \end{align*}
we conclude that $\mathbb{H}^2(\mathbf{C})$ vanishes, namely  the infinitesimal deformation of $(E,\mathcal{L},\nabla,\varphi)$ is unobstructed.

 Note that the nonzero element in $H^1(\mathbb{P}^1,\mathcal{O}_{\mathbb{P}^1}(-3))$ represents a nontrivial extension $0\rightarrow\mathcal{O}_{\mathbb{P}^1}(2)\rightarrow \mathcal{E}\rightarrow \mathcal{O}_{\mathbb{P}^1}(-1)\rightarrow0$ for some coherent sheaf $\mathcal{E}$ on $\mathbb{P}^1$ and the local freeness is an open condition. We can find a family of triples $(E_t,\nabla_t,\varphi_t)$ parameterized by $t\in(0,\varepsilon]$ for a small positive real number $\varepsilon$
such that
\begin{itemize}
  \item $E_t\simeq B$;
  \item $(E_t,\nabla_t)$ has a spectrum $\overrightarrow{\nu}$;
\item $\varphi_t$ is an inclusion  of subbundle;
  \item $\lim\limits_{t\rightarrow0}(E_t,\nabla_t,\varphi_t)=(E,\nabla,\varphi)$,
\end{itemize}
 i.e. we have $(E_t,\nabla_t,\varphi_t)\in \Psi^{-1}_{\overrightarrow{w}}(\mathrm{Fix}^1)$ for all $t\in(0, \varepsilon]$. The theorem follows.
\end{proof}

So far, combining  the above steps,  the parabolic version of Simpson's conjecture is confirmed for our case. In summary, we have the following theorem.

\begin{theorem}Fix a non-special weight system $\overrightarrow{w}$ with $\sum\limits_{i=1}^5w_i<1$, and write $\mathrm{Fix}(1,\overrightarrow{d},\overrightarrow{w})=\coprod\limits_{\alpha}\mathrm{Fix}_\alpha(1,\overrightarrow{d},\overrightarrow{w})$ as the union of the connected components.
\begin{enumerate}
  \item $\Psi_{\overrightarrow{w}}: M(1,\overrightarrow{d},\overrightarrow{\nu})\rightarrow \mathrm{Fix}(1,\overrightarrow{d},\overrightarrow{w})$ is a surjective morphism with 2-dimensional fibers, which fit together into a regular foliation
on $ M(1,\overrightarrow{d},\overrightarrow{\nu})$.
  \item There is a coarse granulation of the index set $A=\{\alpha\}$
as
\begin{align*}
  A=A^0\coprod A^1,
\end{align*}
where
\begin{align*}
  A^0&=\{\alpha\in A: \dim_\mathbb{C}\mathrm{Fix}_\alpha(1,\overrightarrow{d},\overrightarrow{w})=0\}\neq \emptyset,\\
A^1&=\{\alpha\in A: \dim_\mathbb{C}\mathrm{Fix}_\alpha(1,\overrightarrow{d},\overrightarrow{w})=2\}\neq \emptyset,
\end{align*}
such that
\begin{align*}
 \overline{M^1(1,\overrightarrow{d},\overrightarrow{\nu})}\backslash M^1(1,\overrightarrow{d},\overrightarrow{\nu})=M^0(1,\overrightarrow{d},\overrightarrow{\nu}).
\end{align*}
\end{enumerate}

\end{theorem}

\
\

\appendix

\section{Elementary Transformations}\label{aa}

For a parabolic  bundle $(E, \mathcal{L})$ of type $(d,\overrightarrow{d})$ over $(C,D)$, fixing a point  $z_j\in D$, we consider a subbundle  $ E^\natural_j\subset E$
generated by those sections directed by $L_j$, thus $E^\natural_j$ satisfies the following commutative diagram
\begin{align*}
\begin{array}{ccccccccc}
   &  &0 &&0&& & & \\
   &  &\uparrow &&\uparrow&& & &\\
0& \rightarrow &L_j &\rightarrow&E|_{z_j}&\rightarrow&E|_{z_j}/L_j &\rightarrow &0 \\
&  &\uparrow &&\uparrow&&|| & & \\
0& \rightarrow &E^\natural_j &\rightarrow&E&\rightarrow&E|_{z_j}/L_j &\rightarrow &0 \\
&  &\uparrow &&\uparrow&& & &\\
&  &E(-z_j) &=&E(-z_j)&& & &\\
&  &\uparrow &&\uparrow&& & &\\
 &  &0 &&0&& & &
  \end{array}.
\end{align*}
Note that  $E^\natural_j$ coincides with $E$ over $C\backslash\{z_j\}$, and  $\deg(E^\natural_j)=d-1 $.
 The parabolic structure $\mathcal{L}_j^\natural=\{ L^\natural_i\}_{i=1,\cdots,n}$ on $E^\natural_j$ is given by
\begin{align*}
  L^\natural_i=\left\{
              \begin{array}{ll}
                L_i, & \hbox{$i\neq j$;} \\
               (E|_{z_i}/L_i)\otimes \mathcal{O}_{C}(-z_i)|_{z_i}, & \hbox{$i=j$,}
              \end{array}
            \right.
\end{align*}
thus $ L^\natural_j$ satisfies the short exact sequence
\begin{align*}
 0\rightarrow L^\natural_j\rightarrow E^\natural_j|_{z_j}\rightarrow L_j\rightarrow0.
\end{align*}
 The \emph{elementary transformation}\footnote{ In some literatures, it is also called Hecke transformation (e.g. \cite{ag,aaa}), or Hecke modification (e.g. \cite{no,bbb}).} of $(E, \mathcal{L})$ at the point $z_j$ is defined as the parabolic bundle $ \mathrm{Elm}_j (E,\mathcal{L}):=(E^\natural_j, \mathcal{L}_j^\natural)$ of type $(d-1,\overrightarrow{d})$.
 The \emph{twisted  elementary transformation} of $(E, \mathcal{L})$ at the point $z_j$, which is denoted as $\check{ \mathrm{Elm}}_j(E,\mathcal{L})$, is defined as the elementary transformation of $(\check{E}_j, \check{\mathcal{L}}_j)$ at  $z_j$, where $\check{E}_j=E(z_j)$ and the parabolic structure $\check{\mathcal{L}}_j=\{ \check{L}_i\}_{i=1,\cdots,n}$ on $\check{E}_j$ is given by
\begin{align*}
  \check{L}_i=\left\{
              \begin{array}{ll}
                L_i, & \hbox{$i\neq j$;} \\
              L_i\otimes \mathcal{O}_C(z_i)|_{z_i}, & \hbox{$i=j$,}
              \end{array}
            \right.
\end{align*} thus $\check{ \mathrm{Elm}}_j(E,\mathcal{L})={ \mathrm{Elm}}_j(\check{E}_j, \check{\mathcal{L}}_j)$.
Let $(e,f )$ be a basis of $E$ in terms of a local trivialization around $z_j$ such that $e$ generates $L_j$ at $z_j$, then $E^\natural_j$ and $(\check{E}_j)^\natural_j$ are generated by $(e, (z-z_j)f)$ and $(\frac{e}{z-z_j}, f)$ near $z_j$, respectively. One easily checks that
\begin{align*}
  \mathrm{Elm}_j\circ \check{ \mathrm{Elm}}_j (E,\mathcal{L})=\check{ \mathrm{Elm}}_j \circ \mathrm{Elm}_j(E,\mathcal{L})=(E,\mathcal{L}).
\end{align*}

\begin{proposition}
If a parabolic bundle $(E, \mathcal{L})$ of  type $(d,\overrightarrow{d})$ over $(C,D)$ is $\overrightarrow{w}$-stable, then its elementary transformation $(E^\natural_j, \mathcal{L}_j^\natural)$ at the point $z_j$ is $\overrightarrow{ w^\natural}$-stable, where the weight system $\overrightarrow{ w^\natural}$ is given by
\begin{align*}
   w^\natural_i=\left\{
               \begin{array}{ll}
                 w_i, & \hbox{$i\neq j$;} \\
                 1-w_i, & \hbox{$i=j$.}
               \end{array}
             \right.
\end{align*}
\end{proposition}
\begin{proof} Under the  elementary transformation, a subbundle $F$ of $E$ becomes a   subbundle $F^\natural_j$ of $E^\natural_j$, then we have
\begin{align*}
 F^\natural_j|_{z_i}\left\{
                      \begin{array}{ll}
                      = F|_{z_i}, & \hbox{$i\neq j$;} \\
                        =L_i^\natural, & \hbox{$i=j$ and $F|_{z_j}\neq L_j$;} \\
                        \neq L_i^\natural, & \hbox{$i=j$ and $F|_{z_j}=L_j$,}
                      \end{array}
                    \right.
\end{align*}
and
\begin{align*}
  \deg (F^\natural_j)=\left\{
                      \begin{array}{ll}
                        \deg (F)-1, & \hbox{$F|_{z_j}\neq L_j$;} \\
                        \deg (F), & \hbox{$F|_{z_j}=L_j$.}
                      \end{array}
                    \right.
\end{align*}
 It follows from the definition that
\begin{align*} s_{\overrightarrow{w}}(F)&=\left\{
        \begin{array}{ll}
          d-2\deg( F^\natural_j)+\sum\limits_{i\neq j,L_i\neq  F^\natural_j|_{z_i}}w_i-\sum\limits_{i\neq j,L_i= F^\natural_j|_{z_i}}w_i-w_j, & \hbox{$  F|_{z_i}= L_j$;} \\
          d-2(\deg( F^\natural_j)+1)+\sum\limits_{i\neq j,L_i\neq F^\natural_j|_{z_i}}w_i-\sum\limits_{i\neq j,L_i= F^\natural_j|_{z_i}}w_i+w_j, & \hbox{$ F|_{z_i}\neq L_j$,}
        \end{array}
      \right.\\
&=s_{\overrightarrow{ w^\natural}}(F^\natural_j).
\end{align*}
 The conclusion thus follows.
\end{proof}

 Let $(E, \mathcal{L},\nabla)$ be a $\overrightarrow{\nu}$-parabolic  logarithmic flat  bundle of  type $(d,\overrightarrow{d})$ over $(C,D)$, and  $(E^\natural_j,\mathcal{L}_j^\natural)$ be the elementary transformation of the underlying parabolic bundle $(E, \mathcal{L})$ at the point $z_j$. Since $L_j$ is the eigenspace of $\mathrm{Res}_{z_j}(\nabla)$, $\nabla$ induces a logarithmic connection on
$E^\natural_j$, which is also denoted by $\nabla$, then the spectrum $\overrightarrow{ \nu^\natural}$ of $(E^\natural_j,\nabla)$ reads \cite{mk,l}
\begin{align*}
 (\nu^\natural)^{+}_{i}=\left\{
                        \begin{array}{ll}
                         \nu^{+}_{i} , & \hbox{$i\neq j$;} \\
                          1+\nu^{-}_{i}, & \hbox{$i=j$,}
                        \end{array}
                      \right.\   (\nu^\natural)^{-}_{i}=\left\{
                        \begin{array}{ll}
                         \nu^{-}_{i} , & \hbox{$i\neq j$;} \\
                          \nu^{+}_{i}, & \hbox{$i=j$.}
                        \end{array}
                      \right.
\end{align*}
The $\overrightarrow{ \nu^\natural}$-parabolic logarithmic flat bundle $(E^\natural_j, \mathcal{L}_j^\natural,\nabla)$ of type $(d-1,\overrightarrow{d})$ over $(C,D)$ is called the  \emph{elementary transformation} of $(E, \mathcal{L},\nabla)$ at the point $z_j$. In particular, if $\overrightarrow{\nu}$ is Kostov-generic/non-resonant, then so is $\overrightarrow{\nu^\natural}$. The \emph{twisted elementary transformation} of $(E, \mathcal{L},\nabla)$ at the point $z_j$ is defined as the elementary transformation of $(\check{E}_j, \check{\mathcal{L}}_j,\check{\nabla}_j)$ at $z_j$, where the logarithmic  connection $\check{\nabla}_j$ on $\check{E}_j$ is given by $\check{\nabla}_j=\nabla\otimes \mathrm{Id}+\mathrm{Id}\otimes \mathfrak{D}_j$ for the logarithmic connection $\mathfrak{D}_j=\mathrm{d}-\frac{dz}{z-z_j}$ on $\mathcal{O}(z_j)$.

\begin{theorem}\label{po}Assuming $\overrightarrow{\nu}$ is Kostov-generic, we have the   isomorphisms of moduli spaces
\begin{align*}
  \mathrm{Elm}_j: M(d,\overrightarrow{\nu})&\rightarrow M(d-1, \overrightarrow{\nu^\natural})\\
(E, \mathcal{L},\nabla)&\mapsto (E^\natural_j, \mathcal{L}_j^\natural,\nabla),\\
\check{\mathrm{Elm}}_j: M(d,\overrightarrow{\nu})&\rightarrow M(d+1, \overrightarrow{\nu^\natural})\\
(E, \mathcal{L},\nabla)&\mapsto ((\check{E}_j)^\natural_j, (\check{\mathcal{L}}_j)_j^\natural,\check{\nabla}_j).
\end{align*}

\end{theorem}

\section{Geometry of the Moduli Space of Indecomposable Parabolic Bundles }\label{ab}
In this appendix, we describe the birational geometry of $M^{\textrm{ind}}(1,\overrightarrow{d})$ for the case of $d=1, n=|D|=5$. In order to construct moduli spaces by variation of geometric invariant theory \cite{dh}, we should  impose  a common  stability condition on as many  parabolic bundles as possible.
For a  parabolic bundle $(E,\mathcal{L})$ of type $(1,\overrightarrow{d})$  over $(C,D)$, if it is $\overrightarrow{w}$-stable, we have $E\simeq B=\mathcal{O}_{\mathbb{P}^1}\oplus\mathcal{O}_{\mathbb{P}^1}(1)$ or $E\simeq B'=\mathcal{O}_{\mathbb{P}^1}(-1)\oplus\mathcal{O}_{\mathbb{P}^1}(2)$.

Firstly, we  consider  the parabolic structure $\mathcal{L}=\{L_i\}_{i=1,\cdots,5}$ on $B$ parameterized by  $(u_1,\cdots, u_5)\in (\mathbb{P}^1)^5$. Let  $\overrightarrow{w}$ be a  non-special weight system on $(B,\mathcal{L})$, then  a line subbundle $F\subset B$ is not a  $\overrightarrow{w}$-destabilizing bundle if
 \begin{align*}
 2\deg(F)-\sum_{L_i\neq F|_{z_i}}w_i+\sum_{L_i= F|_{z_i}}w_i<1.
\end{align*}
  Since $0< w_i<1$, the above inequality is automatically satisfied when $\deg(F)\leq -2$.
Hence, we should consider the following three cases.

\begin{description}
    \item[Case I $F\simeq\mathcal{O}_{\mathbb{P}^1}(-1)$] The morphism $F\hookrightarrow B$ is given by a pair $(q,r)$ with
                 \begin{align*}
                        q&=q_0+q_1z,\\
                        r&=r_0+r_1z+r_2z^2.
                 \end{align*}
        If $F|_{z_i}=L_i$ at the point $z_i\in D$, we must have
                      \begin{align*}
                             [q_0+q_1z_i:r_0+r_1z_i+r_2z_i^2]=u_i.
                      \end{align*}
        Introduce a matrix
              \begin{align*}
                   \Delta_{\mathcal{L}}=\left(
                      \begin{array}{ccccc}
                         1 & z_1 & z_1^2& u_1& u_1z_1 \\
                         1 & z_2 & z_2^2& u_2& u_2z_2 \\
                         1 & z_3 & z_3^2& u_3& u_3z_3 \\
                         1 & z_4 & z_4^2& u_4& u_4z_4 \\
                         1 & z_5 & z_5^2& u_5& u_5z_5 \\
                       \end{array}
                    \right),
                \end{align*}
        where if some $u_i=\infty$, the corresponding row vector is replaced by $(0,0,0,1,z_i)$. Note that when $n_\infty=3$, any  line subbundle of degree $-1$ is automatically  not a  $\overrightarrow{w}$-destabilizing bundle. The necessary and sufficient conditions of any  line subbundle of degree $-1$ being not a  $\overrightarrow{w}$-destabilizing bundle are listed as follows for different cases
            \begin{align*}
                 \begin{tabular}{|c|c|}
                      \hline
                 Case & Condition \\
                      \hline
                 $n_\infty=5$ & $\sum\limits_{i=1}^5w_i<3$ \\
                      \hline
                 $n_\infty=4$   & $ -w_j+ \sum\limits_{z_i\in D\backslash\{z_j\}}w_i<3$ for $\{z_j\}=D\backslash D_\infty$\\
                      \hline
                 $n_\infty=2$   &$ -w_j+ \sum\limits_{z_i\in D\backslash\{z_j\}}w_i<3$ for any $z_j\in D_\infty$\\
                      \hline
                 $n_\infty\leq 1, \det\Delta_{\mathcal{L}}=0$ & $\sum\limits_{i=1}^5w_i<3$\\
                       \hline
                 $n_\infty\leq 1, \det\Delta_{\mathcal{L}}\neq0$ & $ -w_j+ \sum\limits_{z_i\in D\backslash\{z_j\}}w_i<3$ for any $z_j\in D$\\
                       \hline
                 \end{tabular}.
                 \end{align*}

    \item[Case II $F\simeq\mathcal{O}_{\mathbb{P}^1}(1)$] There is a unique morphism $F\hookrightarrow B$, this subbundle is  not a  $\overrightarrow{w}$-destabilizing bundle if and only if
                \begin{align*}
                     1+\sum_{z_i\in D_\infty}w_i-\sum_{z_i\in\widehat{D_\infty}}w_i<0.
                \end{align*}
Obviously, if $\sum_{i=1}^5w_i<1$, the above inequality is never true, hence a necessary condition is
\begin{align*}
  \sum_{i=1}^5w_i>1.
\end{align*}
                                                       \item[Case III $F\simeq\mathcal{O}_{\mathbb{P}^1}$] The morphism $F\hookrightarrow B$ as a saturated subsheaf is given by $(1,r)$ with
                                                           \begin{align*}
                                                             r=r_0+r_1z,
                                                           \end{align*}
                                                           hence if $F|_{z_i}=L_i$ at the point $z_i\in D$, we have
                                                           \begin{align*}
                                                             r_0+r_1z_i=u_i.
                                                           \end{align*}
                                                           If $n_\infty=4,5$, any  line subbundle of degree $0$ is always  not a  $\overrightarrow{w}$-destabilizing bundle.   Assume $n_\infty\leq 3$, for a subset $D_1=\{z_{i_1},\cdots,z_{i_\ell}\}\subset\widehat{D_\infty}$, one defines
                                                           \begin{align*}
                                                             \mathrm{rk}_{D_1}=\mathrm{rank}\left(
                                                                                     \begin{array}{ccc}
                                                                                       1 & z_{i_1}&u_{i_1} \\
                                                                                       \vdots & \vdots&\vdots \\
                                                                                       1&z_{i_\ell}&u_{i_\ell}\\
                                                                                     \end{array}
                                                                                   \right),
                                                           \end{align*} and let
                                                           \begin{align*}
                                                             m_\infty=\max_{D_1\subset \widehat{D_\infty}, \ \mathrm{rk}_{D_1}=2}\{|D_1|\}
                                                           \end{align*}
                                                           Then any  line subbundle of degree $0$ is  not a  $\overrightarrow{w}$-destabilizing bundle if and only if
                               for any  subsets $D_1\subset \widehat{D_\infty}$ with
                               $|D_1|= m_\infty$,
                                 the following inequality holds
        \begin{align*}
                                                                \sum_{z_i\in D_1}w_i-\sum_{z_i\in {D\backslash  D_1}}w_i<1.
                                                              \end{align*}
                                                    \end{description}

 \begin{proposition}
 There exists some  non-special weight $\overrightarrow{w}$ such that all the indecomposable parabolic structures  on $B$ lying in $U_{i_1\cdots i_\ell}$ are $\overrightarrow{w}$-stable.
\end{proposition}

\begin{remark}
This proposition exhibits  a phenomenon that all indecomposable objects can be stabilized simultaneously. Such phenomenon also appears in other fields, for example, Reineke's conjecture for  representations of  Dynkin quivers \cite{r}. For $A_n$-type quivers, this conjecture has been settled independently by Hu--Huang \cite{hh}, Apruzzese--Igusa \cite{ai} and Kinser \cite{ki}, for general Dynkin quivers, this problem has been recently studied by Diza, Gilbert and Kinser \cite{dk}, and  the modified Reineke's conjecture \cite{hh} has been confirmed by Chang--Qiu--Zhang \cite{ch}.
\end{remark}

\begin{proof}
We only need to check the cases of $\ell\geq 3$. There are three cases.
 \begin{description}
   \item [Case I $\ell=5$] Firstly, note that when considering Case III as above we must have $m_\infty\leq 4$. Otherwise,  we have
\begin{align*}
  u_i-u_k=(u_i-u_j)\frac{z_i-z_k}{z_i-z_j}
\end{align*}
for  any three points $z_i,z_j,z_k\in D$, which means that
 \begin{align*}
   \det \Pi_{D}=\det \left(
                  \begin{array}{ccc}
                    z_1 & 1 & u_1 \\
                    \vdots & \vdots & \vdots \\
                    z_5 & 1 & u_5 \\
                  \end{array}
                \right)=0,
 \end{align*}
 hence there is an automorphism of $E$ to make  the parabolic structure $\mathcal{L}$  into an indecomposable one with all $u_i$'s being zero. Therefore, in order to find the desired weight system, it suffices to solve the following inequalities
\begin{align*}
  \sum_{i=1}^5w_i&<3,\\
  \sum_{i=1}^5w_i&>1,\\
  -w_j+ \sum_{z_i\in D\backslash\{z_j\}}w_i&<1,\ j=1,\cdots,5.
\end{align*}
 These inequalities always admit solutions, for example, pick $w_1=\cdots=w_5=w$ with
\begin{align*}
  \frac{1}{5}<w<\frac{1}{3}.
\end{align*}
   \item [Case II $\ell=4$]Assume $u_i=\infty$. We similarly solve the  inequalities
\begin{align*}
  \sum_{j=1}^5w_j&<3,\\
 1+w_i-\sum_{z_j\in D\backslash\{z_i\}}w_j&<0,\\
   \sum_{z_j\in D_1}w_j-\sum_{z_j\in D\backslash D_1}w_j&<1,
\end{align*}
where $D_1$ is any subset of $D\backslash\{z_i\}$ with $|D_1|=3$. These inequalities also have  solutions $w_1=\cdots=w_5=w$ with
\begin{align*}
  \frac{1}{3}<w<\frac{3}{5}.
\end{align*}
   \item[Case III $\ell=3$] Assume $u_i=u_j=\infty$. The following inequalities
\begin{align*}
  -w_k+ \sum_{z_l\in D\backslash\{z_k\}}w_l&<3,\ k=i,j,\\
  1+w_i+w_j-\sum_{z_l\in D\backslash\{z_i,z_j\}}w_l&<0,\\
   \sum_{z_l\in D_1}w_l-\sum_{z_l\in D\backslash D_1}w_l&<1,
\end{align*}
where $D_1$ is any subset of $D\backslash\{z_i,z_j\}$ with $|D_1|=2$, have solutions
\begin{align*}
  \left\{
  \begin{array}{ll}
    w_i=w,\ w_j=1-w,  \\
    w_k= w,\  k\neq i, j
  \end{array}
\right.
\end{align*}
with\begin{align*}
      \frac{2}{3}<w<\frac{4}{5}.
    \end{align*}
 \end{description}
Therefore, we complete the proof.
\end{proof}

 \begin{theorem}\label{zzz}
      There is a  coarse moduli space $ M^{\mathrm{ind}}(0,1,\overrightarrow{d})$ associate to the moduli stack $N^{\mathrm{ind}}(0,1,\overrightarrow{d})$, which is given by
      \begin{align*}
      M^{\mathrm{ind}}(0,1,\overrightarrow{d})=P_0\biguplus P_1\biguplus P_2\biguplus P_3\biguplus P_4\biguplus P_5,
  \end{align*}
  where $P_0$ is isomorphic to a Del Pezzo surface of degree 4, $P_1$ is isomorphic to $\mathbb{P}^2$, $P_2, P_3, P_4, P_5$ are all isomorphic to  $ \mathbb{P}^2$ minus one point,
  and $\biguplus$ stands for  patching two charts via the morphisms from Del Pezzo surface of degree 4 to $ \mathbb{P}^2$  along the maximal open subsets where they are
one-to-one.
           \end{theorem}
        \begin{proof}
  By variation of geometric invariant theory, the charts of  the coarse moduli space $ M^{\mathrm{ind}}(0,1,\overrightarrow{d})$ can be constructed by imposing stability condition with respect to certain weight,  and there are wall-crossing  phenomena  when changing the weights, hence the corresponding charts are related by a special birational transformation (flip) \cite{bo,dh}. Let $ M^{\overrightarrow{w}}(0,1,\overrightarrow{d})$ be the coarse moduli space of $\overrightarrow{w}$-stable parabolic structures on $B$ for a fixed non-special weight $\overrightarrow{w}$.
  \begin{description}
    \item[Step 1] Choose $\overrightarrow{w_1}=(w,\cdots,w)$ with $\frac{1}{5}<w<\frac{1}{3}$. We have seen that
  \begin{align*}
   M^{\overrightarrow{w_1}}(0,1,\overrightarrow{d})=U'_{12345}/A_B\simeq \mathbb{P}^2.
  \end{align*}
    \item[Step 2] Change the weight into $\overrightarrow{w_2}=(w,\cdots,w)$ with $\frac{1}{3}<w<\frac{3}{5}$.  One easily finds the followings.
    \begin{itemize}
      \item Some parabolic structures lying in $U'_{12345}$ become unstable. These parabolic structures must be of $m_\infty=4$,  then in the moduli space $M^{\overrightarrow{w_1}}(0,1,\overrightarrow{d})$ they are five points that represent the equivalence classes of
          \begin{align*}
            \mathcal{L}=(1,0,0,0,0), (0,1,0,0,0), (0,0,1,0,0), (0,0,0,1,0), (0,0,0,0,1)\in U'_{12345}.
          \end{align*}
         We denote these five points  by $A_1, A_2, A_3, A_4, A_5$, respectively, and it is clear that they are in generic position.
      \item The  parabolic structures lying in $U'_{1\cdots\hat{i}\cdots5}$, $i=1,\cdots,5$,  become stable, and we have shown that \begin{align*}
                                                                                 U'_{1\cdots\hat{i}\cdots5}/A_B\simeq\mathbb{P}^1.
                                                                                                                                    \end{align*}
    \end{itemize}
    Therefore, the moduli space $ M^{\overrightarrow{w_2}}(0,1,\overrightarrow{d})$  is the blowing-up of $ M^{\overrightarrow{w_1}}(0,1,\overrightarrow{d})$ at the  points $A_1,A_2, A_3, A_4,A_5$ with the corresponding  exceptional divisors  denoted by $C_1, C_2, C_3, C_4, C_5$, respectively. Namely,  $M^{\overrightarrow{w_2}}(0,1,\overrightarrow{d})$ is exactly isomorphic to a  Del Pezzo surface of degree 4. This  blowing-up is denoted  by
    $
      \flat_1: M^{\overrightarrow{w_2}}(0,1,\overrightarrow{d})\rightarrow M^{\overrightarrow{w_1}}(0,1,\overrightarrow{d})$.

    \item[Step 3] Consider the weight $\overrightarrow{w_{12}}=(w,1-w,w,w,w)$ with  $ \frac{2}{3}<w<\frac{4}{5}$. Similarly, we have the followings.
    \begin{itemize}
      \item
    The  parabolic structures lying in $U'_{12345}\coprod (\coprod_{i=1}^5U'_{1\cdots\hat{i}\cdots5})$  that are not $\overrightarrow{w_{12}}$-stable are described  as follows.
  \begin{itemize}
  \item The equivalence classes corresponding to $A_1,A_2,A_3,A_4,A_5$ mentioned previously, and the equivalence classes of
        \begin{align*}
         \mathcal{L}= (u_1,u_2,0,0,0), (0,u_2,u_3,0,0), (0,u_2,0, u_4,0), (0,u_2,0,0,u_5)\in U'_{12345}
        \end{align*}
         with $u_i\neq 0,i=1,\cdots,5$. We denote the projective line  in the moduli space  $M^{\overrightarrow{w_1}}(0,1,\overrightarrow{d})$ through the points $A_i, A_j$ by $C_{ij}$, then  these parabolic structures  in the moduli space  $M^{\overrightarrow{w_1}}(0,1,\overrightarrow{d})$ are parameterized by  $C_{12}, C_{23}, C_{24}, C_{25}$.
    \item The equivalence classes of $\mathcal{L}=(u_1,\cdots,u_5)\in U'_{12345}$ with $\det\Delta_{\mathcal{L}}=0$. They can take the  form $\mathcal{L}=(u_1,u_2,u_3,0,0)$ with $\mathbb{C}^\times$-action,  then the constraint $\det\Delta_{\mathcal{L}}=0$ is reduced to
        \begin{align*}
          \frac{(z_2-z_4)(z_2-z_5)}{(z_2-z_1)(z_2-z_3)}u_1u_3+\frac{(z_1-z_4)(z_1-z_5)}{(z_1-z_2)(z_1-z_3)}u_2u_3+\frac{(z_3-z_4)(z_3-z_5)}{(z_3-z_1)(z_3-z_2)}u_1u_2=0.
        \end{align*}
       One easily checks that
       \begin{align*}
        (u_1,u_2,u_3)=&\
        (1,0,0),(0,1,0),(0,0,1),\\
   &\ (z_1-z_4,z_2-z_4,z_3-z_4), (z_1-z_5,z_2-z_5,z_3-z_5)
                        \end{align*}
   are solutions to the above equation. Hence, these parabolic structures in the moduli space  $M^{\overrightarrow{w_1}}(0,1,\overrightarrow{d})$ are parameterized by the unique  conic through five points  $A_1, A_2, A_3, A_4, A_5$, denoted by $\mathfrak{C}$, which is isomorphic to $\mathbb{P}^1$.
 \item The equivalence classes of
 \begin{align*}
  K_1&= (\infty,1,0, 0,0),  K_2=(\infty,1,\frac{(z_2-z_1)(z_3-z_4)(z_3-z_5)}{(z_3-z_1)(z_2-z_4)(z_2-z_5)}, 0,0)\in C_{1},; \\
   K_3&= (1,\infty,0, 0,0), K_4= (0,\infty,1, 0,0),  K_5= (0,\infty,0,1,0),  K_6=(0,\infty,0,0,1),\\
K_7&= (1,\infty,\frac{(z_1-z_2)(z_3-z_4)(z_3-z_5)}{(z_3-z_2)(z_1-z_4)(z_1-z_5)}, 0,0) \in C_{2};  \\
                K_8&=  (0,1,\infty, 0,0), K_9=(\frac{(z_2-z_3)(z_1-z_4)(z_1-z_5)}{(z_1-z_3)(z_2-z_4)(z_2-z_5)},1,\infty, 0,0)\in C_{3};\\
                   K_{10}&= (0,1,0,\infty,0), K_{11}=(\frac{(z_1-z_3)(z_2-z_4)(z_1-z_5)}{(z_2-z_3)(z_1-z_4)(z_2-z_5)},1,0,\infty,0)\in C_{4};\\
                   K_{12}&= (0,1,0,0, \infty), K_{13}=(\frac{(z_1-z_3)(z_1-z_4)(z_2-z_5)}{(z_2-z_3)(z_2-z_4)(z_1-z_5)},1,0,0,\infty)\in C_{5}.
                  \end{align*}
  \end{itemize}

      \item The parabolic structures lying in $U'_{345}, U'_{145}, U'_{135}, U'_{134}$ are all $\overrightarrow{w}_{12}$-stable, which are the  equivalence classes of
      \begin{align*}
        \mathcal{L}=(\infty,\infty,1,0,0), (1,\infty,\infty,0,0), (1,\infty,0,\infty,0), (1,\infty,0,0,\infty),
      \end{align*}
      respectively. Note that they are infinitesimally closed to
      $C_1, C_3, C_4, C_5$.
      respectively.
    \end{itemize}

With respect to the blowing-up $\flat_1$, the strict transforms of the line $C_{ij}\in M^{\overrightarrow{w_1}}(0,1,\overrightarrow{b})$ and the conic $\mathfrak{C}\in M^{\overrightarrow{w_1}}(0,1,\overrightarrow{b})$ are   denoted by   $\widetilde{C}_{ij}\in M^{\overrightarrow{w_2}}(0,1,\overrightarrow{b})$ and $\widetilde{\mathfrak{C}}\in M^{\overrightarrow{w_2}}(0,1,\overrightarrow{b})$, respectively. Contracting five  $(-1)$-curves $\widetilde{C}_{12}, \widetilde{C}_{23}, \widetilde{C}_{24}, \widetilde{C}_{25}, \widetilde{\mathfrak{C}}$ gives rise to a morphism $\flat_{12}:M^{\overrightarrow{w_2}}(0,1,\overrightarrow{b})\rightarrow \check{M}^{\overrightarrow{w_{12}}}(0,1,\overrightarrow{b})$, where the projective variety $\check{M}^{\overrightarrow{w_{12}}}(0,1,\overrightarrow{b})$ is
  isomorphic to $\mathbb{P}^2$. The images of the following points  under $\flat_{12}$ are coincides
  \begin{align*}
   K_1 \textrm{ and } K_3; \\
   K_4 \textrm{ and } K_8;\\
   K_5 \textrm{ and } K_{10};\\
   K_6\textrm{ and } K_{12};\\
   K_2, K_7, K_9 \textrm{ and } K_{13},
  \end{align*}
  and the corresponding points in $\check{M}^{\overrightarrow{w_{12}}}(0,1,\overrightarrow{b})$ are denoted by $Q_1, Q_2, Q_3, Q_4, Q_5$, respectively.
   Let $\check{C}_1, \check{C}_2, \check{C}_3,\check{C}_4, \check{C}_5$ be the images  of $C_1, C_2, C_3,C_4, C_5$ under $\flat_{12}$, then $\check{C}_1,  \check{C}_3,\check{C}_4, \check{C}_5$ intersect $\check{C}_2$ at $Q_1$ and $Q_5$,  $Q_2$ and $Q_5$,  $Q_3$ and $Q_5$,  $Q_4$ and $Q_5$, respectively.
   Therefore, the moduli space $M^{\overrightarrow{w}_{12}}(0,1,\overrightarrow{b})$ is  $M^{\overrightarrow{w}_{12}}(0,1,\overrightarrow{b})\backslash
   \{Q_5\}$
Similarly, for the weights $\overrightarrow{w_{13}}=(w,w, 1-w,w,w,), \overrightarrow{w_{14}}=(w,w, w,1-w,w,), \overrightarrow{w_{15}}=(w,w, w,w, 1-w)$ with  $ \frac{2}{3}<w<\frac{4}{5}$,  we have the moduli spaces $M^{\overrightarrow{w_{13}}}(0,1,\overrightarrow{b})$, $M^{\overrightarrow{w_{14}}}(0,1,\overrightarrow{b}), M^{\overrightarrow{w_{15}}}(0,1,\overrightarrow{b})$,  which are all isomorphic to  $ \mathbb{P}^2$ minus one point, and the corresponding morphisms  $\flat_{13}:M^{\overrightarrow{w_2}}(0,1,\overrightarrow{b})\rightarrow \check{M}^{\overrightarrow{w_{13}}}(0,1,\overrightarrow{b})$, $\flat_{14}:M^{\overrightarrow{w_2}}(0,1,\overrightarrow{b})\rightarrow \check{M}^{\overrightarrow{w_{14}}}(0,1,\overrightarrow{b})$, $\flat_{15}: M^{\overrightarrow{w_2}}(0,1,\overrightarrow{b})\rightarrow \check{M}^{\overrightarrow{w_{15}}}(0,1,\overrightarrow{b})$.
  \end{description}
 Combining the above steps, we can construct the moduli space as
  \begin{align*}
   M^{\mathrm{ind}}_B= M^{\overrightarrow{w_1}}(0,1,\overrightarrow{b})\biguplus M^{\overrightarrow{w_2}}(0,1,\overrightarrow{b})\biguplus M^{\overrightarrow{w_{12}}}(0,1,\overrightarrow{b})\biguplus M^{\overrightarrow{w_{13}}}(0,1,\overrightarrow{b})\biguplus
   M^{\overrightarrow{w_{14}}}(0,1,\overrightarrow{b})\biguplus
   M^{\overrightarrow{w_{15}}}(0,1,\overrightarrow{b}),
  \end{align*}
  where $\biguplus$ means that we patch  these charts via   $\flat_{1}, \flat_{12}, \flat_{13}, \flat_{14}, \flat_{15}$ along the
maximal open subsets where they are one-to-one.
  \end{proof}

Similar arguments also work for the parabolic structure $\mathcal{L}=\{L_i\}_{i=1,\cdots,5}$ on $B'$ parameterized by  $(u_1,\cdots, u_5)\in (\mathbb{P}^1)^5$. Let $F$ be a line subbundle of $B'$.
\begin{description}
                                                      \item[Case I $F\simeq\mathcal{O}_{\mathbb{P}^1}(-1)$]
                                                      The morphism $F\hookrightarrow B'$ is given by a pair $(1,r)$ with
                                                  \begin{align*}
                                                  r=r_0+r_1z+r_2z^2+r_3z^3,
                                                  \end{align*}
                                                      hence if $F|_{z_i}=L_i$ at the point $z_i\in D$, we have
                                                           \begin{align*}
                                                             r_0+r_1z_i+r_2z_i^2+r_3z_i^3=u_i.
                                                           \end{align*}
              If $n_\infty\geq 2$, any line subbundle of degree $-1$ is always not a $\overrightarrow{w}$-destabilizing bundle. The necessary and sufficient conditions of any line subbundle of degree $-1$ being not a  $\overrightarrow{w}$-destabilizing bundle are listed as follows for different cases
                                                \begin{align*}
                                                  \begin{tabular}{|c|c|}
                                                      \hline
                                                       Case & Condition \\
                                                       \hline
                                                      $n_\infty=1$ & $-w_j+ \sum\limits_{z_i\in D\backslash\{z_j\}}w_i<3$ for $\{z_j\}=D_\infty$ \\
                                                  \hline
                                                  $n_\infty=0, \det \Pi_{D}=0$ & $\sum\limits_{i=1}^5w_i<3$ \\
                                                 \hline
                                                  $n_\infty=0, \det \Pi_{D}\neq0$ & $ -w_j+ \sum\limits_{z_i\in D\backslash\{z_j\}}w_i<3$ for any $z_j\in D$\\
                                                  \hline
                                                    \end{tabular},
                                                \end{align*}
                                                where
                                                \begin{align*}
                                                  \Pi_{D}=\left(
                                                                                     \begin{array}{ccccc}
                                                                                       z_1^3 & z_{1}^2&z_1 &1&u_1\\
                                                                                      z_2^3& z_{2}^2&z_2 &1&u_2\\
                                                                                      z_3^3& z_{3}^2&z_3 &1&u_3\\
                                                                                       z_4^3 & z_{4}^2&z_4 &1&u_4\\
                                                                                       z_5^3&z_{5}^2& z_5 &1&u_5\\
                                                                                     \end{array}
                                                                                   \right).
                                                \end{align*}
     \item[Case II $F\simeq\mathcal{O}_{\mathbb{P}^1}(2)$] There is a unique morphism $F\hookrightarrow B$, this subbundle is not a $\overrightarrow{w}$-destabilizing bundle  if and only if
          \begin{align*}
                                                          3+\sum_{z_i\in D_\infty}w_i-\sum_{z_i\in\widehat{D_\infty}}w_i<0.
                                                          \end{align*}
Obviously, if $\sum_{i=1}^5w_i<3$, the above inequality is never true, hence a necessary condition is
\begin{align*}
  \sum_{i=1}^5w_i>3.
\end{align*}
                                                     \end{description}

\begin{theorem}\label{p}There is a coarse moduli space $ M^{\mathrm{ind}}(1,\overrightarrow{d})$ associated to the  moduli stack $ N^{\mathrm{ind}}(1,\overrightarrow{d})$, which is
       given by
      \begin{align*}
      M^{\mathrm{ind}}(1,\overrightarrow{d})=P_0\biguplus P_1\biguplus P_2\biguplus P_3\biguplus P_4,
  \end{align*}
  where $P_0$ is isomorphic to a Del Pezzo surface of degree 4, $P_1, \cdots, P_4$ are all isomorphic to  $ \mathbb{P}^2$,
  and $\biguplus$ stands for  patching two charts via the morphisms  from Del Pezzo surface of degree 4 to $ \mathbb{P}^2$  along the maximal open subsets where they are
one-to-one.
           \end{theorem}
\begin{proof}Note that the parabolic structure $\mathcal{L}=(1,0,0,0,0)$ on $B'$ is $\overrightarrow{w}$-stable if $\overrightarrow{w}=(w,1-w,w,w,w), (w,w,1-w,w,w), (w,w,w,1-w,w), (w,w,w,w,1-w)$ with $\frac{2}{3}<w<\frac{4}{5}$, and it is infinitesimally closed to $\widetilde{\mathfrak{C}}$. Then the theorem follows from Theorem \ref{zzz}.
\end{proof}

\section{Middle Convolution}\label{C}

Take two copies of $C$, which are denoted by $Y, Z$, and view $D$ or $\mathcal{D}$ as the subsets of  $Y$ and $Z$ or  divisors on  $Y$ and $Z$, and denote them by $D_Y, D_Z$ or $\mathcal{D}_Y, \mathcal{D}_Z$, respectively.
The diagonal configuration is defined as a pair $(Y\times Z,V)$,
where
\begin{align*}
  V=(Y\times D_Z)\bigcup(Z\times D_X)\bigcup\Delta_{Y\times Z}
\end{align*}
for the  diagonal $\Delta_{Y\times Z}$ in $Y\times Z$.
After blowing up the  diagonal configuration at the  $n$ triple points $(z_1,z_1),\cdots,(z_n,z_n)$, we get a variety $T$ with a birational map
$f:T\rightarrow Y\times Z$.
Let $\mathcal{J}$ be the corresponding reduced  divisor on $T$ of $J=f^{-1}(V)$, then $\mathcal{J}$ is decomposed as
\begin{align*}
 \mathcal{J}=K+\sum_{i=1}^nH_i+\sum_{i=1}^nG_i+\sum_{i=1}^nU_i,
\end{align*}
where $K$  is the strict transformation of $\Delta_{Y\times Z}$,  $H_i$ is the strict transformation of $\{z_i\}\times Z$, $G_i$ is the strict transformation of $Y\times\{z_i\}$, and $U_i$ is  the exceptional divisor corresponding to $(z_i,z_i)$. Note that each $U_i$ transversally intersects with $L, H_i, G_i$ at three distinct points, and $H_i$ transversally intersects with $G_j$ whenever $i\neq j$. Define the compositions  $\xi: T\xrightarrow{f} Y\times Z\xrightarrow{p_1} Y, \eta: T\xrightarrow{f}Y\times Z\xrightarrow{p_2} Z$, where $p_1, p_2$ are  the natural projections on the first and the second factors, respectively.

Choosing  $\beta\in H^0(T, \Omega^1_{T}(\log\mathcal{J}))$,  let $\beta^{K}, \beta^{H_i}, \beta^{V_i}, \beta^{U_i}$ be the residues along $K$,  $H_i$, $V_i$, $U_i$, respectively, then they are subject to the following conditions
\begin{align*}
  \beta^{K}+\sum_{i=1}^n\beta^{H_i}&=0,\\
  \beta^{K}+\sum_{i=1}^n\beta^{V_i}&=0,\\
  \beta^{U_i}-\beta^{K}+\beta^{H_i}+\beta^{V_i}&=0,\ \  i=1,\cdots, n.
\end{align*}
Now given  a  logarithmic flat bundle $(E,\nabla)$ of rank 2 and degree $d$ over $(Y,D_Y)$ with the spectrum $\overrightarrow{\nu}$ and  $\beta\in H^0(T, \Omega^1_{T}(\log \mathcal{J}))$, we define a  logarithmic flat bundle $(F,\nabla_F)$ on $(T, J)$ as
\begin{align*}
 (F,\nabla_F)=\xi^*(E,\nabla)\otimes (\mathcal{O}_T, \mathrm{d}+\beta),
\end{align*}
and define the defect
\begin{align*}
  \delta(\overrightarrow{\nu},\beta)=2(n-2)-\sum_{i=1}^n\delta(\beta^{H_i}),
\end{align*}
where
\begin{align*}
  \delta(\beta^{H_i})=\left\{
                        \begin{array}{ll}
                          1, & \hbox{$\beta^{H_i}=-\nu_i^+$ \textrm{or }$-\nu_i^-$;} \\
                          0, & \hbox{\textrm{otherwise}.}
                        \end{array}
                      \right.
\end{align*}

Consider the logarithmic de Rham complex
        \begin{align*}
        \mathbf{DR}_{T}(F,\nabla_F)&: F\xlongrightarrow{\nabla_F}F\otimes \Omega^1_{T}(\log\mathcal{J})\xlongrightarrow{\nabla_F}F\otimes \Omega^2_{T}(\log\mathcal{J}),
        \end{align*}
 and the relative logarithmic de Rham complex $\mathbf{DR}_{T/Z}(F,\nabla_F)$ is defined by the short exact sequence of complexes
\begin{align*}
  0\rightarrow \mathbf{DR}_{T}(F,\nabla_F)[-1]\otimes \eta^*\Omega^1_Z(\mathcal{D}_Z)\rightarrow\mathbf{DR}_{T}(F,\nabla_F)\rightarrow\mathbf{DR}_{T/Z}(F,\nabla_F)\rightarrow0.
\end{align*}
 Then the connecting morphism of the above short exact sequence defines the logarithmic  Gauss--Manin connection
\begin{align*}
 \nabla^{\textrm{GM}}: \mathbb{R}^1\eta_*\mathbf{DR}_{T/Z}(F,\nabla_F)\rightarrow\mathbb{R}^1\eta_*\mathbf{DR}_{T/Z}(F,\nabla_F)\otimes \Omega^1_Z(\mathcal{D}_Z).
\end{align*}

Let $J_0=K\bigcup\limits_{i=1}^nH_i$, and define the sheaf $F^0_{J_0/Z}$ by  $F^0_{J_0/Z}|_{\eta^{-1}(z)}=\bigcup\limits_{\mathfrak{z}\in \eta^{-1}(z)\bigcap J_0}F^0_\mathfrak{z}$ for any $z\in Z$, where $F^0_\mathfrak{z}$ is the 0-eigenspace of  $\mathrm{Res}_{\mathfrak{z}}(\nabla_F)$. The middle relative logarithmic de Rham complex $\mathbf{MDR}_{(T/Z,J_0)}(F,\nabla_F)$ associated to $(F,\nabla_F)$ is defined by the following short exact sequence of complexes
\begin{align*}
 0\rightarrow \mathbf{MDR}_{(T/Z,J_0)}(F,\nabla_F)\rightarrow\mathbf{DR}_{T/Z}(F,\nabla_F)\rightarrow F^0_{J_0/Z}[-1]\rightarrow0,
\end{align*}
which yields an injection $\mathbb{R}^1\eta_*\mathbf{MDR}_{(T/Z,J_0)}(F,\nabla_F)\hookrightarrow \mathbb{R}^1\eta_*\mathbf{DR}_{T/Z}(F,\nabla_F)$.
The \emph{middle convolution} of $(E,\nabla)$ associated  to $\beta$ is defined as the pair
 \begin{align*}
   \mathrm{MC}_\beta(E,\nabla)=(\mathbb{R}^1\eta_*\mathbf{MDR}_{(T/Z,J_0)}(F,\nabla_F),\nabla^{\textrm{GM}}),
 \end{align*}
 where the restriction of $\nabla^{\textrm{GM}}$ on $\mathbb{R}^1\eta_*\mathbf{MDR}_{(T/Z,J_0)}(F,\nabla_F)$ is denoted by the same notation.

\begin{proposition}Fixing a non-special $\overrightarrow{\nu}\in R(d)$, let $(E,\nabla)$ be a logarithmic flat bundle of rank 2 and degree $d$ over  $(C,D)$ with the spectrum $\overrightarrow{\nu}$, and choose $\beta\in H^0(T, \Omega^1_{T}(\log\mathcal{J}))$ such that
\begin{align*}
   \beta^{H_i}=-\nu_i^+\textrm{ or }-\nu_i^-, i=1,\cdots,n.
   \end{align*}
   Then the middle convolution $\mathrm{MC}_\beta(E,\nabla)$ is a logarithmic flat bundle of degree $d$ and rank $n-2$ over  $(C,D)$ with the semisimple residues and the  spectrum
             \begin{align*}
  \{\overrightarrow{\mu^{(i)}}\}_{i=1,\cdots,n}=(\underbrace{\beta^{V_i},\cdots,\beta^{V_i}}_{n-3},d+ \beta^{V_i}-\nu^{H_i}-\sum_{j\neq i}\nu_j^{\breve{H_j}})\}_{i=1,\cdots,n},
\end{align*}
  where
\begin{align*}
 \nu_i^{\sigma}=\left\{
                  \begin{array}{ll}
                    \nu_i^{H_i}, & \hbox{$\beta^{H_i}=-\nu_i^{\sigma}$;} \\
                    \nu_i^{\breve{H_i}}, & \hbox{$\rm{otherwise}$,}
                  \end{array}
                \right.
\end{align*}
for $\sigma\in\{+,-\}$. Moreover, denote by $M^{\textrm{ss}}(d,n-2,\{\overrightarrow{\mu^{(i)}}\}_{i=1,\cdots,n})$  the moduli space of logarithmic flat bundles of degree $d$ and rank $n-2$ over $(C,D)$ with the  semisimple (i.e. diagonalizable) residues and the spectrum $\{\overrightarrow{\nu^{(i)}}\}_{i=1,\cdots,5}$,  where
\begin{align*}
 \overrightarrow{\mu^{(i)}}=(\underbrace{\mu_i^+,\cdots,\mu_i^+}_{n-3},\mu_i^-),\  i=1,\cdots,n
\end{align*}
satisfy the following conditions
 \begin{itemize}
 \item $\sum\limits_{z_i\in D'}(\mu_i^+-\mu_i^-)-(|D'|-1)\sum\limits_{i=1}^n\mu_i^+\notin \mathbb{Z}$ for any subset $D'\subset D$;
   \item $\mu_i^+-\mu_i^--\sum\limits_{i=1}^n\mu_i^+\notin \mathbb{Z}$,
    \end{itemize}
 then the middle convolution provides an isomorphism
\begin{align*}
  \mathrm{MC}_\beta:M(d,\overrightarrow{d},\overrightarrow{\nu})&\rightarrow M^{\textrm{ss}}_{\mathrm{dR}}(d,n-2,\{\overrightarrow{\mu^{(i)}}\}_{i=1,\cdots,n})\\
(E,\nabla)&\mapsto \mathrm{MC}_\beta(E,\nabla).
\end{align*}
\end{proposition}
\begin{proof}Firstly, one  easily checks the following conditions
\begin{align*}
 \beta^K\notin \mathbb{Z},\  \nu_i^\sigma+\beta^{H_i}+\beta^K\notin \mathbb{Z},\ \nu_i^\sigma+\beta^{H_i}\notin \mathbb{Z}\backslash\{0\}
\end{align*}
for any $i=1,\cdots,n$ due to the spectrum of $(E,\nabla)$ being non-special.
Then one can apply \cite [Theorem 5.9] {si}, which yields the following properties
of  the logarithmic flat bundle $ \mathrm{MC}_\beta(E,\nabla)$.
\begin{itemize}
   \item The rank $r_\beta$ of $\mathbb{R}^1\eta_*\mathbf{MDR}_{(T/Z,J_0)}(F,\nabla_F)$ is exactly the dimension of $H^1(Y,j_*\mathbb{L}_q)$ (\cite[Lemma 2.7]{si}),
where $j:Y\backslash (D_Y\bigcup q)\hookrightarrow Y$ is the natural embedding for a point $q\in Y$ not lying $D_Y$, $\mathbb{L}_q$ is the local system on $Y\backslash (D_Y\bigcup q)$ associated to the logarithmic  flat bundle $(F,\nabla_F)|_{\eta^{-1}(q)}$. It is calculated as (\cite[Corollary 2.8]{si})
\begin{align*}
  r_\beta&=2+\delta(\overrightarrow{\nu},\beta)\\
  &=2+2(n-2)-n=n-2.
\end{align*}

   \item The residue $\mathrm{Res}_{z_i}(\nabla^{\textrm{GM}})$ has two distinct  eigenvalues  given by Katz transformations (\cite[Scholium 4.9 and Scholium 5.8]{si}). More precisely,  one is exactly
$\mu_i^+=\beta^{V_i}$
with the  multiplicity
\begin{align*}
  m(\beta^{V_i})&=\delta(\beta^{H_i})+\delta(\overrightarrow{\nu},\beta)\\
  &=1+2(n-2)-n=n-3,
\end{align*}
 and the other one is
\begin{align*}
\mu_i^- =\beta^{U_i}+ \nu_i^{\breve{H_i}}&=\beta^{V_i}+\beta^K+\beta^{H_i}+\nu_i^{\breve{H_i}}\\
 &=\beta^{V_i}+\nu_i^{\breve{H_i}}+\sum_{j\neq i}\nu_j^{H_j}\\
&=d+\beta^{V_i}-\nu^{H_i}-\sum_{j\neq i}\nu_j^{\breve{H_j}}.\end{align*}
It is clear  that  $\overrightarrow{\nu}$ being non-special is equivalent to  $\{\overrightarrow{\mu^{(i)}}\}_{i=1,\cdots,n}$
satisfying the conditions in Proposition.

    \item The degree $d_\beta$ of $\mathbb{R}^1\eta_*\mathbf{MDR}_{(T/Z,J_0)}(F,\nabla_F)$ reads by Fuchs relation as
\begin{align*}
  d_\beta&=-\sum_{i=1}^n((n-2)\beta^{V_i}+\nu_i^{\breve{H_i}}+\sum_{j\neq i}\nu_j^{H_j})\\
&=-(n-2)\beta^{V_i}-(\sum_{i=1}^n\nu_i^{\breve{H_i}}+(n-1)\sum_{i=1}^n\nu_i^{H_i})\\
&=d-(n-2)\sum_{i=1}^n(\beta^{V_i}-\beta^{H_i})\\
&=d.
\end{align*}
 \end{itemize}
 The inverse of $\mathrm{MC}_\beta$ is given by $\mathrm{MC}_{-c^*\beta}$, where $c:T\rightarrow T$ is the automorphism which flips the factors (\cite[Lemma 4.14 and Theorem 5.9]{si}).
\end{proof}

\section{Nonabelian Hodge Correspondence}\label{F}

In this appendix, $ C $ is a general compact Riemann surface.

\begin{definition} 
Let $(\tilde E,\tilde \theta)$ be a Higgs bundle over $\tilde C:=C\backslash D$, and let $h$ be a Hermitian metric on $\tilde E$.
\begin{enumerate}
\item The metric $h$ is called \emph{harmonic} if
\begin{align*}
 R(h)+[\tilde \theta,\tilde \theta^{*_h}]=0,
\end{align*}
where  $R(h)$ stands for the curvature of the metric  connection.
\item The metric $h$ is called \emph{acceptable} if near each point $z_i\in D$ \begin{align*}
 |R(h)|_h\leq f^{(i)}+\frac{\delta}{|t^{(i)}|^2(\log|t^{(i)}|)^2}
\end{align*}
for some constant $\delta$ and $f^{(i)}\in L^p(U^{(i)}\backslash\{z_i\})$ with $p>1$, where $t^{(i)}$ is the coordinate over a  small enough neighborhood $U^{(i)}$ of $z_i$ in $C$ such that $t^{(i)}=0$ corresponds to $z_i$.
\item The Higgs bundle $(\tilde E,\tilde \theta)$ is called \emph{tame} if writing $\theta=\Theta^{(i)} dt^{(i)}$ over $U^{(i)}\backslash\{z_i\}$ with $\Theta^{(i)}\in \End(\tilde E|_{U^{(i)}\backslash\{z_i\}})$, any  eigenvalue  $e(\Theta^{(i)})$ of $\Theta^{(i)}$ satisfies
    \begin{align*}
      |e(\Theta^{(i)})|\leq\frac{\delta}{|t^{(i)}|}
    \end{align*}
     for some constant $\delta$.
     \item The triple $(\tilde E,\tilde \theta,h)$ is called a \emph{tame harmonic bundle} over $\tilde C$ if $h$ is a harmonic metric and $(\tilde E,\tilde \theta)$ is a tame Higgs bundle.
\end{enumerate}
\end{definition}

For a Higgs bundle $(\tilde E,\tilde \theta)$ over $\tilde C$ with an acceptable metric $h$, by Cornalba--Griffiths theory \cite{cg}, one can define a coherent sheaf $\tilde E_{\overrightarrow{\alpha}}$ over $C$ for a vector $\overrightarrow{\alpha}=(\alpha^{(1)},\cdots, \alpha^{(n)})$ with $\alpha^{(i)}\in \mathbb{R},i=1,\cdots, n$, called an $h$-metric extension of $\tilde E$, as follows \cite{s,t}. Let $t^{(i)}:\mathcal{O}_C\rightarrow\mathcal{O}_C(z_i)$ be the canonical section, and choose a Hermitian metric $h_i$ on $\mathcal{O}_C(z_i)$, then for any open subset $U\subset C$ one puts
 \begin{align*}
  \Gamma(U,\tilde E_{\overrightarrow{\alpha}})=\{s\in \Gamma(U\bigcap \tilde C,\tilde E): |s|_h=O(\prod_{i=1}^n|t^{(i)}|^{\alpha^{(i)}-\varepsilon}_{h_i})\  \forall \varepsilon>0\}.
 \end{align*}
By definition, we have  the following properties \cite{s}
\begin{itemize}
  \item $j_*\tilde E=\bigcup_{\overrightarrow{\alpha}}\tilde E_{\overrightarrow{\alpha}}$, where $j:\tilde C\hookrightarrow C$ is the inclusion;
  \item  $\tilde E_{\overrightarrow{\alpha}}\subseteq \tilde E_{\overrightarrow{\beta}}$ whenever $\overrightarrow{\alpha}\geq \overrightarrow{\beta}$, i.e. $\alpha^{(i)}\geq \beta^{(i)}$, $i=1,\cdots, n$;
  \item for each $\overrightarrow{\alpha}$ there exists $\overrightarrow{\varepsilon'}>\overrightarrow{0}$\footnote{ For a fixed number $c$, the notation $\overrightarrow{c}$ stands for a vector with $n$ components, each of which is $c$. } such that $\tilde E_{\overrightarrow{\alpha}-\overrightarrow{\varepsilon}}=\tilde E_{\overrightarrow{\alpha}}$ for all $\overrightarrow{0}\leq \overrightarrow{\varepsilon}\leq \overrightarrow{\varepsilon'}$;
  \item $\tilde E_{\overrightarrow{\alpha}+\overrightarrow{1}}=\tilde E_{\overrightarrow{\alpha}}(-\mathcal{D})$ for all $\overrightarrow{\alpha}$.
\end{itemize}
Moreover, one can show the following.

\begin{proposition}[\cite{s,t}]\label{x}
Let $(\tilde E,\tilde \theta,h)$ be a tame harmonic bundle over $X$.
\begin{enumerate}
  \item The metric $h$ is acceptable.
  \item For any $\overrightarrow{\alpha}$, the sheaf $\tilde E_{\overrightarrow{\alpha}}$ over $C$ is locally free.
\end{enumerate}
\end{proposition}

Then  a tame harmonic bundle $(\tilde E,\tilde \theta,h)$ over $\tilde C$ induces a parabolic bundle $(\tilde E_{\overrightarrow{0}}, \mathcal{L}_{\overrightarrow{0}})$ over $(C,D)$, where the parabolic structure $\mathcal{L}_{\overrightarrow{0}}=\{{\mathcal{F}}^{(i)}\}_{i=1,\cdots,n}$ on $\tilde E_{\overrightarrow{0}}$ is defined as follows.
Let
\begin{align*}
  \mathbb{F}_{\overrightarrow{\alpha}}=\tilde E_{\overrightarrow{\alpha}}/\tilde E_{\overrightarrow{1}}
\end{align*}
for $\overrightarrow{0}\leq\overrightarrow{\alpha}<\overrightarrow{1}$, and let
\begin{align*}
  \mathbb{F}^{(i)}_{\beta}=\bigcup_{\overrightarrow{0}\leq\overrightarrow{\alpha}<\overrightarrow{1}, \alpha^{(i)}\geq \beta}\mathbb{F}_{\overrightarrow{\alpha}}
\end{align*}
for $0\leq \beta<1$. One says $\beta$ is a jumping value if  $\mathbb{F}^{(i)}_\beta\neq \mathbb{F}^{(i)}_{\beta+\varepsilon}$ for any $\varepsilon>0$.
Then the filtration ${\mathcal{F}}^{(i)}$ is defined by $F^{(i)}_{\jmath}=
                 \mathbb{F}^{(i)}_{\beta^{(i)}_\jmath}|_{z_i}$ for  $\jmath=0,\cdots,\ell^{(i)}$,
where $0\leq\beta^{(i)}_0<\cdots<\beta^{(i)}_{\ell^{(i)}}<1$ are jumping values. $(\tilde E_{\overrightarrow{0}}, \mathcal{L}_{\overrightarrow{0}})$ is called the associated parabolic bundle of $(\tilde E,\tilde \theta,h)$, and it can be equipped with a weight system $\{\overrightarrow{\beta^{(i)}}\}_{i=1,\cdots,n}$ with
$\overrightarrow{\beta^{(i)}}=(\beta^{(i)}_0,\cdots,\beta^{(i)}_{\ell^{(i)}})$, which is called the induced weight system.

The following proposition is well-known (for example, cf. \cite[Lemma 6.2]{s}, \cite[Proposition 2.5]{e},  \cite[Lemma 2.13]{papa}, \cite[Proposition 3.15]{HKSZ}).
\begin{proposition}\label{s}
With respect to the induced weight system $\{\overrightarrow{\beta^{(i)}}\}_{i=1,\cdots,n}$, the parabolic degree of $(\tilde E_{\overrightarrow{0}}, \mathcal{L}_{\overrightarrow{0}})$ is zero.
\end{proposition}
\begin{proof}Obviously, $j^*\tilde E_{\overrightarrow{0}}=\tilde E$. For any $v>0$, let $C_v=\{z\in C: d(z,D)\geq e^{-v}\}$, $B_v=C\backslash C_v$, where $d(\bullet,\bullet)$ denotes the distance function on $C$, and  let $h_v$ be a smooth metric on $\tilde E_{\overrightarrow{0}}$ which coincides with the harmonic metric $h$ in a neighborhood of $C_v\subset C$. By definition,  we have
\begin{align*}
  \deg (\tilde E_{\overrightarrow{0}})&=\frac{\sqrt{-1}}{2\pi}\int_{C_v}\Tr(R(h_v))+\frac{\sqrt{-1}}{2\pi}\int_{B_v}\Tr(R(h_v))\\
  &=\frac{\sqrt{-1}}{2\pi}\int_{C_v}\Tr(R(h))+\frac{\sqrt{-1}}{2\pi}\int_{B_v}\Tr(R(h_v))\\
  &=\frac{\sqrt{-1}}{2\pi}\int_{B_v}\Tr(R(h_v)).
\end{align*}
When $v$ is sufficiently large, we assume $B_v=\bigcup\limits_{i=1}^n\mathcal{B}_{\delta_i}$ with $\mathcal{B}_{\delta_i}$ being a  disk  with center $z_i$ and small radius $\delta_i$.
Over $\mathcal{B}_{\delta_i}$, $\tilde E_{\overrightarrow{0}}$ is a direct sum of line bundles $L^{\beta_\jmath^{(i)}}$, where the norm of unit section of $L^{\beta_\jmath^{(i)}}$ is given approximately by $|1|_{h_v}\sim |t^{(i)}|^{\beta_\jmath^{(i)}}$. Then we have
\begin{align*}
  \deg (L^{\beta_\jmath^{(i)}})&=\frac{\sqrt{-1}}{\pi}\int_{\partial \mathcal{B}_{\delta_i}}|t^{(i)}|^{-\beta_\jmath^{(i)}}\partial|t^{(i)}|^{\beta_\jmath^{(i)}}=-\beta_\jmath^{(i)},
\end{align*}
which immediately yields that
\begin{align*}
  \mathrm{ pdeg}(\tilde E_{\overrightarrow{0}},\mathcal{L}_{\overrightarrow{0}})=\deg (\tilde E_{\overrightarrow{0}})+\sum\limits_{i=1}^n\sum\limits_{\jmath=0}^{\ell^{(i)}}d^{(i)}_\jmath\beta_\jmath^{(i)}=0
\end{align*}
for the dimension system $\overrightarrow{d^{(i)}}=(d^{(i)}_{0},\cdots,d^{(i)}_{\ell^{(i)}})$ of the associated parabolic structure $\mathcal{L}_{\overrightarrow{0}}$.
\end{proof}

Simpson proved the Kobayashi--Hitchin correspondence for stable weakly parabolic logarithmic Higgs bundles.

 \begin{theorem}[{\cite{s,t}}]\label{mb}
Let $(E,\mathcal{L},\theta)$ be a weakly parabolic logarithmic Higgs bundle of rank $r$ over $(C,D)$. If it is stable and of zero parabolic degree with respect to the given weight system $\{\overrightarrow{w^{(i)}}\}_{i=1,\cdots,n}$, then  the tame Higgs bundle $( E|_{\tilde C},\theta|_{\tilde C})$ over $\tilde C$ admits a unique harmonic metric $h$ up to multiplication by a  positive constant  such that the associated parabolic bundle
  is isomorphic to $(E,\mathcal{L})$
  and the induced weight system coincides with $\{\overrightarrow{w^{(i)}}\}_{i=1,\cdots,n}$.
\end{theorem}

  Let $(E,\mathcal{L},\theta)$ be a weakly parabolic logarithmic Higgs bundle over $(C,D)$ as that in the above theorem, hence it produces a flat bundle $( E', \nabla)$ over $\tilde C$ by the harmonic metric $h$ on $( E|_{\tilde C},\theta|_{\tilde C})$, where  $E'$ is the underlying $C^\infty$ vector bundle of $E|_{\tilde C}$ with the holomorphic structure $d''=\bar\partial_E+\theta^{*_h}$ and $\nabla=\partial_h+\theta$ is the flat connection. By $h$-metric extensions, we have a parabolic bundle $(E'_{\overrightarrow{(0)}}, \mathcal{L}'_{\overrightarrow{0}})$  over $(C,D)$ associated to $(E', \nabla)$. The connection $\nabla$ can be also extended to a logarithmic connection $\nabla_{\overrightarrow{0}}:E'_{\overrightarrow{0}}\rightarrow E'_{\overrightarrow{0}}\otimes \Omega^1_C(\mathcal{D})$ on $ E'_{\overrightarrow{0}}$, which is weakly compatible with $\mathcal{L}'_{\overrightarrow{0}}$. Consequently, $( E', \nabla)$  is extended to a weakly  parabolic logarithmic flat bundle  $(E'_{\overrightarrow{0}},  \mathcal{L}'_{\overrightarrow{0}}, \nabla_{\overrightarrow{0}})$ over $(C,D)$. The nonabelian Hodge correspondence is given by
  \begin{align*}
 \textsf{NH}: (E,\mathcal{L},\theta)\mapsto(E'_{\overrightarrow{0}}, \mathcal{L}'_{\overrightarrow{0}}, \nabla_{\overrightarrow{0}}).
\end{align*}

Finally, we compare the metric extension described as above and the Deligne extension \cite{d,z}. In general, let $(\mathcal{E},\mathfrak{D})$ be a flat bundle over $X$, and let $T^{(i)}=T(\gamma^{(i)})\in \mathrm{GL}(V^{(i)}_*)$ be the monodromy of $( \mathcal{E}, \mathfrak{D})$ along a loop $\gamma^{(i)}$ once around $z_i$, where $V^{(i)}_*=\mathcal{E}|_{z^{(i)}_*}$ for some point $z^{(i)}_*\in U^{(i)}\backslash\{z_i\}$. Assume each monodromy is nonexpanding (i.e. all eigenvalues  have absolute value 1 \cite{e}).
For any  $0\leq \alpha^{(i)}<1$, one introduces
\begin{align*}
 Q_{\alpha^{(i)}}=\{v\in V^{(i)}_*:(T^{(i)}-\zeta_{\alpha^{(i)}})^rv=0\},
\end{align*}
where $\zeta_{\alpha^{(i)}}=e^{2\pi\sqrt{-1}{\alpha^{(i)}}}$, then there are finitely many $\alpha^{(i)}\in[0,1)$ such that  $ W_{\alpha^{(i)}}$ are non-zero, and we denote them
by $\alpha^{(i)}_{\jmath},\jmath=0,\cdots,\ell^{(i)}$, with $0\leq \alpha^{(i)}_{0}<\cdots<\alpha^{(i)}_{\ell^{(i)}}<1$. Define the unipotent operator
\begin{align*}
 T_{\alpha^{(i)}_\jmath}=\zeta_{\alpha^{(i)}_\jmath}^{-1} T^{(i)}|_{Q_{\alpha^{(i)}}},
\end{align*}
and  the nilpotent operator
\begin{align*}
 N_{\alpha^{(i)}_\jmath}&=\frac{1}{2\pi\sqrt{-1}}\log T_{\alpha^{(i)}_\jmath}\\
 &=\frac{1}{2\pi\sqrt{-1}}\sum_{k=1}^{m_{\alpha^{(i)}_\jmath}}(-1)^{k+1}\frac{(T_{\alpha^{(i)}_\jmath}-\mathrm{Id})^k}{k}
\end{align*}
with $n_{\alpha^{(i)}_\jmath}=\left\{
                                \begin{array}{ll}
                                  m_{\alpha^{(i)}_\jmath}, & \hbox{$N_{\alpha^{(i)}_\jmath}\neq 0$} \\
                                  0, & \hbox{$N_{\alpha^{(i)}_\jmath}=0$}
                                \end{array}
                              \right.
$ being the nilpotency index of $N_{\alpha^{(i)}_\jmath}$. There is a unique  weight filtration
\begin{align*}
  \mathcal{W}^{(\alpha^{(i)}_\jmath)}: 0\subset W^{(\alpha^{(i)}_\jmath)}_0\subseteq {W}^{(\alpha^{(i)}_\jmath)}_1\subseteq\cdots\subseteq {W}^{(\alpha^{(i)}_\jmath)}_{2n^{(\alpha^{(i)}_\jmath)}}=Q_{\alpha^{(i)}_\jmath}
\end{align*}
 on $Q_{\alpha^{(i)}_\jmath}$ determined by the following properties
    \begin{itemize}
          \item $ N_{\alpha^{(i)}_\jmath}(W^{(\alpha^{(i)}_\jmath)}_a)\subseteq W^{(\alpha^{(i)}_\jmath)}_{a-2}$, $a=0,1,\cdots,2n^{(\alpha^{(i)}_\jmath)} $,
          \item $ N^{a-n^{(\alpha^{(i)}_\jmath)}}_{\alpha^{(i)}_\jmath}: \mathrm{Gr}^{\mathcal{W}^{(\alpha^{(i)}_\jmath)}}_a\simeq \mathrm{ Gr}^{\mathcal{W}^{(\alpha^{(i)}_\jmath)}}_{2n^{(\alpha^{(i)}_\jmath)}-a}$, $a=n^{(\alpha^{(i)}_\jmath)}+1, \cdots,2m^{(\alpha^{(i)}_\jmath)} $.
     \end{itemize}
By removing the duplicate copies, it gives rise to a filtration
\begin{align*}
 \widetilde{ \mathcal{W}}^{(\alpha^{(i)}_\jmath)}: 0\subset\widetilde{W}^{(\alpha^{(i)}_\jmath)}_0\subset \widetilde{W}^{(\alpha^{(i)}_\jmath)}_1\subset\cdots\subset \widetilde{W}^{(\alpha^{(i)}_\jmath)}_{n^{(\alpha^{(i)}_\jmath)}-1}=Q_{\alpha^{(i)}_\jmath}.
\end{align*}
Let $v^{(i)}_1,\cdots, v^{(i)}_r$ be a basis of $V^{(i)}_*$ adapted to the decomposition
\begin{align*}
  V^{(i)}_*=\bigoplus_{\jmath=0}^{\ell^{(i)}}Q_{\alpha^{(i)}_\jmath},
\end{align*}
 which can be viewed as multi-valued horizontal sections of $( \mathcal{E},\mathfrak{D})|_{U^{(i)}\backslash\{z_i\}}$, then
for $v^{{(i)}}_s\in Q_{\alpha^{(i)}_\jmath}$, $s\in \{1,\cdots,r\}$, we define the Deligne extension
\begin{align*}
  \tilde v^{{(i)}}_s(t^{(i)})&=\exp(\log t^{(i)}\alpha^{(i)}_\jmath+\log t^{(i)}N_{\alpha^{(i)}_\jmath})v^{{(i)}}_s\\
  &=(t^{(i)})^{\alpha^{(i)}_\jmath}\sum_{k=0}^{m_{\alpha^{(i)}_\jmath}}\frac{1}{k!}(\log t^{(i)}N_{\alpha^{(i)}_\jmath})^kv^{{(i)}}_s
\end{align*}
as a global section of $(\mathcal{E},\mathfrak{D})|_{U^{(i)}\backslash\{z_i\}}$. We extend $\mathcal{E}$ to a vector bundle  $\overline{\mathcal{E}}$ over
 $C$ via defining the  space of sections over $U^{(i)}$ to be the $\mathcal{O}_{U^{(i)}}$-module generated by $\{\tilde v^{(i)}_s\}_{s=1,\cdots,r}$, meanwhile the connection $\mathfrak{D}$ can be extended to a logarithmic  connection $\overline{\mathfrak{D}}:\overline{\mathcal{E}}\rightarrow\overline{\mathcal{E}}\otimes \Omega^1_C(\mathcal{D})$ on $\overline{\mathcal{E}}$. Moreover,    we can define two filtrations  $\mathcal{F}^{(i)}$ and $ \mathcal{G}^{(i)}$ on the fiber $\overline{\mathcal{E}}|_{z_i}\simeq V^{(i)}_*$  by
 \begin{align*}
 F^{(i)}_\jmath&=\bigoplus_{\imath=\jmath}^{\ell^{(i)}} Q_{\alpha^{(i)}_\imath}, \ \jmath=0,\cdots,\ell^{(i)},\\
 G^{(i)}_{\sum_{\jmath=0}^{\kappa-1}n_{\alpha^{(i)}_\jmath}+\imath_\kappa}&=(\bigoplus_{a=\imath_\kappa}^{n_{\alpha^{(i)}_\kappa}-1}\mathrm{Gr}^{\widetilde{\mathcal{W}}^{(\alpha^{(i)}_\kappa)}}_{a})
  \bigoplus (\bigoplus_{\jmath=\kappa+1}^{\ell^{(i)}}\bigoplus_{a=0}^{n^{(\alpha^{(i)}_\jmath)}-1} \mathrm{Gr}^{\widetilde{\mathcal{W}}^{(\alpha^{(i)}_\jmath)}}_{a}),\ \kappa=0,\cdots,\ell^{(i)}, \imath_\kappa=0,\cdots, n_{\alpha^{(i)}_\kappa}-1,
 \end{align*}
 respectively. Then $(\overline{\mathcal{E}},\overline{\mathcal{L}}=\{\mathcal{F}^{(i)}\}_{i=1,\cdots,n},\overline{\mathfrak{D}})$ defines a weakly parabolic logarithmic flat bundle, and $(\overline{\mathcal{E}},\overline{\mathcal{L}'}=\{\mathcal{G}^{(i)}\}_{i=1,\cdots,n},\overline{\mathfrak{D}})$ defines a parabolic logarithmic flat bundle. One assigns $\mathrm{Gr}^{\mathcal{F}^{(i)}}_\jmath$ the number $\alpha^{(i)}_\jmath$ to form an induced  weight system $\{\overrightarrow{\alpha}\}_{i=1,\cdots,n}$ on $(\overline{\mathcal{E}},\overline{\mathcal{L}})$. Similarly, one assigns $\mathrm{Gr}^{\mathcal{G}^{(i)}}_{\sum_{\jmath=0}^{\kappa-1}n_{\alpha^{(i)}_\jmath}+\imath_\kappa}$ the number $\alpha^{(i)}_{\kappa-1}+\varepsilon^{(i)}_{\imath_k}$ with $\alpha^{(i)}_{\kappa-1}+\varepsilon^{(i)}_{\imath_k}<\alpha^{(i)}_{\kappa-1}+\varepsilon^{(i)}_{\imath_k+1} $ to form a weight system $\{\overrightarrow{\alpha_\varepsilon}\}_{i=1,\cdots,n}$ on $(\overline{\mathcal{E}},\overline{\mathcal{L}'})$.

\begin{proposition}[{cf. \cite[Section 7]{s}, \cite[Lemma 2.4]{e}, and \cite[Proposition 3.28]{t}}] 
Let $(E,\mathcal{L},\theta)$ be a nilpotent parabolic logarithmic  Higgs bundle  over $(C,D)$ as given in Theorem \ref{mb}, and let $( E', \nabla)$ be the flat bundle over $\tilde C$ obtained by the harmonic metric $h$ on $( E|_{\tilde C},\theta|_{\tilde C})$.
 \begin{enumerate}
   \item $(\overline{E'},\overline{\mathcal{L}},\overline{\nabla})$ described as above is isomorphic to $(E'_{\overrightarrow{(0)}},  \mathcal{L}'_{\overrightarrow{(0)}}, \nabla_{\overrightarrow{(0)}})$ with the same induced weight system.
   \item $v\in \mathcal{W}^{(\alpha^{(i)}_\jmath)}_a$, $a=0,1,\cdots,2m^{(\alpha^{(i)}_\jmath)}$, if and only if
      \begin{align*}
        |\tilde v(t^{(i)})|_h=O(|t^{(i)}|^{\alpha^{(i)}_\jmath}|\log |t^{(i)}||^{\frac{n_{\alpha^{(i)}_\jmath}-a}{2}}).
      \end{align*}
   \item If the weight system $\{\overrightarrow{\alpha_\varepsilon}\}_{i=1,\cdots,n}$ satisfies that
  \begin{itemize}
    \item $\min\{|\varepsilon^{(i)}_{\imath_k}|\}$ is  sufficiently small,
    \item $\sum\limits_{i=1}^n\sum\limits_{\kappa=0}^{\ell^{(i)}} \sum\limits_{\imath_\kappa=0}^{n_{\alpha^{(i)}_\kappa}-1} (\alpha^{(i)}_{\kappa-1}+\varepsilon^{(i)}_{\imath_k}) \dim_\mathbb{C}Gr^{\mathcal{G}^{(i)}}_{\sum_{\jmath=0}^{\kappa-1}n_{\alpha^{(i)}_\jmath}+\imath_\kappa}=\sum\limits_{i=1}^n\sum\limits_{\jmath=0}^{\ell^{(i)}} \alpha^{(i)}_{\jmath}\dim_\mathbb{C}Gr^{\mathcal{F}^{(i)}}_{\jmath}$,
  \end{itemize}
then the parabolic logarithmic flat bundle
$(\overline{E'},\overline{\mathcal{L'}},\overline{\nabla})$ is $\{\overrightarrow{\alpha_\varepsilon}\}_{i=1,\cdots,n}$-stable.
 \end{enumerate}

\end{proposition}


\begin{thebibliography}{9}
\bibitem{aaa}D. Alfaya: 
    Automorphism group of the moduli space of parabolic vector bundles with fixed degree, \emph{Bull. Sci. Math.} \textbf{175} (2022), 103112.

\bibitem{ag}D. Alfaya, T. G\'{o}mez: 
    Automorphism group of the moduli space of parabolic bundles over a curve, \emph{Adv.  Math.}  \textbf{393} (2021), 108070.

\bibitem{ai}P. Apruzzese, K. Igusa: 
    Stability conditions for affine type $A$,  \emph{Algebr. Represent. Theory} \textbf{23} (2020), 2079-2111.

\bibitem{a}D. Arinkin, S. Lysenko: 
    On the moduli of $\mathrm{SL}(2)$-bundles with connections on $\mathbb{P}^1-\{t_1,\cdots,t_4\}$, \emph{Int. Math. Res. Not.} \textbf{19} (1997), 983-999.

\bibitem{bh}U. Bhosle:  
    Parabolic vector bundles on curves, \emph{Ark.  Math.}  \textbf{27} (1989), 15-22.
    
\bibitem{bq}O. Biquard, P. Boalch: 
    Wild non-abelian Hodge theory on curves,\emph{Compo. Math.} \textbf{140} (2004), 179-204.
    
\bibitem{papa}O. Biquard, O. Garc\'{i}a-Prada, I. Riera:  
    Parabolic Higgs bundles and representations of the fundamental group of a punctured surface into a real group, \emph{Adv. Math.} \textbf{372} (2020), 107305.

\bibitem{bis}I. Biswas: 
     Parabolic bundles as orbifold bundles, \emph{Duke Math. J.} \textbf{88} (1997), 305-325.
     
\bibitem{b}I. Biswas: 
    A criterion for the existence of a flat connection on a parabolic vector bundle,  \emph{Adv. Geom.} \textbf{2} (2002), 231-241.
    
\bibitem{bo}H. Boden, Y. Hu: 
    Variations of moduli of parabolic bundles, \emph{Math. Ann.} \textbf{301} (1995), 539-559.
    
\bibitem{by}H. Boden, K. Yokogawa: 
    Moduli spaces of parabolic Higgs bundles and parabolic $K(D)$ pairs over smooth curves: I, \emph{Int. J. Math.} \textbf{7} (1996), 573-598.
    
\bibitem{bbb}D. Boozer: 
    Moduli spaces of Hecke modifications for rational and elliptic curves, \emph{Algebr. Geom. Topol.} \textbf{21} (2021),  543-600.
    
\bibitem{bv}N. Borne, A. Vistoli:  
    Parabolic sheaves on logarithmic schemes, \emph{Adv. Math.} \textbf{231} (2012),  1321-1363.

\bibitem{ch}W. Chang, Y. Qiu, X. Zhang: 
    Geometric model for module categories of Dynkin quivers via hearts of total stability conditions, \emph{J. Algebra} \textbf{638} (2024), 57-89.

\bibitem{Q1}G. Chen, N. Li: 
    Asymptotic geometry of the moduli space of rank two irregular Higgs bundles over the projective line,  \href{https://arxiv.org/abs/2206.11883}{arXiv:2206.11883}.

\bibitem{cg}M. Cornalba, P. Griffiths: 
    Analytic cycles and vector bundles on noncompact varieties, \emph{Invent. Math.} \textbf{28} (1975),  1-106.

\bibitem{w}W. Crawley-Boevey: 
    Indecomposable parabolic bundles and the existence of matrices in prescribed conjugacy class closures with product equal to the identity, \emph{Publ. Math. I.H.E.S.} \textbf{100} (2004), 171-207.
    
\bibitem{d}P. Deligne: 
    Equations diff\'{e}rentielles \`{a} points singuliers r\'{e}guliers, \emph{Lecture Notes in Mathematics} \textbf{163} (1970), Springer-Verlag.

\bibitem{dk}Y. Diaz, C. Gilbert, R. Kinser: 
    Total stability and Auslander--Reiten theory for Dynkin quivers, \href{https://arxiv.org/abs/2208.02445}{arXiv:2208.02445}.
    
\bibitem{dh}I. Dolgachev, Y. Hu:  
    Variation of geometric invariant theory quotients, \emph{Publ. Math. I.H.E.S.} \textbf{87} (1998),  5-51.
    
\bibitem{do}R. Donagi, T. Pantev: 
    Parabolic Hecke eigensheaves, \emph{Ast\'{e}risque} \textbf{435} (2022).

\bibitem{e}A. Eskin, M.Kontsevich, M.  M\"{o}ller, A. Zorich:  
    Lower bounds for Lyapunov exponents of flat bundles on curves, \emph{Geom. Topol.} \textbf{22} (2018), 2299-2338.
    
\bibitem{eg}H. Esnault, M. Groechenig: 
    On Simpson's foliation conjecture, Forthcoming (2024).

\bibitem{fl}T. Fassarella, F. Loray: 
    Moduli of Higgs bundles over the five punctured sphere, \href{https://arxiv.org/abs/2303.12602}{arXiv:2303.12602}.

\bibitem{Q3}J. Fisher, S. Rayan: 
    Hyperpolygons and Hitchin systems, \emph{Int. Math. Res. Not.} \textbf{2016} (2016), 1839-1870.

\bibitem{hh}Z. Hu, P. Huang: 
    Stability and indecomposability of representaions of quivers of $A_n$-type, \emph{Comm. Alg.} \textbf{48} (2020), 2905-2919.
    
\bibitem{HH22a}Z. Hu, P. Huang:  
    Simpson filtration and oper stratum conjecture, \emph{Manus. Math.} \textbf{167} (2022), 653-673.

\bibitem{HH22b}Z. Hu, P. Huang:
    Simpson--Mochizuki correspondence for $\lambda$-flat bundles, \emph{J. Math. Pures Appl.} \textbf{164} (2022), 57-92.
    
\bibitem{hhq}Z. Hu, P. Huang, W. Yan, R. Zong: 
    Moduli space of parabolic representation pairs of fundamental groups of punctured surfaces, in preparation.

\bibitem{HKSZ}P. Huang, G. Kydonakis, H. Sun, L. Zhao: 
    Tame parahoric nonabelian Hodge correspondence on curves, \href{https://arxiv.org/abs/2205.15475}{arXiv:2205.15475}.
    
\bibitem{ki}R. Kinser: 
    Total stability functions for type $A$ quivers, \emph{Algebr. Represent. Theory} \textbf{25} (2022), 835-845.

\bibitem{k}A. Komyo: 
    Mixed Hodge structures of the moduli spaces of parabolic connections, \emph{Nagoya Math. J.} \textbf{225} (2017), 185-206.
    
\bibitem{ko}H. Konno:
    Construction of the moduli space of stable parabolic Higgs bundles on a Riemann surface, \emph{J. Math. Soc. Japan} \textbf{45} (1993), 253-276.

\bibitem{l}F. Loray, M.-H. Saito: 
    Lagrangian fibrations in duality on moduli spaces of rank 2 logarithmic connections over the projective line, \emph{Int. Math. Res. Not.} \textbf{4} (2015), 995-1043.
    
\bibitem{lm}F. Loray, M.-H. Saito, C. Simpson:  
    Foliations on the moduli space of rank two connections on the projectve line minus four points, \emph{S\'{e}minaires et Congr\`{e}s} \textbf{27} (2013), 115-168.
    
\bibitem{lo}M. L\"{u}bke,  C. Okonek: 
    Moduli spaces of simple bundles and Hermitian--Einstein connections, \emph{Math. Ann.} \textbf{276} (1987),  663-674.
    
\bibitem{m}M. Inaba: 
    Moduli of parabolic connections on curves and Riemann--Hilbert correspondence, \emph{J. Alg. Geom.} \textbf{22} (2013), 407-480.
    
 \bibitem{mk}M. Inaba, K. Iwasaki, M.-H. Saito: 
     Moduli of stable parabolic connections, Riemann--Hilbert correspondence and geometry of Painlev\'{e} equation of type VI, Part I, \emph{Publ. Res. Inst. Math. Sci.} \textbf{42} (2006), 987-1089.
     
\bibitem{j}J. Iyer, C. Simpson:  
     The Chern character of a parabolic bundle, and a parabolic corollary of Reznikov's theorem, Geometry and Dynamics of Groups and Spaces, \emph{Progr. Math.} \textbf{265} (2008), 439-485.
     
\bibitem{my}M. Maruyama, K. Yokogawa:  
    Moduli  of parabolic  stable  sheaves, \emph{Math. Ann.} \textbf{293} (1992), 77-99.
    
\bibitem{MS80}B. Mehta, C. Seshadri:  
    Moduli of vector bundles on curves with parabolic structures, \emph{ Math. Ann.} \textbf{248} (1980), 205-239.
    
\bibitem{t}T. Mochizuki: 
    Kobayashi--Hitchin correspondence for tame harmonic bundles and an application,  \emph{Ast\'{e}risque} \textbf{309} (2006).
    
\bibitem{t1}T. Mochizuki: 
    Kobayashi--Hitchin correspondence for tame harmonic bundles II, \emph{Geom. Topol.} \textbf{13} (2009), 359-455.
    
\bibitem{NS}M. Narasimhan, C. Seshadri: 
    Stable and unitary vector bundles on a compact Riemann surface, \emph{Ann. Math.} {\bf 82} (1965), 540-564.
    
\bibitem{no}P. Norbury:   
    Magnetic monopoles on manifolds with boundary, \emph{Trans. Amer. Math. Soc.}\textbf{ 363 }(2011), 1287-1309.

\bibitem{p}J. Poritz:  
    Parabolic vector bundles and Hermitian--Yang--Mills connections over a Riemann surface, \emph{Int. J.  Math.} \textbf{4} (1993), 467-501.
    
\bibitem{Q2}S. Rayan, L. Schaposnik: 
    Moduli spaces of generalized hyperpolygons, \emph{Q. J. Math.} \textbf{72} (2021), 137-161.

\bibitem{r}M. Reineke: 
    The Harder--Narasimhan system in quantum groups and cohomology of quiver moduli, \emph{Invent. Math.} \textbf{152} (2003), 349-368.
    
\bibitem{sh}O. Schiffmann: 
    Indecomposable vector bundles and stable Higgs bundles over smooth projective curves, \emph{Ann. Math.} \textbf{183} (2016), 297-362.
    
    
\bibitem{se}C. Seshadri:  
    Moduli of vector bundles on curves with parabolic structures, \emph{Bull.  Amer.  Math.  Soc.} \textbf{83} (1977), 124-126.

\bibitem{sx}C. Simpson:  
    Constructing variation of Hodge structure using Yang--Mills theory and applications to uniformization, \emph{J. Amer. Math. Soc.} \textbf{1} (1988), 867-918.
    
\bibitem{s}C. Simpson: 
    Harmonic bundles on noncompact curves, \emph{J. Amer. Math. Soc.} \textbf{3} (1990), 713-770.
    
\bibitem{si}C. Simpson: 
    Katz's middle convolution algorithm, \emph{Pure Appl. Math. Q.} \textbf{5} (2009), 781-852.
    
\bibitem{s1}C. Simpson: 
    Iterated destabilizing modifications for vector bundles with connection, \emph{Contemp.  Math.} \textbf{522} (2010), 183-206.

\bibitem{s2}C. Simpson: 
    An explicit view of the Hitchin fibration on the Betti side for $\mathbb{P}^1$ minus 5 points, Geometry and Physics: Volume IIA Festschrift in honor of Nigel Hitchin (2018), Oxford.

\bibitem{sz}S. Szab\'{o}: 
    Hitchin WKB-problem and $P=W$ conjecture in lowest degree for rank 2 over the 5-punctured sphere, \emph{Quart. J. Math.}, \textbf{74} (2023), 687-746.

\bibitem{ta}M. Talpo:
    Parabolic sheaves with real weights as sheaves on the Kato-Nakayama space,  \emph{Adv. Math.} \textbf{336} (2018), 97-148.
    
\bibitem{v}N. Vargas:   
    Geometry of the moduli of parabolic bundles on elliptic curves, \emph{Trans. Amer. Math. Soc.} \textbf{374} (2021), 3025-3052.
    
\bibitem{y}K. Yokogawa: 
    Compactification of moduli of parabolic sheaves and moduli of parabolic Higgs sheaves,  \emph{J. Math. Kyoto Univ.} \textbf{33} (1993), 451-504.
    
\bibitem{y1}K. Yokogawa: 
    Infinitesimal deformation of parabolic Higgs sheaves, \emph{Int. J. Math.} 6 (1995), 125-148.
    
\bibitem{z}S. Zucker:
    Hodge theory with degenerating coefficients: $L_2$ cohomology in the Poincar\'{e} metric, \emph{Ann. Math.} \textbf{109} (1979), 415-476.
    
\end{thebibliography}
\end{document}